\documentclass[a4paper]{amsart}

\usepackage[a4paper,width={16cm},left=2.5cm,bottom=3cm, top=3cm]{geometry}
\usepackage{enumerate}
\numberwithin{equation}{section}
\usepackage[utf8]{inputenc}
\usepackage[english]{babel}

\usepackage{a4wide}
\usepackage{xcolor}
\usepackage{amsmath}
\usepackage{amssymb}
\usepackage{amsfonts}
\usepackage{amsthm,graphicx}

\usepackage{hyperref}
\usepackage{scalerel,stackengine}
\stackMath
\newcommand\reallywidehat[1]{
\savestack{\tmpbox}{\stretchto{\scaleto{\scalerel*[\widthof{\ensuremath{#1}}]{\kern-.6pt\bigwedge\kern-.6pt}{\rule[-\textheight/2]{1ex}{\textheight}}}{\textheight}}{0.5ex}}\stackon[1pt]{#1}{\tmpbox}
}

\newtheorem{teo}{Theorem}[section]
\newtheorem{cor}[teo]{Corollary}
\newtheorem{lem}[teo]{Lemma}
\newtheorem{prop}[teo]{Proposition}
\newtheorem{rem} [teo] {Remark}

\newcommand{\R}{\mathbb{R}}

\usepackage{pgf}
\usepackage{tikz}
\usepackage{pgfplots}
\usepackage{mathrsfs}
\usetikzlibrary{arrows}
\usetikzlibrary{decorations.pathreplacing}

\begin{document}
\title[Geometric assumptions for unique continuation from the edge of a crack]
{A note on geometric assumptions for unique continuation from the edge of a crack}

\author{Alessandra De Luca}

\address{Alessandra De Luca 
\newline \indent Dipartimento di Matematica Giuseppe Peano, Universit\`{a} degli studi di Torino
\newline \indent  Via Carlo Alberto 10, 10123, Torino, Italy}
\email{a.deluca@unito.it}
\date{July 25, 2024}
\thanks{The author is member of the GNAMPA research group of INdAM - Istituto Nazionale di Alta Matematica, partially supported by the INDAM-GNAMPA
    2024 grant ``Nuove frontiere nella capillarit\`{a} non locale'' CUP E53C23001670001.}
    \keywords{Unique continuation; monotonicity formula; local asymptotics; crack.}
 
 \begin{abstract}  
The present paper aims at representing an improvement of the result in \cite{DelFel},  
where a strong unique continuation property and a description of the local behaviour around the edge of a crack for solutions to an elliptic problem are established, by relaxing the star-shapedness condition on the complement of the crack. 
More specifically, this assumption will be dropped off by applying a suitable diffeomorphism which straightens the boundary of the crack, before performing the approximation procedure 
%of the fractured domain 
developed in \cite{DelFel} in order to derive a suitable monotonicity formula. This will yield the appearence of a matrix  in the equation, which shall be handled appropriately: for this we will take a hint from \cite{DelFelVit}.
\end{abstract}

\maketitle
\section{Introduction}
We start our dissertation by summarizing the problem investigated in \cite{DelFel}, if any, including additional details. 
In \cite{DelFel} the authors provide a suitable monotonicity argument to prove local asymptotics and consequently a strong unique continuation principle at the origin
for solutions to the elliptic equation  
\begin{equation*}
-\Delta u=fu \quad \text{in $\Omega\setminus \Gamma$}, 
\end{equation*}
complemented with the partial homogeneous Dirichlet boundary condition
\begin{equation*}
u=0\quad \text{on $\Gamma\cap \Omega$},
\end{equation*}
where $\Gamma$ is a closed set in $\mathbb{R}^N$ with $N\geq 2$ (more explicitly described below in \eqref{gamma}), $\Omega$ is a bounded open domain in $\mathbb{R}^{N+1}$ such that $0\in \Omega\cap \partial \Gamma$ (as depicted in Figure \ref{fig:omega}), and $f\in L^\infty_{\mathrm{loc}}(\Omega\setminus\{0\} )$ satisfies two sets of alternative assumptions (inspired by \cite{Sus1, Sus2}) reformulated as follows. For every $r>0$ we denote by $B_r$ the open ball in $\R^{N+1}$ centered at 0 with radius $r$, namely 
$$B_r:=\{x=(x',x_N,x_{N+1})\in \mathbb{R}^{N-1}\times \mathbb{R} \times \mathbb{R} : |x'|^2+x_N^2+x_{N+1}^2<r^2\}.$$
If $\bar{r}>0$ is such that $B_{\bar{r}}\subset \Omega $, for every $r\in (0,\bar{r})$ and $h\in L^\infty_{\mathrm{loc}}(B_{\bar{r}}\setminus\{0\})$, let
\begin{equation*}
\xi_h(r):=\sup_{x\in \overline{B_r}}|h(x)||x|^2 \qquad\text{and}\qquad \eta_h(r):= \sup _{u\in H^1(B_r)\setminus\{0\}}\frac{\int_{B_r}|h|u^2\,dx}{\int_{B_r}|\nabla u|^2\,dx+\frac{N-1}{2r}\int_{\partial B_r} |u|^2\, ds},
\end{equation*}
where $ds$ denotes the volume element on the sphere $\partial B_r$, and with abuse of notation we still denote by $u$ its trace on $\partial B_r$.
The potential $f$ is also assumed to enjoy the following two sets of alternative assumptions:
\begin{itemize}
\item[(i)] $\displaystyle\lim_{r\to 0^+}\xi_f(r)= 0$, $\displaystyle\frac{\xi_f(r)}{r}\in L^1(0,\bar{r})$, $\displaystyle\frac{1}{r}\int_0^r \frac{\xi_f(t)}{t}\,dt\in L^1(0,\bar{r})$;
\item[(ii)] $\displaystyle\lim_{r\to 0^+}\eta_f(r)= 0$, $\displaystyle\frac{\eta_f(r)}{r}\in L^1(0,\bar{r})$, $\displaystyle\frac{1}{r}\int_0^r \frac{\eta_f(t)}{t}\,dt\in L^1(0,\bar{r})$ and also $\nabla f \in L^\infty_{\mathrm{loc}}(B_{\bar{r}}\setminus \{0\})$, $\displaystyle\frac{\eta_{\nabla f\cdot x}(r)}{r}\in L^1(0,\bar{r})$, $\displaystyle\frac{1}{r}\int_0^r \frac{\eta_{\nabla f\cdot x}(t)}{t}\,dt\in L^1(0,\bar{r})$.
\end{itemize}
\begin{rem}\label{esempi}
It is easy to check that the set of assumptions (i) is satisfied if for example $$|f(x)|= O(|x|^{-2+\delta}) \ \text{as $|x|\to 0$ for some $\delta>0$},$$
and the set of assumptions (ii) if for instance
$$\nabla f \in L^\infty_{\mathrm{loc}}(B_{\bar{r}}\setminus \{0\})\quad\text{and}\quad f,\nabla f\in L^p(B_{\bar{r}})\ \text{for some $p>(N+1)/2$.}$$
Indeed in the latter case, we recall that 
for every $r>0$ and $u\in H^1(B_r)$
\begin{equation}\label{stimadiL2star}
\left(\int_{B_{r}}|u(x)|^{2^\ast} dx\right)^{\frac{2}{2^\ast}} \leq C_{N,p} \left(\int_{B_r} |\nabla u|^2\, dx + \frac{N-1}{2r}\int_{\partial B_r} u^2\, ds\right),
\end{equation}
for some constant $C_{N,p}>0$ depending only on $N$ and $p$, where $2^\ast= 2(N+1)/(N-1)$.
Therefore, being $\omega_{N+1}$ the $N+1$-dimensional Lebesgue measure of $B_1$ and setting
\begin{equation}\label{SNp}
S_{N,p}:= \omega_{N+1}^{\frac{2p-N-1}{pN}},
\end{equation}
thanks to the H\"{o}lder inequality, for all $r\in (0,\bar{r}]$ and $u\in H^1(B_r)$ it holds that
\begin{equation*}\label{eq2}
\begin{split}
\int_{B_r}|f| u^2 \, dx&\leq S_{N,p}\, r ^{\frac{2p-N-1}{p}}\left(\int_{B_{\bar{r}}}|f|^p dx\right)^{\frac{1}{p}}\left(\int_{B_{r}}|u(x)|^{2^\ast} dx\right)^{\frac{2}{2^\ast}}\\
&\leq \,S_{N,p}C_{N,p} \Vert f\Vert_{L^p(B_{\bar{r}})} r ^{\frac{2p-N-1}{p}} \left( \int_{B_r}|\nabla u|^2\, dx +  \frac{N-1}{2r}\int_{\partial B_r}u^2\,ds\right).\\
\end{split}
\end{equation*}
Repeating the same argument with $\nabla f$ in place of $f$, we deduce that the assumptions in (ii) are trivially satisfied.

Furthermore the property $\lim_{r\to 0^+}\eta_f(r)=0$ of the set of assumptions (ii) is verified if for instance $f$ is in the Kato class $K_{N+1}$, namely by definition (see e.g. \cite{FabGarLin}) if  
\begin{equation*}
\tilde{\eta}_f(r) := \sup_{x\in\mathbb{R}^{N+1}}\int_{|x-y|<r}\frac{|f(y)|}{|x-y|^{N-1}}\, dy\to 0^+\quad \text{as $r\to 0^+$}. 
\end{equation*} 
Indeed, in \cite[Lemma 1.1]{FabGarLin} the authors prove that if $V\in K_{N+1}$ with $N\geq 2$, then there exists a dimensional constant $C>0$ such that
\begin{equation*}
\int_{B_r} |V|u^2\,dx \leq C\tilde{\eta}_V(r)\left(\int_{B_r}|\nabla u|^2\,dx + \frac{1}{r}\int_{\partial B_r}u^2\,ds\right)
\end{equation*}
for every $u\in C^\infty(\mathbb{R}^{N+1})$ and for every $r>0$. 
\end{rem}
\begin{figure}[ht]
\begin{tikzpicture}[line cap=round,line join=round,>=triangle 45, scale=2]
%\draw [->] (-1.4,0)-- (1.4,0) node [right] {$x_N$};
%\draw [->] (0,-1.4)-- (0,1.4) node [left] {$t$};
%\draw [black, opacity=0.3] (-1,0) to [out=90, in=180] (0,1) to [out=0, in=120] (0.866,0.5) to [out=238, in=20] (0.43,0.15) to [out=200, in=90]  (0.33,0.076) to [out=270, in=90] (0.35,0.040) to [out=270, in=90] (0.35,0) -- (-1,0);
\draw  (-1,0) to [out=120, in=120] (-0.5,0.7) to [out=30, in=180] (0,1) to [out=0, in=20] (0.43,0)  -- (0,0);
\draw (-1,0) to [out=300, in=180 ] (-0.3,-0.5) to [out=0, in=180 ] (0, -0.5)to [out=310, in=180 ] (0.33,-0.8) to [out=0, in=340 ] (0.43,0) ;
\draw  (0.83,0) -- (0,0);
\draw  (0,0.12) node { $0$};
\draw  (-0.615, 0.3) node {$\Omega$};
\draw  (0.315, 0.12) node {$\Gamma$};
%\draw  (-1,0) to [out=270, in=180] (0,-1)  to [out=0, in=240] (0.866,-0.5) to [out=122, in=340] (0.43,-0.15) to [out=340, in=270]  (0.33,-0.076) to [out=90, in=270] (0.35,-0.040) to [out=90, in=270] (0.35,0) -- (-1,0); 

%\draw[fill] (0.35,0) circle [radius=0.010];
%			\node [below] at (0.33,0) {$n^{-1/8}$};
%			\draw[color=black] (-0.5,-0.12) node {$\sigma_n$};
%\draw[color=black] (0.73,-0.1008) node {$n^{-1/8}$};
%\draw[color=black] (-0.65,0.5) node {$\mathcal{B}_n$};		\draw[color=black] (0.866,0.2) node {$\mathcal{\gamma}_n$};	
\end{tikzpicture}
%   \caption{A representation of the set $\mathcal{B}_n$ when $x'=0$ with $\alpha=1/4$}
\caption{The section of a possible configuration of the domain $\Omega$ with a fracture $\Gamma$}
\label{fig:omega}
\end{figure}
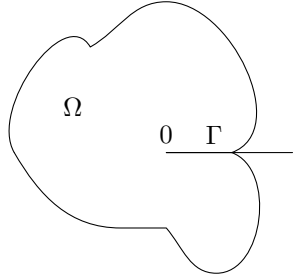

As for assumptions on $\Gamma$, in \cite{DelFel} the authors assume that
\begin{equation}\label{gamma}
\Gamma= \{(x',x_N)\in \mathbb{R}^{N-1}\times\mathbb{R}:\,x_N\geq g(x')\},
\end{equation} 
for some function $g\in C^2(\mathbb{R}^{N-1},\mathbb{R})$, satisfying
\begin{equation}\label{assumtpionssug}
g(0)=0 \qquad \text{and}\qquad \nabla g(0)=0,
\end{equation}
namely, $\partial \Gamma$ is tangent to the hyperplane $\{x_N=0\}$ at 0. 
%It is worth noticing that, expanding $g$ around $x'=0$, from \eqref{assumtpionssug} it follows that
%\begin{equation*}
%g(x')= O(|x'|^2)\quad \text{as $|x'|\to 0$}.
%\end{equation*}
Moreover an additional assumption on $\Gamma$ is given by 
\begin{equation}\label{assumincriminata}
g(x')-\nabla g(x')\cdot x' \geq 0 \quad \text{in $B'_{\bar{r}}:=\{x'\in \mathbb{R}^{N-1}: |x'|<\bar{r}\}$},
\end{equation}
that is, the complement of $\Gamma$ is star-shaped with respect to the origin in a neighbourhood of the origin; in fact, the outward unit normal vector to $\Gamma^c$ at any point $(x',x_N)\in \partial \Gamma^c$ is given by 
\begin{equation*}
\frac{(-\nabla g(x'), 1)}{\sqrt{1+|\nabla g(x')|^2}}.
\end{equation*} 
We underline that all the assumptions are made in a neighbourhood of the origin since the study is of local type. 
In the above setting the authors of \cite{DelFel}, inspired by the procedure developed by Garofalo and Lin in \cite{garofalo}, derive a Pohozaev-type inequality in order to get an estimate from below of the derivative of the Almgren function,   via an approximation argument based on the study of a sequence of boundary value problems on certain approximating domains for which more regularity is available. 

In the present paper we aim at proving the same results in \cite{DelFel} removing assumption \eqref{assumincriminata}, 
on which in \cite{DelFel}
the property of the approximating domains to be star-shaped with respect to the origin heavily relies, since the left-hand side of \eqref{assumincriminata} appears in the computation of the outer unit normal vector to the boundary of the approximating domains (we refer to \cite[Lemma 2.4]{DelFel}). To this purpose, we apply a local diffeomorphism (the same as in \cite{DelFelSic} and \cite{DelFelVit}) straightening the edge of the crack $\Gamma$ around the origin before carrying out our approximation procedure; this will have the effect of dealing with an operator in divergence form in the presence of a non-identity matrix.

We are particularly motivated to remove the star-shapedness condition not only to give a generalization of the results in \cite{DelFel}, but also because in this way the problem turns out to be directly related to the problem investigated in \cite{DelFelVit} when $s = 1/2$ in virtue of a
strong connection (see \cite[Section 1]{DelFel} for more details about this) with the mixed Dirichlet-Neumann boundary 
value problem resulting after applying the Caffarelli-Silvestre extension (see \cite{CafSil}).  
For the same reason, our result can be also considered as a generalization to any dimension grater than 2 of the results in \cite{FalFelFerNiang}, where the authors analyze a mixed bondary value problem in dimension 2 without any assumption of star-shapedness on the portion on which an homogeneous Dirichlet boundary condition is prescribed. 

As for the study of the strong unique continuation property at boundary points, we mention among others \cite{Dip} and \cite{FelFer}, where a strong unique continuation type result at corners under a zero Dirichlet boundary condition and a non-homogeneous Neumann boundary condition respectively is established via a classification of blow-up limit profiles.  
For a more detailed overview on papers in literature about this topic in a local setting, we strongly suggest the reader to consult \cite[Section 1]{DelFel} and references therein. With regars to the study of the strong unique continuation property in a non local framework, it is worth citing \cite{DelFelSic}, \cite{DelFelVit} and \cite{ui}. 

More precisely, we concentrate our investigation on the problem 
\begin{equation}\label{problocaliz}
\left\{\begin{aligned}
-\Delta u&=fu &&\text{in}\ B_{\bar{r}}\setminus\Gamma, \\
u&=0 && \text{on}\ \Gamma\cap B_{\bar{r}},
\end{aligned}\right. 
\end{equation}
with $\Gamma$ as in \eqref{gamma} only satisfying \eqref{assumtpionssug}; moreover, in light of Remark \ref{esempi}, in order to preserve assumptions on the potential $f$ even after applying the diffeomorphism, we reasonably assume either that
\begin{equation}\label{1}
f(x)= O(|x|^{-2+\delta}) \ \text{as $|x|\to 0^+$ for some $\delta>0$},\tag{a1}
\end{equation}  
or that
\begin{equation}\label{2}
f\in W^{1,p}(B_{\bar{r}}) \ \text{for some $p>(N+1)/2$},\tag{a2}
\end{equation}
in place of (i) and (ii) above, considered in \cite{DelFel}. 
In order to state the main result we introduce an eigenvalue problem on the unit $N$-dimensional sphere $\mathbb{S}^N:= \partial B_1$ with a homogeneous Dirichlet boundary condition on the cut $\Theta:= \tilde{\Gamma}\cap \partial B_1$ with 
\begin{equation}\label{gammatilde}
\tilde\Gamma= \{(x',x_N)\in \mathbb{R}^{N-1}\times\mathbb{R}:\,x_N\geq 0\},
\end{equation}
given by 
\begin{equation}\label{problagliautovalori}
\left\{\begin{aligned}
-\Delta_{\mathbb{S}^N} \Psi&=\mu\Psi &&\text{in}\ \mathbb{S}^N\setminus \Theta, \\
\Psi&=0 && \text{on}\ \Theta.
\end{aligned}\right. 
\end{equation}
Letting 
\begin{equation}\label{Htheta}
\mathcal{H}_\Theta:= \{\psi\in H^1(\mathbb{S}^N): \psi=0\ \text{on $\Theta$ in a trace sense}\},
\end{equation}
we say that $\mu\in \mathbb{R}$ is an eigenvalue of \eqref{problagliautovalori} if there exists a non-trivial $\Psi\in \mathcal{H}_\Theta$ such that 
\begin{equation}\label{fordebpsi}
\int_{\mathbb{S}^N}\nabla _{\mathbb{S}^N}\Psi\cdot \nabla _{\mathbb{S}^N}\phi\,ds=\mu \int_{\mathbb{S}^N} \Psi\phi\,ds\quad \text{for all $\phi\in C^\infty_c(\mathbb{S}^N\setminus\Theta)$};
\end{equation}
such a function $\Psi$ is called an eigenfunction of problem \eqref{problagliautovalori} associated with $\mu$. 
Classical spectral theory ensures the existence of an increasing and diverging sequence of real eigenvalues $\{\mu_k\}_{k\geq 1}$ of problem \eqref{problagliautovalori} (counted with their finite multiplicities) which are explicitly given by the sequence
\begin{equation}\label{autovespliciti}
\left\{\frac{k(k+2N+2)}{4}\right\}_{ k\geq 1}
\end{equation}
(see \cite[Appendix A]{DelFel} for the proof). 

The main result of the present paper is the following: we substantially show that the asymptotics performed in \cite[Theorem 1.1]{DelFel} is still valid in the absence of the star-shapedness condition \eqref{assumincriminata}.  
\begin{teo}\label{mainresult}
Let $u\in H^1(B_{\bar{r}})\setminus\{0\}$ be a weak solution to \eqref{problocaliz}. Then there exist $k_0\geq 1$ and an eigenfunction $\Psi$ of problem \eqref{problagliautovalori} associated with the eigenvalue $\mu_{k_0}$ such that 
\begin{equation*}
\frac{u(\lambda x)}{\lambda^{\frac{k_0}{2}}}\to |x|^{\frac{k_0}{2}}\Psi\left(\frac{x}{|x|}\right)\quad\text{in $H^1(B_1)$ as $\lambda\to 0^+$}. 
\end{equation*}
\end{teo}
This result allows us to deduce that the vanishing order of non-trivial solutions to \eqref{problU} cannot exceed the limit of the Almgren function, which is shown to be equal to $k_0/2$ for some $k_0\geq 1$ (see Lemma \ref{ellelegataautovalore} below).
 Therefore, as a relevant consequence, we are able to establish that solutions to \eqref{problocaliz} impose the validity of a strong unique continuation property at $0$, as evidenced in the following corollary.
\begin{cor}
If $u\in H^1(B_{\bar{r}})$ is a weak solution to \eqref{problocaliz} such that $u(x)=O(|x|^k)$ as $|x|\to 0^+$ for every $k\in\mathbb{N}$, then necessarily $u\equiv 0$ in $B_{\bar{r}}$. 
\end{cor}

\subsection*{Organization of the paper}
The rest of the paper is organized as follows. 
In Section \ref{section1} we introduce the straightened problem and we construct a specific approximation argument consisting in the study of some boundary value problems on approximating domains enjoying good properties to obtain a Pohozaev-type inequality in Section \ref{section2}. In Section \ref{section3} we develop Almgren's monotonicity approach and finally in Section \ref{section4} we carry out the so-called blow-up analysis which gives additional information on finite vanishing order of non-trivial solutions at the origin. 

\section{Removal of star-shapeness condition and approximation argument} \label{section1}
\subsection{A straightening of crack's edge}\label{subdiffeo}
We briefly recall the construction and the main properties of the diffeomorphism introduced in \cite[Section 2]{DelFelVit}. 
There exists $F: \R^{N+1}\to \R^{N+1}$ of class $C^{1,1}$ such that $\mathrm{det Jac }F(0)\neq 0$ and $F(0)=0$.  Thus by the Inverse Function Theorem there exists $\tilde{r}>0$ small enough such that 
\begin{itemize}
\item[(i)] $F|_{B_{\tilde{r}}}: B_{\tilde{r}}\to F(B_{\tilde{r}})$ is invertible, namely $F$ is a local diffeomorphism;
\item [(ii)] $F(B_{\tilde{r}})\subset B_{\bar{r}}$.
\end{itemize}
In particular, letting $\tilde{\Gamma}$ as in \eqref{gammatilde}, we have that 
\begin{equation}\label{trasformazione}
\begin{split}
F(B_{\tilde{r}}\setminus\tilde{\Gamma})&= F(B_{\tilde{r}})\setminus\Gamma\subset B_{\bar{r}}\setminus\Gamma,\\
F(\tilde{\Gamma}\cap B_{\tilde{r}})&= \Gamma\cap F(B_{\tilde{r}})\subset \Gamma\cap B_{\bar{r}}.
\end{split} 
\end{equation}
Moreover it holds that
\begin{equation}\label{serveperlultimoteo}
F(x)=x+ O(|x|^2)\quad \text{and}\quad F^{-1}(x)=x+ O(|x|^2) \quad \text{as $|x|\to 0$},
\end{equation}
where $O(|x|^2)$ denotes a vector in $\R^{N+1}$ with all entries equal to $O(|x|^2)$ as $|x|\to 0$, and consequently
\begin{equation}\label{stimesuF}
\begin{split}
\mathrm{Jac}F(x)=\mathrm{Id}_{N+1}+ O(|x|)\quad \text{and}\quad \mathrm{Jac}F^{-1}(x)=\mathrm{Id}_{N+1}+ O(|x|) \quad \text{as $|x|\to 0$},
\end{split}
\end{equation} 
where $\mathrm{Jac}F$ denotes the Jacobian matrix of $F$ and it does not depend on $x_{N+1}$, $\mathrm{Id}_{N+1}$ is the identity matrix in $\mathbb{R}^{(N+1)^2}$, $O(|x|)$ here stands for a matrix in $\mathbb{R}^{(N+1)^2}$ having all elements being $O(|x|)$ as $|x|\to 0$. 
Additionally it holds that
\begin{equation}\label{determinante}
\mathrm{det Jac }F(x)=1+ O(|x'|^2)+O(x_N)\quad\text{as $|x'|\to 0$ and $x_N\to 0$}.
\end{equation}
In light of \eqref{trasformazione}, if $u$ is a weak solution to \eqref{problocaliz}, then $U:= u \circ F $ is a weak solution to 
\begin{equation}\label{problU}
\left\{\begin{aligned}
-\mathrm{div}(A(x)\nabla U(x))&=\tilde
f(x)U(x) &&\text{in}\ B_{\tilde{r}}\setminus\tilde\Gamma, \\
U&=0 && \text{on}\ \tilde\Gamma,
\end{aligned}\right. 
\end{equation}
where 
\begin{equation}\label{Aftilde}
A(x):=|\mathrm{det Jac }F(x)| \mathrm{Jac }F(x)^{-1}(\mathrm{Jac }F(x)^{-1})^T \quad \text{and}\quad \tilde{f}(x):= |\mathrm{det Jac }F(x)| f(F(x)).
\end{equation}
By the $C^{1,1}$-regularity of $F$ we have that 
\begin{equation}\label{Aelipschitz}
A\in C^{0,1}(B_{\tilde{r}},\mathbb{R}^{(N+1)^2});
\end{equation}
furthermore by \eqref{stimesuF}, we have that $\tilde{f}\in L^\infty_{\mathrm{loc}}(B_{\tilde{r}}\setminus\{0\})$, and in addition either
\begin{equation}\label{asssuftilde1}
\tilde{f}(x)=O(|x|^{-2+\delta}) \ \text{as $|x|\to 0^+ $ for some $\delta>0$},
\end{equation}
under assumption \eqref{1}, or
\begin{equation}\label{asssuftilde2}
\tilde f\in W^{1,p}(B_{\tilde{r}}) \ \text{for some $p>(N+1)/2$},
\end{equation}
under assumption \eqref{2}. 
We underline that by weak solution to \eqref{problU} we mean that the admissible functional space for $U$ is 
\begin{equation*}\label{lospaziodiU}
H^1_{\tilde{\Gamma}}(B_{\tilde{r}}):= \overline{C^\infty(\overline{B_{\tilde{r}}}\setminus\tilde{\Gamma})}^{\Vert \cdot\Vert_{H^1(B_{\tilde{r}})}},
\end{equation*}
namely the closure with respect to the $H^1$-norm of the space of all $C^\infty(\overline{B_{\tilde{r}}})$-functions vanishing in a neighbourhood of $\tilde{\Gamma}$, 
and that $U$ satisfies 
\begin{equation}\label{weaksolu}
\int _{B_{\tilde{r}}} A\nabla U\cdot \nabla \varphi\, dx= \int _{B_{\tilde{r}}}  \tilde{f} U\varphi \, dx \quad \text{for all $\varphi\in C^\infty_c(B_{\tilde{r}}\setminus \tilde{\Gamma})$}. 
\end{equation}
For future purposes, we highlight that $A$ can be written as the following block matrix
\begin{equation}\label{Aablocchi}
  A=A(x',x_N)=
  \left(
    \renewcommand{\arraystretch}{1.5}
    \begin{array}{c|c}
      D(x',x_N) & \underline{0} \\
      \hline
      \underline{0} & \mathrm{det Jac }F(x',x_N)
    \end{array}
  \right),
\end{equation}
where $\underline{0}$ is the null vector in $\mathbb{R}^N$ and 
\begin{equation}\label{Dablocchi}
  D(x',x_N)=
  \left(
    \renewcommand{\arraystretch}{1.5}
    \begin{array}{c|c}
      \mathrm{Id}_{N-1}+O(|x'|^2)+O(x_N) & O(x_N) \\
      \hline
      O(x_N) & 1+O(|x'|^2)+O(x_N)
    \end{array}
  \right),
\end{equation}
where $\mathrm{Id}_{N-1}$ denotes the identity matrix in $\mathbb{R}^{(N-1)^2}$, in the top left block $O(|x'|^2)$ and $O(x_N)$ denote two matrices in $\mathbb{R}^{(N-1)^2}$ with all elements being $O(|x'|^2)$ and $O(x_N)$ as $|x'|\to 0$ and $x_N\to 0$ respectively; in the top right block and in the lower left one, $O(x_N)$ stands for a vector in $\mathbb{R}^{(N-1)\times 1}$ and $\mathbb{R}^{1\times (N-1)}$ respectively, with all entries being $O(x_N)$ as $x_N\to 0$. 
Also, from \eqref{stimesuF}, \eqref{determinante} and \eqref{Aftilde}, we can deduce that 
\begin{equation}\label{stimadiA}
A(x)-\mathrm{Id}_{N+1}=O(|x|) \quad \text{as $|x|\to 0$}.
\end{equation}
Thus, without loss of generality, we can assume that  for every $x\in B_{\tilde{r}}$
\begin{equation}\label{Amaggioreidunmezzo}
A(x)y\cdot y\geq \frac{1}{2}|y|^2\quad \text{ for all $y\in \mathbb{R}^{N+1}$},
\end{equation}
and 
\begin{equation}\label{normadiA}
\Vert A(x)\Vert \leq 2, 
\end{equation}
where $\Vert A(x)\Vert$ denotes the norm of $A(x)$ interpreted as the linear  operator $y\in\mathbb{R}^{n+1}\mapsto A(x)y\in \mathbb{R}^{n+1} $. 
%SERVE? Among others, we could choose $\Vert A\Vert= \max _{j=1,\dots,N+1}\sum_{i=1}^{N+1} |a_{ij}|=1+O(|z|)$ 
Now we define 
\begin{equation}\label{mu}
\mu(x):= \begin{cases}
(A x\cdot x)/|x|^2 &\text{if $x\in B_{\tilde{r}}\setminus\{0\}$},\\
1 &\text{if $x=0$},
\end{cases}
\end{equation}
and 
\begin{equation}\label{beta}
\beta(x):=\frac{Ax}{\mu(x)}.
\end{equation}
We observe that from \eqref{stimadiA} and \eqref{mu} we can conclude that 
\begin{equation}\label{stimadimu}
\mu(x)= 1+O(|x|) \quad \text{as $|x|\to 0$},
\end{equation}
and also that 
\begin{equation}\label{nablamu}
\nabla \mu(x) =O(1) \quad \text{as $|x|\to 0^+$}.
\end{equation}
In particular
$\mu $ is continuous on $B_{\tilde{r}}$ and, without loss of generality, we can assume that 
\begin{equation}\label{mumaggiore}
\mu(x) \geq \frac{1}{2} \quad \text{for every $x\in B_{\tilde{r}}$}.
\end{equation}
Furthermore by \eqref{beta}, \eqref{stimadiA} and \eqref{stimadimu}, we have that as $|x|\to 0$
\begin{equation}\label{lestimedibeta}
\begin{cases}
\beta(x)&= x+O(|x|^2)=O(|x|),\\
\mathrm{Jac}\beta(x)&= A(x)+O(|x|)=\mathrm{Id}_{N+1}+O(|x|),\\
\mathrm{div}\beta(x)&=N+1+O(|x|).
\end{cases}
\end{equation}
From this, again without loss of generality, we can deduce that for every $x\in B_{\tilde{r}}$ 
\begin{equation}\label{betaequasix}
|\beta(x)|\leq \mathrm{const}|x|,
\end{equation}
\begin{equation}\label{jaclimitata}
\Vert \mathrm{Jac}\beta(x)\Vert \leq \mathrm{const},
\end{equation}
for some $\mathrm{const}>0$ independent of $x$, and 
\begin{equation}\label{divbetalimitata}
|\mathrm{div}\beta(x)|\leq N+2.
\end{equation}
Also, using the notation $A=(a_{jk})_{j,k=1,\dots, N+1}$, we define for every for every $x\in B_{\tilde{r}}$ and $v_1,v_2\in\mathbb{R}^{N+1}$ 
\begin{equation*}
dA (x)v_1 v_2:= \left(\sum_{j,k=1}^{N+1}\frac{\partial a_{jk}(x)}{\partial x_1}v_j v_k, \dots,\sum_{j,k=1}^{N+1}\frac{\partial a_{jk}(x)}{\partial x_{N}}v_jv_k,0 \right)\in \mathbb{R}^{N+1}.
\end{equation*}   
By direct computations, one can easily check that for every $x\in B_{\tilde{r}}$ and $v_1,v_2,v_2'\in\mathbb{R}^{N+1}$
\begin{equation}\label{simmetriadA}
dA(x)v_1v_2= dA(x)v_2v_1,
\end{equation}
\begin{equation}\label{differenza}
dA(x)v_1v_2-dA(x)v_1v_2'=dA(x)v_1(v_2-v_2'),
\end{equation}
and 
\begin{equation}\label{leproprietadidA}
|dA(x)v_1v_2|\leq \mathrm{const} |v_1||v_2|,
\end{equation}
using \eqref{Aelipschitz}, for some $\mathrm{const}>0$ independent of $x$, $v_1$ and $v_2$.

\subsection{Some crucial inequalities}
We remind the following 
Hardy-type inequality with boundary terms (see \cite[Theorem 1.1]{Wang})
\begin{equation}\label{hardy}
\left(\frac{N-1}{2}\right)^2\int_{B_r}\frac{|U(x)|^2}{|x|^2}\, dx\leq \int_{B_r} |\nabla U|^2\, dx +\frac{N-1}{2r}\int_{\partial B_r} U^2\, ds,
\end{equation}
which will be often invoked throughout the whole paper. 

Now we prove an adapted version of \cite[Lemma 2.1]{DelFel}, due to the presence of the matrix $A$. 

\begin{lem}\label{dis} 
Let $f$ satisfy either \eqref{asssuftilde1} or \eqref{asssuftilde2}. There exists $r_0\in (0,\tilde{r})$ such that every $U\in H^1(B_r)$ with $r\in (0,r_0]$ satisfies the following inequality 
\begin{equation}\label{unificata}
\int_{B_r}A\nabla U\cdot\nabla U \, dx -\int_{B_r} |\tilde{f}| U^2\, dx + C r^{-1+\varepsilon} \int_{\partial B_r} \mu U^2\, ds \geq \frac{1}{4} \int_{B_r} |\nabla U|^2\, dx,
\end{equation}
where $C>0$ is a positive constant depending only on $N$ under assumption \eqref{asssuftilde1} and depending on $N$, $p$ and $\Vert \tilde{f}\Vert_{L^p(B_{\tilde{r}})}$ under assumption \eqref{asssuftilde2}, and 
\begin{equation}\label{esponenteps}
\varepsilon = \begin{cases}
\delta&\text{under assumption \eqref{asssuftilde1}},\\
\frac{2p-N-1}{p}&\text{under assumption \eqref{asssuftilde2}}.
\end{cases}
\end{equation}
Moreover it holds that 
\begin{equation}\label{serviraancora}
Cr^\varepsilon<\frac{N-1}{4}\quad \text{for every $r\in (0,r_0]$}.
\end{equation} 
\end{lem}
\begin{proof}
Let $U\in H^1(B_r)$ with $r\in (0,\tilde{r})$ to be taken gradually smaller throughout the proof according to the needs. 

We start by proving \eqref{unificata} under assumption \eqref{asssuftilde1}. Using \eqref{hardy} and \eqref{mumaggiore}, we can estimate the second term on the left-hand side of \eqref{unificata} as follows  
\begin{equation*}\label{stimasecterm}
\int_{B_r}|\tilde{f}| U^2 \, dx\leq r^\delta\int_{B_r}\frac{|U(x)|^2}{|x|^2}\,dx\leq \left(\frac{2}{N-1}\right)^2r^\delta \int_{B_r}|\nabla U|^2\, dx + \frac{4}{N-1}r^{-1+\delta} \int_{\partial B_r} \mu U^2\,ds. 
\end{equation*}
From this, choosing $r_0\in (0,\tilde{r})$ sufficiently small in such a way that 
\begin{equation}\label{rdelta}
r_0^\delta< \frac{(N-1)^2}{16},
\end{equation}
we have that for every $r\in (0,r_0]$
\begin{equation}\label{stimasecterm2}
\int_{B_r}|\tilde{f}| U^2 \, dx\leq \frac{1}{4} \int_{B_r}|\nabla U|^2\, dx + \frac{4}{N-1}r^{-1+\delta} \int_{\partial B_r} \mu U^2\,ds.
\end{equation}
Moreover, by \eqref{Amaggioreidunmezzo} we have that for every $r\in (0,r_0]$
\begin{equation}\label{stimaprimoterm}
\int_{B_r} A\nabla U \cdot \nabla U \,dx \geq \frac{1}{2}\int_{B_r} |\nabla U|^2\, dx.
\end{equation}  
%choosing $r$ smaller, if necessary. 
Thus, combining \eqref{stimasecterm2} and \eqref{stimaprimoterm}, we easily infer \eqref{unificata} under assumption \eqref{asssuftilde1} for every $U\in H^1(B_r)$ such that $r\in (0,r_0]$, with $C:= 4/(N-1)$ and $\varepsilon=\delta$. In particular \eqref{serviraancora} easily follows from \eqref{rdelta}.

Now we turn to prove the validity of \eqref{unificata} under assumption \eqref{asssuftilde2}. To this aim, being $S_{N,p}$ as in \eqref{SNp}, thanks to the H\"{o}lder inequality, \eqref{stimadiL2star} and \eqref{mumaggiore}, we have that
\begin{equation}\label{eq2}
\begin{split}
\int_{B_r}|\tilde{f}| U^2 \, dx&\leq S_{N,p}\, r ^{\frac{2p-N-1}{p}}\left(\int_{B_{\tilde{r}}}|\tilde{f}|^p dx\right)^{\frac{1}{p}}\left(\int_{B_{r}}|U(x)|^{2^\ast} dx\right)^{\frac{2}{2^\ast}}\\
&\leq \,S_{N,p}C_{N,p} \Vert \tilde{f}\Vert_{L^p(B_{\tilde{r}})} r ^{\frac{2p-N-1}{p}} \left( \int_{B_r}|\nabla U|^2\, dx +  \frac{N-1}{2r}\int_{\partial B_r}U^2\,ds\right)\\
&\leq \frac{1}{4}\int_{B_r}|\nabla U|^2\, dx +  (N-1)S_{N,p}C_{N,p} \Vert \tilde{f}\Vert_{L^p(B_{\tilde{r}})} \, r ^{-1+\frac{2p-N-1}{p}}\int_{\partial B_r}\mu U^2\,ds,\\
\end{split}
\end{equation}
for every $r\in (0,r_0]$ with $r_0\in (0,\tilde{r})$ such that 
\begin{equation}\label{r0p}
r_0^{\frac{2p-N-1}{p}}< (4S_{N,p}C_{N,p} \Vert \tilde{f}\Vert_{L^p(B_{\tilde{r}})} )^{-1}.
\end{equation} 
Thus, putting together \eqref{stimaprimoterm} and \eqref{eq2}, we get \eqref{unificata} under assumption \eqref{asssuftilde2} for every $U\in H^1(B_r)$ such that $r\in (0,r_0]$, with $C:= (N-1)S_{N,p}C_{N,p} \Vert \tilde{f}\Vert_{L^p(B_{\tilde{r}})}$ and $\varepsilon=(2p-N-1)/p$. In particular \eqref{serviraancora} immediately follows from \eqref{r0p}. 
\end{proof}
\subsection{Approximation argument}
We now really dive into the focal part of our dissertation: it will turn out that removing the star-shapedness condition on the complement of the crack (that is \emph{downstairs}, since $\Gamma^c$ lies on the hyperplane $\{x_{N+1}=0\}$) has the price of making us able to prove an ``almost" star-shapedness (see \eqref{tipostellatura} below) for the approximating domains (that is \emph{upstairs}), due to the presence of the matrix $A$. 
 
We start the construction of the approximating domains by letting $\eta\in C^\infty([0,+\infty))$  be such that 
$0\leq \eta\leq 1$, $\eta'\leq 0$ and 
\begin{equation}\label{eta}
\eta(t) = \begin{cases}
1&\text{if }t\leq 1/2,\\[5pt]
0&\text{if } t\geq 1.
\end{cases}
\end{equation}
Then we fix any real $\alpha>1$ and we introduce $f:[0,+\infty)\to\mathbb{R}$ such that 
\begin{equation}\label{f}
f(t):= \eta(t)+(1-\eta(t))t^{1/\alpha}\quad \text{for every $t\geq 0$.}
\end{equation} 
It holds that $f\in C^\infty([0,+\infty))$ and 
\begin{equation}\label{propertyf}
f(t)-\alpha t f'(t)\geq 0 \quad \text{for every $t\geq 0$.}
\end{equation}
Accordingly, we define a sequence of smooth functions given by 
\begin{equation}\label{fn}
f_n(t):=f(nt)n^{-1/2\alpha} \quad \text{for every $n\geq 1$ and $t\geq 0$,}
\end{equation}
which inherits \eqref{propertyf}, namely we have that
\begin{equation}\label{propertyfn}
f_n(t)-\alpha t f'_n(t)\geq 0 \quad \text{for every $t\geq 0$}.
\end{equation}
From \eqref{eta}, \eqref{f} and \eqref{fn}, we can deduce that  
\begin{equation}\label{fn0}
f_n(0) = n^{-1/2\alpha}. 
\end{equation}
In order to suitably define the approximating domains, we introduce the following sequence of functions 
\begin{equation*}\label{ftilden}
\tilde{f}_n(t) := f_n(|t|) \quad \text{for every $n\geq 1$ and $t\in \mathbb{R}$},
\end{equation*} 
which, thanks to \eqref{propertyfn}, satisfies 
\begin{equation}\label{propertyftilden}
\tilde{f}_n(t)-\alpha t \tilde{f}'_n(t)\geq 0 \quad \text{for every $t\in \mathbb{R}$}.
\end{equation} 
Thus, for every $r\in (0,r_0]$ and $n\in\mathbb{N}\setminus \{0\}$, let
\begin{equation*}
\mathcal{B}_{r,n}:= B_{r} \cap  \{(x',x_N,x_{N+1})\in \mathbb{R}^{N-1}\times \mathbb{R}\times \mathbb{R}: x_N<\tilde{f}_n(x_{N+1})\},
\end{equation*}
see Figure \ref{fig:approx-domains} below. 
The topological boundary of $\mathcal{B}_{r,n}$ can be written as follows 
\begin{equation*}
\partial \mathcal{B}_{r,n} = \overline{\mathcal{S}_{r,n} \cup \gamma_{r,n}}
\end{equation*} 
where
\begin{equation*}\label{Stilden}
\mathcal{S}_{r,n}:= \partial B_{r}\cap \{(x',x_N,x_{N+1})\in \mathbb{R}^{N-1}\times \mathbb{R}\times \mathbb{R}: x_N<\tilde{f}_n(x_{N+1})\}
\end{equation*}
and
\begin{equation}\label{gamman}
\gamma_{r,n}:= B_{r} \cap \{(x',x_N,x_{N+1})\in \mathbb{R}^{N-1}\times \mathbb{R}\times \mathbb{R}: x_N=\tilde{f}_n(x_{N+1})\}.
\end{equation}
The result in the following lemma 
will allow us to get rid of a boundary term over $\gamma_{r,n}$ in the Pohozaev identity, giving rise to a Pohozaev inequality.
\begin{lem}\label{lemtipostellatura}
For all $r\in (0,r_0]$ there exists $\bar{n}=\bar{n}(r)\in\mathbb{N}\setminus\{0\}$ sufficiently large such that for all $n\geq \bar{n}$ 
\begin{equation}\label{tipostellatura}
Ax\cdot\nu(x)\geq 0\quad \text{for every $x\in \gamma_{r,n}$},
\end{equation}
where $\nu(x)$ denotes the unit outward normal vector at $x\in \partial \mathcal{B} _{r,n}$. 
\end{lem}
%We highlight that if $n$ is not large enough in dependence on $r$ in the statement of Lemma \ref{lemtipostellatura}, then $\gamma_{r,n}$ is empty. Indeed, for instance, in the case where $x'=0$, if $n$ is not large enough in dependence on $r$, point $(0,n^{-1/2\alpha},0)$ might be located to the right of point $(0,r,0)$. This turns out to be evident in 
\begin{figure}[ht]
\begin{tikzpicture}[line cap=round,line join=round,>=triangle 45, scale=2]
\draw [->] (-1.4,0)-- (1.4,0) node [right] {$x_N$};
\draw [->] (0,-1.4)-- (0,1.4) node [right] {$x_{N+1}$};
%\draw [black, opacity=0.3] (-1,0) to [out=90, in=180] (0,1) to [out=0, in=120] (0.866,0.5) to [out=238, in=20] (0.43,0.15) to [out=200, in=90]  (0.33,0.076) to [out=270, in=90] (0.35,0.040) to [out=270, in=90] (0.35,0) -- (-1,0);
\draw  (-1,0) to [out=90, in=180] (0,1) to [out=0, in=120] (0.866,0.5) to [out=238, in=20] (0.43,0.15) to [out=200, in=90]  (0.33,0.076) to [out=270, in=90] (0.35,0.040) to [out=270, in=90] (0.35,0) -- (-1,0);
\draw  (-1,0) to [out=270, in=180] (0,-1)  to [out=0, in=240] (0.866,-0.5) to [out=122, in=340] (0.43,-0.15) to [out=340, in=270]  (0.33,-0.076) to [out=90, in=270] (0.35,-0.040) to [out=90, in=270] (0.35,0) -- (-1,0); 

%\draw[fill] (0.35,0) circle [radius=0.010];
%			\node [below] at (0.33,0) {$n^{-1/8}$};
%			\draw[color=black] (-0.5,-0.12) node {$\sigma_n$};
\draw[color=black] (0.73,-0.1008) node {$n^{-1/8}$};
\draw[color=black] (-0.65,0.5) node {$\mathcal{B}_{r,n}$};		\draw[color=black] (0.866,0.2) node {$\mathcal{\gamma}_{r,n}$};	
\end{tikzpicture}
   \caption{A representation of the set $\mathcal{B}_{r,n}$ when $x'=0$ and $\alpha=1/4$.}
\label{fig:approx-domains}
\end{figure}
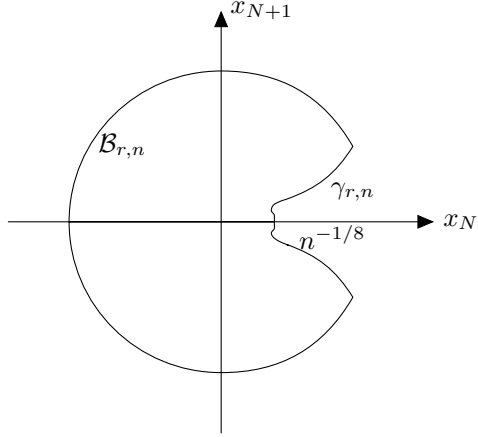
\begin{proof}
Let $r\in (0,r_0]$. We choose $\bar{n}\in\mathbb{N}\setminus \{0\}$ sufficiently large so that 
\begin{equation*}
\bar{n}^{1/2\alpha}>\frac{1}{r},
\end{equation*}
and we consider $n\geq \bar{n }$ in order to make sure that $\gamma_{r,n}$ is non-empty and consequently $\mathcal{B}_{r,n}\subsetneq B_r$, thanks to \eqref{fn0} (this is evident in Figure \ref{fig:approx-domains}).
We immediately observe that if $x=(x',x_N,x_{N+1})\in\gamma_{r,n}$, then by definition of $\gamma_{r,n}$ (given in \eqref{gamman})
\begin{equation}\label{def2}
x_N= \tilde{f}_n(x_{N+1})
\end{equation}
and the outward unit normal vector to $\partial \mathcal{B}_{r,n}$ at $x$ is given by 
\begin{equation*}
\nu(x)=\frac{(\underline{0},1, -\tilde{f}'_n(x_{N+1}) )}{\sqrt{1+[\tilde{f}'_n(x_{N+1})]^2}}. 
\end{equation*}
Combining this with \eqref{Aablocchi}, \eqref{Dablocchi} and \eqref{determinante}, we obtain that 
\begin{equation}\label{prodottoscalare}
\sqrt{1+[\tilde{f}'_n(x_{N+1})]^2} [Ax\cdot\nu(x)]=(1+O(|x'|)+O(x_N))(x_N-x_{N+1}\tilde{f}'_n(x_{N+1})).
\end{equation} 
At this point we notice that possibly choosing $r_0$ smaller from the beginning, for every $x\in B_{r_0}$
\begin{equation}\label{C1}
C_1\leq 1+O(|x'|)+O(x_N) \leq C_2,
\end{equation}
for some two positive constant $C_1<1$ and $C_2>1$ such that 
\begin{equation}\label{mindialpha}
\frac{C_2}{C_1}<\alpha;
\end{equation}
so in particular \eqref{C1} holds for every $x\in \gamma_{r,n}$.
Now if $x_{N+1}$ and $\tilde{f}'_n(x_{N+1})$ have the same sign, we put together \eqref{prodottoscalare} and \eqref{C1} to find that  
\begin{equation*}
\begin{split}
\sqrt{1+[\tilde{f}'_n(x_{N+1})]^2} [Ax\cdot\nu(x)]
&\geq C_1x_N - C_2 x_{N+1} \tilde{f}'_n(x_{N+1}) \\
&\geq C_1 (\tilde{f}_n(x_{N+1})- \alpha x_{N+1} \tilde{f}'_n(x_{N+1}))\geq 0,
\end{split}
\end{equation*}
using in addition \eqref{def2}, \eqref{mindialpha} and \eqref{propertyftilden}.
Instead, if $x_{N+1}$ and $\tilde{f}'_n(x_{N+1})$ have opposite signs,  
%either $t\leq 0$ and $\tilde{f}'_n(t)\geq 0$ or $t\geq 0$ and $\tilde{f}'_n(t)\leq 0$, 
we simply take advantage of the first inequality in  \eqref{C1}, which in turn also implies that $(1+O(|x'|)+O(x_N))x_{N+1} \tilde{f}'_n(x_{N+1})\leq 0$, to conclude that 
\begin{equation*}
\begin{split}
\sqrt{1+[\tilde{f}'_n(x_{N+1})]^2} [Ax\cdot\nu(x)]\geq C_1x_N= C_1\tilde{f}_n(x_{N+1}) \geq 0,
\end{split}
\end{equation*}
as a consequence of \eqref{def2} and by construction of $\tilde{f}_n$. 
\end{proof}

From now on, we fix a non-trivial solution $U\in H^1_{\tilde{\Gamma}}(B_{\tilde{r}})$ to \eqref{problU}. 
Then, inspired by the approximation technique developed in \cite{DelFel} and \cite{DelFelVit}, we introduce a sequence of boundary value problems on the approximating domains and we prove that the resulting sequence of solutions $\{U_n\}$ converges to $U$. The goal of performing such a construction is to derive a Pohozaev-type identity (actually inequality by making use of Lemma \ref{lemtipostellatura}) for each $U_n$ which enjoys enough regularity, and consequently for $U$.

More precisely, let $r_0\in (0,\tilde{r})$ be as in Lemma \ref{dis} and $\bar{n}=\bar{n}(r_0) \in \mathbb{N}\setminus \{0\}$ as in Lemma \ref{lemtipostellatura}.  
For every $n\geq \bar{n}$ we consider the boundary value problem
\begin{equation}\label{indicizzato}
\begin{cases}
-\mathrm{div}\left(A\nabla U_n\right)=\tilde fU_n &\mathrm{in \ } \mathcal{B}_{r_0,n},\\
U_n=G_n & \mathrm{on \ } \partial \mathcal{B}_{r_0,n},
\end{cases}
\end{equation}
where $\{G_n\}\subseteq C^\infty(\overline{B_{\tilde{r}}}\setminus\tilde{\Gamma})$ and
\begin{equation}\label{gnconvau}
G_n\to U \quad \text{in $H^1(B_{\tilde{r}})$}.
\end{equation}
Without loss of generality we can suppose that $G_n$ vanishes on $\gamma_{\tilde{r},n}$ for each fixed $n$. In particular, by a weak solution to \eqref{indicizzato} we mean a function $U_n\in H^1(\mathcal{B}_{r_0,n})$ such that 
\begin{equation*}
U_n = G_n\quad \text{on $\partial \mathcal{B}_{r_0,n}$ in the trace sense,}
\end{equation*}
and 
\begin{equation*}\label{soluzdebole}
\int_{\mathcal{B}_{r_0,n}} A\nabla U_n \cdot \nabla V \, dx = \int_{\mathcal{B}_{r_0,n}} \tilde{f}U_n V\, dx \quad \text{for every $V\in H^1_0(\mathcal{B}_{r_0,n})$}. 
\end{equation*} 
\begin{prop}\label{convdelleun}
It holds that
\begin{itemize}
\item[(i)] for every $n\geq \bar{n}$ problem \eqref{indicizzato} admits a unique solution $U_n$;
\item[(ii)] $U_n\to U$ in $H^1(B_{r_0})$ after extending each $U_n$ to zero in $B_{r_0}\setminus\mathcal{B}_{r_0,n}$. 
\end{itemize}
\end{prop}
\begin{proof}
To prove both points (i) and (ii), it is more convenient to study the following homogeneous boundary value problem 
\begin{equation}\label{indicizzatoomogeneo}
\begin{cases}
-\mathrm{div}\left(A\nabla V_n\right)=\tilde fV_n +\tilde{f}G_n+\mathrm{div}\left (A\nabla G_n\right)&\mathrm{in \ } \mathcal{B}_{r_0,n},\\
V_n=0 & \mathrm{on \ } \partial \mathcal{B}_{r_0,n}.
\end{cases}
\end{equation}
We observe that $U_n$ is a weak solution to \eqref{indicizzato} if and only if $V_n:=U_n-G_n$ is a weak solution to \eqref{indicizzatoomogeneo} for each fixed $n\geq \bar{n}$, in the sense that 
$V_n\in H^1_0(\mathcal{B}_{r_0,n})$ and 
\begin{equation}\label{uguaglianzatrafunzionali}
\int_{\mathcal{B}_{r_0,n}}[ A \nabla V_n \cdot\nabla \varphi-\tilde{f} V_n\varphi]\, dx= \int_{\mathcal{B}_{r_0,n}}[\tilde{f} G_n\varphi+\mathrm{div}(A\nabla G_n)\varphi]\, dx \quad\text{for every $\varphi\in H^1_0(\mathcal{B}_{r_0,n})$}. 
\end{equation}
First, we will prove that problem \eqref{indicizzatoomogeneo} admits a unique solution $V_n$ for every $n\geq \bar{n}$.
From this, it will follow that also problem \eqref{indicizzato} has a unique solution given by
\begin{equation}\label{Unelasomma}
U_n=V_n+G_n,
\end{equation}
and thus (i) will be proved.
For every $n\geq \bar{n}$, we set
\begin{equation*}
\begin{split}
&Q_n: H^1_0(\mathcal{B}_{r_0,n})\times H^1_0(\mathcal{B}_{r_0,n}) \to \mathbb{R},\\ 
&Q_n(\psi,\varphi):= \int_{\mathcal{B}_{r_0,n}}[ A \nabla \psi \cdot\nabla \varphi-\tilde{f} \psi\varphi]\, dx,
\end{split}
\end{equation*}
and
\begin{equation*}
\begin{split}
&F_n:H^1_0(\mathcal{B}_{r_0,n})\to \mathbb{R},\\
&F_n(\varphi):=\int_{\mathcal{B}_{r_0,n}}[\tilde{f} G_n+\mathrm{div}(A\nabla G_n)]\varphi\, dx.
\end{split}
\end{equation*} 
We observe that $F_n$ is a linear and bounded operator on $H^1_0(\mathcal{B}_{r_0,n})$: indeed, after extending $\varphi\in H^1_0(\mathcal{B}_{r_0,n})$ to zero in $B_{r_0}\setminus \mathcal{B}_{r_0,n}$ using \eqref{asssuftilde1}, \eqref{hardy}, the boundedness of $\{G_n\}$ in $H^1(B_{r_0})$ by \eqref{gnconvau}, and the continuity of the trace map 
\begin{equation}\label{traceoperator}
H^1(B_r) \to L^2(\partial B_r)\quad\text{for every $r>0$}, 
\end{equation}
we have 
\begin{equation*}
\begin{split}
\left|\int_{\mathcal{B}_{r_0,n}}\tilde{f} G_n\varphi\,dx\right|
\leq& \ \frac{4r_0^\delta}{(N-1)^2} \left(\int_{B_{r_0}}|\nabla G_n|^2\, dx + \frac{N-1}{2r_0}\int_{\partial B_{r_0}}G_n^2\,ds\right)^{1/2}\cdot \left(\int_{B_{r_0}}|\nabla \varphi|^2\, dx\right)^{1/2}\\
\leq & \ \tilde{C} \Vert \varphi\Vert_{H^1_0(\mathcal{B}_{r_0,n})},
\end{split}
\end{equation*} 
for some $\tilde{C}>0$ independent of $n$;
integrating by parts, using \eqref{normadiA} and H\"{o}lder's inequality, we get 
\begin{equation*}
\begin{split}
\left|\int_{\mathcal{B}_{r_0,n}}\mathrm{div}(A\nabla G_n)\varphi\, dx\right|\leq& \ 2 \left(\int_{\mathcal{B}_{r_0,n}}|\nabla G_n|^2\,dx\right)^{1/2}\left(\int_{\mathcal{B}_{r_0,n}}|\nabla \varphi|^2\,dx\right)^{1/2}\\
\leq & \ \tilde{C_1} \Vert \varphi\Vert_{H^1_0(\mathcal{B}_{r_0,n})},
\end{split}
\end{equation*}
for some $\tilde{C_1} >0$ independent of $n$. 
Arguing as above, we can deduce that also the bilinear form $Q_n$ is continuous on $H^1_0(\mathcal{B}_{r,n})$. Moreover, from Lemma \ref{dis}, we can deduce that
\begin{equation}\label{Qecoerciva}
Q_n(\psi,\psi)\geq \frac{1}{4}\Vert \psi\Vert^2_{H^1_0(\mathcal{B}_{r_0,n})}\quad\text{for every $\psi\in H^1_0(\mathcal{B}_{r_0,n})$},
\end{equation} 
and thus $Q_n$ is coercive on $H^1_0(\mathcal{B}_{r_0,n})$.  
By the Lax-Milgram Theorem problem \eqref{indicizzatoomogeneo} has a unique solution $V_n$ for every $n\geq \bar{n}$, and thus (i) is proved.  
As for (ii), we notice that, since 
 $\Vert V_n\Vert _{H^1_0(B_{r_0})}\leq 4(\tilde{C}+\tilde{C_1})$ (up to extend to zero $V_n$ in $B_{r_0}$ outside of $\mathcal{B}_{r_0,n}$), $\{V_n\}$ turns out to be bounded in $H^1_0(B_{r_0})$. Thus there exists a subsequence $\{V_{n_k}\}$ such that 
\begin{equation}\label{vnconvdeb}
V_{n_k} \rightharpoonup V\quad \text{ in $H^1_0(B_{r_0})$}
\end{equation}
for some $V\in H^1_0(B_{r_0})$.  
Furthermore $V$ has null trace on $\tilde{\Gamma}$: indeed $V_n$ has null trace on the set $$\{(x',x_N)\in \R^{N-1}\times \R: x_N\geq \delta\}\cap B_{r_0}$$ for every $\delta>0$, provided that $n$ is sufficiently large (in dependence on $\delta$); this is because $V_n$ is identically zero in $B_{r_0}\setminus \mathcal{B}_{r_0,n}$ and $$\{(x',x_N)\in \R^{N-1}\times \R: x_N\geq \delta\}\cap B_{r_0}\subset B_{r_0}\setminus \mathcal{B}_{r_0,n}$$ for every $\delta>0$, provided that $n$ is sufficiently large (in dependence on $\delta$).
Hence, by a density argument, we are allowed to take $\varphi=V$ in identity \eqref{weaksolu}, having that
\begin{equation}\label{uguale}
\int_{B_{r_0}} A\nabla U\cdot \nabla V\, dx-\int_{B_{r_0}} \tilde{f} UV\, dx=0.    
\end{equation}
On the other hand, using \eqref{uguaglianzatrafunzionali}, integrating by parts, and exploiting \eqref{gnconvau} and \eqref{vnconvdeb}, we have 
\begin{equation*}
\begin{split}
Q_{n_k}(V_{n_k},V_{n_k})= F_{n_k}(V_{n_k}) &= \int_{B_{r_0}} [\tilde{f} G_{n_k}V_{n_k}  + \mathrm{div}(A\nabla G_{n_k}) V_{n_k}]\, dx\\
&= \int_{B_{r_0}} \tilde{f} G_{n_k}V_{n_k}\,dx - \int_{B_{r_0}} A\nabla G_{n_k}\cdot \nabla V_{n_k}\, dx\\
&\to \int_{B_{r_0}}\tilde{f} U V\, dx - \int_{B_{r_0}} A \nabla U \cdot\nabla V\, dx\quad \text{as $k\to\infty$}.
\end{split}
\end{equation*}
Combining this with \eqref{Qecoerciva} and \eqref{uguale}, we deduce that $V_{n_k}\to 0$ in $H^1_0(B_{r_0})$.
Repeating the same argument, one can find out that any other subsequence of $\{V_n\}$ always admits limit equal to 0. Hence, by Urysohn's subsequence principle we can conclude that $V_n\to 0$ in $H^1_0(B_{r_0})$, and therefore, putting together \eqref{Unelasomma} and \eqref{gnconvau}, we infer that
$U_n\to U$ in $H^1(B_{r_0})$, thus proving (ii).
\end{proof}
\section{A Pohozaev-type inequality}\label{section2}
In this section we prove a Pohozaev-type inequality satisfied by any weak solution $U\in H^1_{\tilde{\Gamma}}(B_{\tilde{r}})$ to \eqref{problU}. The strategy of the proof consists in deriving first a ``family" of Pohozaev-type identities and consequently inequalities (taking advantage of the result in Lemma \ref{lemtipostellatura}) for the family of solutions $\{U_n\}$ to \eqref{indicizzato}. 
Then we use the convergence of $\{U_n\}$  to $U$ in $H^1$-sense (proved in Proposition \ref{convdelleun}) to infer a Pohozaev-type inequality for $U$ as well. 
%In this section we will derive two different Pohozaev-type inequalities for any weak solution $u\in H^1(B_{\tilde{r}})$ to \eqref{problU} depending on whether we are under assumption \eqref{asssuftilde1} or \eqref{asssuftilde2}. 
\begin{prop}\label{propopoho}
Let $f$ satisfy either \eqref{asssuftilde1} or \eqref{asssuftilde2}. Let $U\in H^1_{\tilde{\Gamma}}(B_{\tilde{r}})$ be a weak solution to problem \eqref{problU} and let $\nu=\nu(x)=x/|x|$ for every $x\in \partial B_r$. Then, for a.e. $r\in(0, r_0)$ 
\begin{equation}\label{poho1}
\begin{split}
r\int_{\partial B_r} (A\nabla U\cdot \nabla U)\, ds &- 2r\int_{\partial B_r}\frac{1}{\mu}(A\nabla U \cdot\nu)^2\,ds\geq \int_{B_r} (\mathrm{div}\beta)A\nabla U\cdot \nabla U \,dx \\
&-2\int_{B_r}\mathrm{Jac}\beta(A\nabla U)\cdot\nabla U\,dx+\int_{B_r} (dA\nabla U \nabla U)\cdot \beta\, dx\\
&+ 2\int_{B_r} (\beta\cdot \nabla U)\tilde{f} U\, dx
\end{split}
\end{equation}
under assumption \eqref{asssuftilde1}, and
\begin{equation}\label{poho}
\begin{split}
r\int_{\partial B_r} (A\nabla U\cdot \nabla U)\, ds &- 2r\int_{\partial B_r}\frac{1}{\mu}(A\nabla U \cdot\nu)^2\,ds\geq \int_{ B_r} (\mathrm{div}\beta)A\nabla U\cdot \nabla U \,dx \\
&-2\int_{B_r}(\mathrm{Jac}\beta)(A\nabla U)\cdot\nabla U\,dx+\int_{B_r} (dA\nabla U \nabla U)\cdot \beta\, dx\\
&+ r\int_{\partial B_r} \tilde{f} U^2\, ds-\int_{B_r} (\tilde{f}\mathrm{div}\beta+\nabla \tilde{f}\cdot\beta) U^2\, dx
\end{split}
\end{equation}

under assumption \eqref{asssuftilde2}.
\end{prop}
\begin{proof}  
We start by observing that any solution $U_n$ to problem \eqref{indicizzato} satisfies $U_n\in H^2(\mathcal{B}_{r,n}\setminus B_\delta)$ for all $r\in (0,r_0)$, $n\geq \bar{n}$, being $\bar{n}=\bar{n}(r)$ as in Lemma \ref{lemtipostellatura}, and $\delta<1/n^{2\alpha}$ (see Figure \ref{fig:approx-domains2}).
This descends from \cite[Section 2.4]{gri}, since by assumption $\tilde {f}\in L^\infty_{\mathrm{loc}}(B_{r_0}\setminus \{0\})$ and consequently $\tilde{f}U_n\in L^2_{\mathrm{loc}}(\mathcal{B}_{r_0,n}\setminus \{0\})$, the equation in \eqref{indicizzato} holds true in a smooth domain containing $\mathcal{B}_{r,n}\setminus B_\delta$ and in virtue of \eqref{Aelipschitz} and interior regularity.  
\begin{figure}[ht]
\begin{tikzpicture}[line cap=round,line join=round,>=triangle 45, scale=2]
\draw [->] (-1.4,0)-- (1.4,0) node [right] {$x_N$};
\draw [->] (0,-1.4)-- (0,1.4) node [left] {$x_{N+1}$};
%\draw [black, opacity=0.3] (-1,0) to [out=90, in=180] (0,1) to [out=0, in=120] (0.866,0.5) to [out=238, in=20] (0.43,0.15) to [out=200, in=90]  (0.33,0.076) to [out=270, in=90] (0.35,0.040) to [out=270, in=90] (0.35,0) -- (-1,0);
\draw  (-1,0) to [out=90, in=180] (0,1) to [out=0, in=120] (0.866,0.5) to [out=238, in=20] (0.43,0.15) to [out=200, in=90]  (0.33,0.076) to [out=270, in=90] (0.35,0.040) to [out=270, in=90] (0.35,0) -- (-1,0);
\draw  (-1,0) to [out=270, in=180] (0,-1)  to [out=0, in=240] (0.866,-0.5) to [out=122, in=340] (0.43,-0.15) to [out=340, in=270]  (0.33,-0.076) to [out=90, in=270] (0.35,-0.040) to [out=90, in=270] (0.35,0) -- (-1,0); 

%\draw[fill] (0.35,0) circle [radius=0.010];
%			\node [below] at (0.33,0) {$n^{-1/8}$};
%			\draw[color=black] (-0.5,-0.12) node {$\sigma_n$};
%\draw[color=black] (-0.6,0.9) node {$\tau_n$};
%\draw[color=black] (-0.85,0.8) node {$\mathcal{B}_n$};
\draw  (-0.15,0) to [out=90, in=180] (0,0.15) to [out=0, in=90] (0.15,0) to [out=270, in=0] (0,-0.15) to [out=180, in=270] (-0.15,0);
\draw[color=black] (-0.15,0.22) node {$B_\delta$}; 
\draw  (-0.7,0) to [out=90, in=180] (0,0.7) to [out=0, in=120] (0.7,0.32) to [out=238, in=20] (0.43,0.15) to [out=200, in=90]  (0.33,0.076) to [out=270, in=90] (0.35,0.040) to [out=270, in=90] (0.35,0) -- (-1,0);
\draw (0,-0.7)  to [out=0, in=240] (0.7,-0.32) to [out=122, in=340] (0.43,-0.15) to [out=340, in=270]  (0.33,-0.076) to [out=90, in=270] (0.35,-0.040) to [out=90, in=270] (0.35,0) -- (-1,0);
\draw (0,-0.7) to [out=180, in=270] (-0.7,0);
\draw[color=black] (-0.599,0.6) node {$\mathcal{B}_{r,n}$};
\draw[color=black] (-1,0.7) node {$\mathcal{B}_{r_0,n}$};
\end{tikzpicture}
   \caption{A representation of the set $\mathcal{B}_{r,n}\setminus B_\delta$ with $r\in (0,r_0)$, $x'= 0$, $\alpha=1/4$}
\label{fig:approx-domains2}
\end{figure}
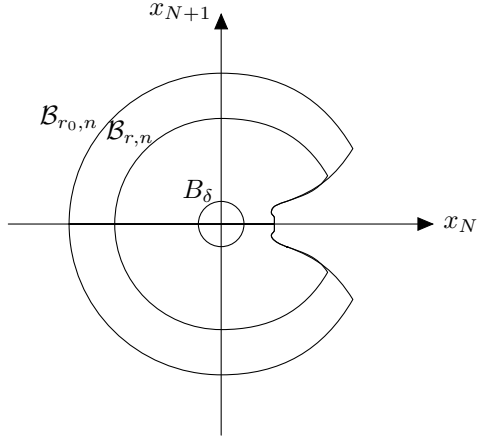

In particular we deduce that 
$$\mathrm{div}((A\nabla U_n\cdot\nabla U_n)\beta - 2 A\nabla U_n (\beta\cdot\nabla U_n))\in L^1(\mathcal{B}_{r,n}\setminus B_\delta).$$ So we can integrate the Rellich-Ne\u{c}as identity 
\begin{equation*} \label{rellich-neca}
\begin{split}
\mathrm{div}((A\nabla U_n\cdot\nabla U_n)\beta - 2 A\nabla U_n (\beta\cdot\nabla U_n))=& \,(\mathrm{div}\beta) A\nabla U_n \cdot \nabla U_n - 2 \mathrm{Jac}\beta (A\nabla U_n)\cdot \nabla U_n \\
&+ (dA\nabla U_n \nabla U_n)\cdot \beta - 2 (\beta\cdot \nabla U_n)\mathrm{div} (A\nabla U_n)
\end{split}
\end{equation*}
over $\mathcal{B}_{r,n}\setminus B_\delta$ and then, by applying \cite[Proposition 2.7]{hofmann2010singular}, we obtain that 
\begin{equation}\label{idenityt}
\begin{split}
r&\int_{\mathcal{S}_{r,n}}(A\nabla U_n\cdot\nabla U_n)\,ds- 2r\int_{\mathcal{S}_{r,n}} \frac{1}{\mu}( A\nabla U_n \cdot \nu)^2\, ds  - \int_{\gamma_{r,n}} \frac{1}{\mu}\left|\frac{\partial U_n}{\partial\nu}\right|^2(A\nu\cdot\nu)(Ax\cdot \nu)\, ds\\
&- \delta \int_{\partial B_\delta} (A\nabla U_n\cdot\nabla U_n)\,ds  + 2 \delta\int_{\partial B_\delta}\frac{1}{\mu} (A\nabla U_n \cdot \nu)^2\, ds = \int_{\mathcal{B}_{r,n}\setminus B_\delta} (\mathrm{div}\beta) A\nabla U_n\cdot \nabla U_n \, dx \\
& - 2\int_{\mathcal{B}_{r,n}\setminus B_\delta} \mathrm{Jac}\beta (A\nabla U_n)\cdot \nabla U_n\, dx +\int_{\mathcal{B}_{r,n}\setminus B_\delta} (dA\nabla U_n\nabla U_n)\cdot \beta\,dx\\ 
& + 2 \int_{\mathcal{B}_{r,n}\setminus B_\delta} (\beta
\cdot \nabla U_n) \tilde{f} U_n\, dx, 
\end{split}
\end{equation} 
where $\nu$ denotes the outward unit normal vector to $ \partial(\mathcal{B}_{r,n}\setminus B_\delta)$. To get \eqref{idenityt}, we used that $\beta\cdot x=|x|^2$ by \eqref{beta} and \eqref{mu}; in addition, if $x\in\mathcal{S}_{r,n}$ then $\nu=x/r$, if $x\in\partial B_\delta$ then $\nu=-x/\delta$; at last, the tangential component of $\nabla U _n $ is null on $\gamma_{r,n}$ since $U_n=G_n=0$ on $\gamma_{\tilde{r},n}\supset\gamma_{r,n}$ in the trace sense, and hence $|\nabla U_n|^2=|\frac{\partial U_n}{\partial\nu}|^2$.

Now we use Lemma \ref{lemtipostellatura} and \eqref{Amaggioreidunmezzo} to infer that 
\begin{equation*}
\int_{\gamma_{r,n}} \frac{1}{\mu} \left|\frac{\partial U_n}{\partial\nu}\right|^2(A\nu\cdot\nu) (Ax\cdot\nu)\,ds\geq 0,
\end{equation*}
and consequently \eqref{idenityt} turns into the following inequality 
\begin{equation}\label{ineq}
\begin{split}
r&\int_{\mathcal{S}_{r,n}}(A\nabla U_n\cdot\nabla U_n)\,ds- 2r\int_{\mathcal{S}_{r,n}} \frac{1}{\mu}( A\nabla U_n \cdot \nu)^2\, ds  \\
&- \delta \int_{\partial B_\delta} (A\nabla U_n\cdot\nabla U_n)\,ds  + 2 \delta\int_{\partial B_\delta}\frac{1}{\mu} (A\nabla U_n \cdot \nu)^2\, ds \geq  \int_{\mathcal{B}_{r,n}\setminus B_\delta} (\mathrm{div}\beta) A\nabla U_n\cdot \nabla U_n \, dx \\
& - 2\int_{\mathcal{B}_{r,n}\setminus B_\delta} \mathrm{Jac}\beta (A\nabla U_n)\cdot \nabla U_n\, dx +\int_{\mathcal{B}_{r,n}\setminus B_\delta} (dA\nabla U_n\nabla U_n)\cdot \beta\,dx\\ 
& + 2 \int_{\mathcal{B}_{r,n}\setminus B_\delta} (\beta
\cdot \nabla U_n) \tilde{f} U_n\, dx.
\end{split}
\end{equation}
For the sake of convenience, from now on we split the proof into two steps.

\emph{Step one.} 
In this step we will consider the limit as $\delta \searrow 0$ in \eqref{ineq}. To such aim, we first observe that, since $U_n\in H^1(\mathcal{B}_{{r_0},n})$ and $\tilde{f}\in L^\infty_{\mathrm{loc}}(\mathcal{B}_{{r_0},n}\setminus \{0\})$, then for all $r\in (0,r_0)$ and $n\geq \bar{n}$ there exists a sequence $\{\delta_h\}\searrow 0$ such that 
\begin{equation}\label{nuova}
\delta_h\int_{\partial B_{\delta_h}}(|\nabla U_n|^2+|\tilde{f}|U_n^2)\, ds \to 0 \quad \text{as $h\to\infty$}.
\end{equation} 
So, combining this with \eqref{mumaggiore} and \eqref{normadiA}, we can deduce that 
\begin{equation}\label{sommavaa0}
\delta_h \left[2\int_{\partial B_{\delta_h}} \frac{1}{\mu} (A\nabla U_n\cdot \nu)^2\, ds - 
\int_{\partial B_{\delta_h}} (A\nabla U_n\cdot\nabla U_n)\,ds\right]\to 0\quad\text{as $h \to \infty$}.
\end{equation}
Next, by \eqref{normadiA}, \eqref{divbetalimitata}, \eqref{jaclimitata}, \eqref{leproprietadidA}, \eqref{betaequasix}, and the fact that $U_n\in H^1(B_{r})$ after extending $U_n$ to zero in $B_r\setminus \mathcal{B}_{r,n}$, applying the absolute continuity of the Lebesgue integral, we have that as $\delta\searrow 0$
\begin{equation}\label{imp}
\begin{split}
\int_{\mathcal{B}_{r,n}\setminus B_\delta}(\mathrm{div}\beta)A\nabla U_n\cdot\nabla U_n \, dx=\int_{B_r\setminus B_\delta}(\mathrm{div}\beta)A\nabla U_n\cdot\nabla U_n \, dx &\to \int_{B_r} (\mathrm{div}\beta)A\nabla U_n\cdot\nabla U_n \, dx,\\
\int_{\mathcal{B}_{r,n}\setminus B_\delta} \mathrm{Jac}\beta (A\nabla U_n)\cdot\nabla U_n\,dx= \int_{B_r\setminus B_\delta}\mathrm{Jac}\beta (A\nabla U_n)\cdot\nabla U_n\,dx&\to \int_{B_r} \mathrm{Jac}\beta (A\nabla U_n)\cdot\nabla U_n\,dx,\\
\int_{\mathcal{B}_{r,n}\setminus B_\delta} (dA\nabla U_n\cdot\nabla U_n)\cdot\beta\, dx =\int_{B_r\setminus B_\delta}(dA\nabla U_n\cdot\nabla U_n)\cdot\beta\, dx&\to \int_{B_r} (dA\nabla U_n\cdot\nabla U_n)\cdot\beta\, dx.
\end{split}
\end{equation}
Moreover, under assumption \eqref{asssuftilde1}, using \eqref{betaequasix} and applying the H\"{o}lder inequality to $|\nabla U_n|$ and $|U_n|/|x|$, we obtain that
\begin{equation*}\label{ispirata}
\begin{split}
\int_{B_r}|(\beta\cdot\nabla U_n)\tilde{f}U_n|\, dx &\leq\mathrm{const}\, r^\delta \int_{B_r} |\nabla U_n|\cdot \frac{|U_n|}{|x|}\, dx\\
& \leq\mathrm{const}\, r^\delta \left(\int_{B_r} |\nabla U_n|^2\, dx\right)^{1/2}\left(\int_{B_r} \frac{|U_n(x)|^2}{|x|^2}\,dx\right)^{1/2},
\end{split}
\end{equation*}
for some $\mathrm{const}>0$ independent of $r$. 
From this, \eqref{hardy} and the fact that $\{U_n\}$ is bounded in $H^1(B_r)$ thanks to Proposition \ref{convdelleun}, we can infer that $(\beta\cdot\nabla U_n)\tilde{f}U_n\in L^1(B_r)$. Therefore by Lebesgue's dominated convergence theorem, we can conclude that as $\delta\searrow 0$
\begin{equation}\label{imp2}
\int_{\mathcal{B}_{r,n}\setminus B_{\delta}}(\beta\cdot \nabla U_n) \tilde{f}U_n \, dx=\int_{B_r\setminus B_\delta}(\beta\cdot \nabla U_n) \tilde{f}U_n \, dx \to \int_{B_r} (\beta\cdot \nabla U_n) \tilde{f}U_n \, dx.
\end{equation} 
Taking into account \eqref{sommavaa0}, \eqref{imp} and \eqref{imp2}, passing to the limit as $h\to\infty$ in \eqref{ineq} with $\delta=\delta_h$, we get that for all $r\in (0,r_0)$ and $n\geq \bar{n}$
\begin{equation}\label{ineq2}
\begin{split}
r\int_{\partial B_r}(A\nabla U_n\cdot\nabla U_n)\,ds&- 2r\int_{\partial B_r} \frac{1}{\mu}( A\nabla U_n \cdot \nu)^2\, ds  \geq   \int_{B_r} (\mathrm{div}\beta) A\nabla U_n\cdot \nabla U_n \, dx \\
& - 2\int_{B_r} \mathrm{Jac}\beta(A\nabla U_n)\cdot \nabla U_n\, dx +\int_{B_r} (dA\nabla U_n\nabla U_n)\cdot \beta\,dx\\ 
& +2\int_{B_r} (\beta\cdot \nabla U_n)\tilde{f} U_n\, dx,
\end{split}
\end{equation}
under assumption \eqref{asssuftilde1}. 

On the other hand, under assumption \eqref{asssuftilde2}, 
since it holds that
\begin{equation*}
(\beta\cdot \nabla U_n) \tilde{f} U_n= \tilde{f} (\beta\cdot \nabla (U_n^2/2)),
\end{equation*}
we use the divergence theorem to rewrite the last term in \eqref{ineq} as follows
\begin{equation}\label{div}
\begin{split}
\int_{\mathcal{B}_{r,n}\setminus B_\delta}(\beta\cdot \nabla U_n) \tilde{f} U_n\, dx =& \,\frac{r}{2}\int_{\mathcal{S} _{r,n}}\tilde{f}U_n^2\,ds-\frac{\delta}{2}\int_{\partial B_\delta} \tilde{f} U_n^2\,ds\\
& -\frac{1}{2}\int_{\mathcal{B}_{r,n}\setminus B_\delta} (\nabla \tilde{f}\cdot \beta)U_n^2\, dx -\frac{1}{2}\int_{\mathcal{B}_{r,n}\setminus B_\delta} (\mathrm{div}\beta) \tilde{f}U_n^2\, dx. 
\end{split}
\end{equation} 
In particular, we have that as $\delta\searrow 0$
\begin{equation}\label{pezzinuovi}
\begin{split}
\int_{\mathcal{B}_{r,n}\setminus B_\delta}(\nabla \tilde{f}\cdot\beta)U_n^2\,dx &= \int_{B_r\setminus B_\delta}(\nabla \tilde{f}\cdot\beta)U_n^2\,dx\to \int_{B_r} (\nabla \tilde{f}\cdot\beta)U_n^2\,dx,\\
\int_{\mathcal{B}_{r,n}\setminus B_\delta} (\mathrm{div}\beta)\tilde{f}U_n^2\,dx &= \int_{B_r}(\mathrm{div}\beta)\tilde{f}U_n^2\,dx,
\end{split}
\end{equation}
by the absolute continuity of the Lebesgue integral, since $(\nabla \tilde{f}\cdot\beta)U_n^2$ and $(\mathrm{div}\beta)\tilde{f}U_n^2$ are both in $L^1(B_r)$, as a consequence of the H\"{o}lder inequality, the fact that $\tilde{f}\in W^{1,p}(B_r)$, \eqref{betaequasix} and \eqref{divbetalimitata}. 
Thus, plugging \eqref{div} into \eqref{ineq}, by \eqref{sommavaa0}, \eqref{imp}, \eqref{nuova} and \eqref{pezzinuovi}, passing to the limit as $h\to\infty$ with $\delta=\delta_h$, we get that for all $r\in (0,r_0)$ and $n\geq \bar{n}$
\begin{equation}\label{ineq3}
\begin{split}
r\int_{\partial B_r}(A\nabla U_n\cdot\nabla U_n)\,ds&- 2r\int_{\partial B_r} \frac{1}{\mu}( A\nabla U_n \cdot \nu)^2\, ds  \geq   \int_{B_r} (\mathrm{div}\beta) A\nabla U_n\cdot \nabla U_n \, dx \\
& - 2\int_{B_r} \mathrm{Jac}\beta(A\nabla U_n)\cdot \nabla U_n\, dx +\int_{B_r} (dA\nabla U_n\nabla U_n)\cdot \beta\,dx\\ 
& +r\int_{\partial B_r} \tilde{f} U_n^2\, ds-\int_{B_r} (\tilde{f}\mathrm{div}\beta + \nabla \tilde{f}\cdot\beta)U_n^2\,dx,
\end{split}
\end{equation}
under assumption \eqref{asssuftilde2}.

\emph{Step two.} In this step we will take the limit as $n\to \infty$ in \eqref{ineq2} and \eqref{ineq3}. 
We start by noting that, since $A$ is symmetric, we have the following identity 
\begin{equation*}\label{aggiungoesottraggo}
A\nabla U_n\cdot \nabla U_n -A\nabla U\cdot \nabla U= A\nabla U_n\cdot \nabla (U_n-U)+A\nabla U\cdot\nabla (U_n-U).
\end{equation*} 
By this, \eqref{divbetalimitata}, and \eqref{normadiA}, applying H\"{o}lder's inequality,
using that $U_n\to U$ in $H^1(B_{r})$ thanks to Proposition \ref{convdelleun}, and consequently that $\{\nabla U_n\}$ is bounded in $L^2(B_r)$, we thus have that
\begin{equation}\label{primopezzo}
\begin{split}
\int_{B_r}&|\mathrm{div}\beta||A\nabla U_n\cdot \nabla U_n -A\nabla U\cdot \nabla U|\, dx\\
\leq& \int_{B_r}|\mathrm{div}\beta||A\nabla U_n|| \nabla (U_n-U) |\,dx+\int_{B_r}|\mathrm{div}\beta|
|A\nabla U|| \nabla (U_n-U) |\, dx\\
\leq&\, 2(N+2)\left(\int_{B_r} |\nabla U_n||\nabla (U_n-U) |\,dx+ \int_{B_r} |\nabla U||\nabla (U_n-U) |\,dx\right)\\
\leq&\, 2(N+2) \left(\int_{B_r} |\nabla (U_n-U) |^2\, dx\right)^{1/2}\left[\left(\int_{B_r} |\nabla U_n |^2\, dz\right)^{1/2}+\left(\int_{B_r} |\nabla U |^2\, dz\right)^{1/2}\right]\to 0
\end{split}
\end{equation}
as $n\to \infty$. 
Similarly to the above identity, we can write 
\begin{equation*}
\mathrm{Jac}\beta (A\nabla U_n) \cdot \nabla U_n- \mathrm{Jac}\beta (A\nabla U) \cdot \nabla U = \mathrm{Jac}\beta (A \nabla(U_n-U)) \cdot \nabla U_n+ \mathrm{Jac}\beta (A\nabla U)\cdot \nabla(U_n-U), 
\end{equation*}
obtained by adding $\pm \mathrm{Jac}\beta(A\nabla U)\cdot \nabla U_n$ to the left-hand side. 
Using this, \eqref{jaclimitata}, \eqref{normadiA}, and arguing as when inferring \eqref{primopezzo}, we can deduce that for some $\mathrm{const}>0$ independent of $r$ 
\begin{equation}\label{jacconv}
\begin{split}
\int_{B_r}& |\mathrm{Jac}\beta(A\nabla U_n) \cdot \nabla U_n- \mathrm{Jac}\beta (A\nabla U) \cdot \nabla U|\, dx\\
&\leq \int_{B_r} \Vert\mathrm{Jac}\beta\Vert |A\nabla (U_n-U)||\nabla U_n|\, dx + \int_{B_r}  \Vert\mathrm{Jac}\beta\Vert |A\nabla U||\nabla (U_n-U)|\,dx\\
&\leq \mathrm{const}\left(\int_{B_r}|\nabla (U_n-U)||\nabla U_n| \,dx + \int_{B_r} |\nabla U||\nabla (U_n-U)|\, dx\right)\to 0
\end{split} 
\end{equation}
as $n\to\infty$. In analogous way, we can write 
\begin{equation*}
dA\nabla U_n\nabla U_n  - dA\nabla U\nabla U = dA\nabla U_n\nabla (U_n-U) + dA\nabla U\nabla (U_n-U),
\end{equation*}
obtained by adding $\pm dA\nabla U\nabla U_n$ to the left-hand side, and applying \eqref{simmetriadA} and \eqref{differenza}.  
Exploiting the last identity, \eqref{betaequasix}, \eqref{leproprietadidA}, and arguing as when inferring \eqref{primopezzo}, we have that for some $\mathrm{const}>0$ independent of $r$ 
\begin{equation}\label{dAconv}
\begin{split}
\int_{B_r} |dA \nabla U_n\nabla U_n&\cdot\beta- dA \nabla U\nabla U\cdot\beta|\, dx\leq  \int_{B_r} |dA \nabla U_n\nabla U_n- dA \nabla U\nabla U||\beta|\,dx\\
\leq &\,\mathrm{const}\,r\left(\int_{B_r}|dA\nabla U_n\nabla (U_n-U)|\, dx +\int_{B_r}|dA\nabla U\nabla (U_n-U)|\, dx\right)\\
\leq &\,\mathrm{const}\,r \left(\int_{B_r}|\nabla U_n||\nabla (U_n-U)|\, dx+ \int_{B_r} |\nabla U||\nabla (U_n-U)|\, dx\right)\to 0
\end{split}
\end{equation}
as $n\to\infty$.  
To conclude the convergence as $n\to \infty$  of the right-hand side of \eqref{ineq2}, we observe that 
\begin{equation*}
(\beta\cdot \nabla U_n)\tilde{f} U_n - (\beta\cdot\nabla U) \tilde{f}U = (\beta\cdot \nabla U_n)\tilde{f}(U_n-U)+ \tilde{f}U(\beta\cdot\nabla (U_n-U)) 
\end{equation*}
obtained by adding $\pm (\beta\cdot \nabla U_n) \tilde{f}U$ to the left-hand side; therefore, by this and \eqref{betaequasix}, applying H\"{o}lder's inequality and \eqref{hardy}, we have that for some $\mathrm{const}>0$ independent of $r$
\begin{equation}\label{riscritta}
\begin{split}
\int_{B_r}|(\beta\cdot \nabla U_n)\tilde{f} &U_n - (\beta\cdot\nabla U) \tilde{f} U|\,dx\\
\leq\int_{B_r} |(\beta\cdot &\nabla U_n) \tilde{f}(U_n-U)|\,dx +\int_{B_r}  |\tilde{f} U (\beta\cdot \nabla (U_n-U))|\, dx\\
\leq \mathrm{const}\,r^\delta&\left( \int_{B_r}|\nabla U_n|\frac{|U_n-U|}{|x|} \, dx + \int_{B_r} \frac{|U(x)|}{|x|}|\nabla (U_n-U)|\, dx \right)\\
\leq \mathrm{const}\,r^\delta& \biggl(\left[\int_{B_r}|\nabla U_n|^2\, dx\right]^{1/2}\left[\int_{B_r}|\nabla (U_n-U)|^2\,dx + \frac{N-1}{2r}\int_{\partial B_r} |U_n-U|^2\, ds\right]^{1/2}\\
& +\left[\int_{B_r}|\nabla (U_n-U)|^2\, dx\right]^{1/2}\left[\int_{B_r}|\nabla U|^2\, dx + \frac{N-1}{2r}\int_{\partial B_r}U^2\, ds\right]^{1/2}\biggr)\to 0
\end{split}
\end{equation}
as $n\to \infty$, where in addition we have used the convergence $U_n\to U$ in $H^1(B_{r})$ proved in Proposition \ref{convdelleun}, which in turn implies that $\{\nabla U_n\}$ is bounded in $L^2(B_r)$ and that $U_n\to U$ in $L^2(\partial B_r)$ by the continuity of the trace operator in \eqref{traceoperator}.
As for the left-hand side of \eqref{ineq2}, arguing precisely as when deriving \eqref{primopezzo}, we have that 
\begin{equation}\label{servira}
\int_{B_{r_0}}A\nabla U_n\cdot\nabla U_n\, dx\to \int_{B_{r_0}}A\nabla U\cdot\nabla U\, dx\quad \text{as $n\to\infty$}.
\end{equation} 
From this, using the coarea formula, we deduce that up to a subsequence still denoted with $U_n$
\begin{equation}\label{spuntorisultato}
\int_{\partial B_r} A\nabla U_n\cdot \nabla U_n\, ds\to \int_{\partial B_r} A\nabla U\cdot \nabla U\, ds \quad\text{as $n\to\infty$, for a.e. $r\in (0,r_0)$}.
\end{equation}
At last, using \eqref{mumaggiore}, \eqref{normadiA}, the H\"{o}lder inequality, the convergence $U_n\to U$ in $H^1(B_{r_0})$ and consequently the boundedness of $\{\nabla U_n\}$ in $L^2(B_{r_0})$, we have that
\begin{equation*}\label{secondotermdibordo}
\begin{split}
\int_{ B_{r_0}}\frac{1}{\mu}|(A\nabla U_n\cdot \nu)^2 &- (A\nabla U \cdot\nu)^2|\, ds\leq 2 \int_{B_{r_0}} |A\nabla (U_n-U)\cdot\nu||A\nabla (U_n+U)\cdot\nu|\, ds\\
\leq& \,4 \left(\int_{B_{r_0}} |\nabla (U_n-U)|^2\, dx\right)^{1/2}\left(\int_{B_{r_0}} |\nabla (U_n+U)|^2\, dx\right)^{1/2}\ \to 0\\
\end{split}
\end{equation*}
as $n\to \infty$; hence up to a subsequence still denoted with $U_n$, it holds that 
\begin{equation}\label{serve2}
\int_{\partial B_r} \frac{1}{\mu} (A\nabla U_n\cdot \nu)^2 \, ds\to \int_{\partial B_r} \frac{1}{\mu} (A\nabla U\cdot \nu)^2 \, ds\quad \text{as $n\to\infty$, for a.e. $r\in (0,r_0)$}. 
\end{equation}
Putting together \eqref{primopezzo}, \eqref{jacconv}, \eqref{dAconv}, \eqref{riscritta}, \eqref{spuntorisultato}, \eqref{serve2}, and passing to the limit as $n\to\infty$ in \eqref{ineq2}, we have shown the validity of \eqref{poho1} for a.e. $r\in (0,r_0)$, under assumption \eqref{asssuftilde1}. 

On the other hand, under assumption \eqref{asssuftilde2}, using \eqref{divbetalimitata}, and applying the H\"{o}lder inequality and \eqref{stimadiL2star}, we have that for some $\mathrm{const}>0$ independent of $r$
\begin{equation}\label{secondass}
\begin{split}
\int_{B_r}|\tilde{f}\mathrm{div}\beta|&|U_n^2-U^2|\, dx\leq (N+2)\int_{B_r}|\tilde{f}||U_n-U||U_n+U|\, dx\\
\leq (N+2)&\left(\int_{B_r}|\tilde{f}||U_n-U|^2\,dx\right)^{1/2}\left(\int_{B_r}|\tilde{f}||U_n+U|^2\,dx\right)^{1/2}\\
\leq \mathrm{const}&\, r^{\frac{2p-N-1}{p}}\left(\int_{B_r}|\nabla (U_n-U)|^2\,dx+\frac{N-1}{2r}\int_{\partial B_r}|U_n-U|^2\,ds\right)^{1/2}\\
&\cdot \left(\int_{B_r}|\nabla (U_n+U)|^2\,dx+\frac{N-1}{2r}\int_{\partial B_r}|U_n+U|^2\,ds\right)^{1/2}\to 0
\end{split}
\end{equation}
as $n\to\infty$, where in addition we have used the convergence $U_n\to U$ in $H^1(B_{r})$ proved in Proposition \ref{convdelleun}, which in turn, summed to the continuity of the trace operator in \eqref{traceoperator}, implies that $U_n\to U$ in $L^2(\partial B_r)$ and $\{U_n\}$ is bounded in $L^2(\partial B_r)$.
Moreover, still under assumption \eqref{asssuftilde2}, taking advantage of \eqref{betaequasix} and applying again the H\"{o}lder inequality, we get that for some $\mathrm{const}>0$ independent of $r$
\begin{equation}\label{cisiamoquasi}
\begin{split}
\int_{B_r} |\nabla \tilde{f}\cdot \beta||U_n^2-U^2|\, dx \leq \mathrm{const}\, r \int_{B_r}|\nabla \tilde{f}||U_n-U||U_n+U|\,dx\to 0
\end{split}
\end{equation}
as $n\to \infty$, reasoning as in \eqref{secondass}. 
Eventually, under assumption \eqref{asssuftilde2}, using that $\tilde{f}\in L^\infty_\mathrm{loc}(B_{r_0}\setminus \{0\})$, it also holds that for some $\mathrm{const}>0$ independent of $r$
\begin{equation}\label{endup2}
\int_{\partial B_r} |\tilde{f}||U_n^2-U^2|\, ds\leq \mathrm{const} \int_{\partial B_r}|U_n-U||U_n+U|\to 0\quad \text{as $n\to\infty$},
\end{equation}
applying in addition the H\"{o}lder inequality, exploiting the convergence $U_n\to U$ in $L^2(\partial B_r)$ and the boundedness of $\{U_n\}$ in $L^2(\partial B_r)$, both as a consequence of Proposition \ref{convdelleun} and the continuity of the trace map in \eqref{traceoperator}. Assembling \eqref{primopezzo}, \eqref{jacconv}, \eqref{dAconv}, \eqref{spuntorisultato}, \eqref{serve2}, \eqref{secondass}, \eqref{cisiamoquasi} and \eqref{endup2}, and passing to the limit as $n\to\infty$ in \eqref{ineq3}, we have shown the validity of \eqref{poho} for a.e. $r\in (0,r_0)$, under assumption \eqref{asssuftilde2}.  
\end{proof}
\begin{lem}
Let $f$ satisfy either \eqref{asssuftilde1} or \eqref{asssuftilde2}. Let $U\in H^1_{\tilde{\Gamma}}(B_{\tilde{r}})$ be a weak solution to problem \eqref{problU} and let $\nu=\nu(x)=x/|x|$ for every $x\in \partial B_r$. Then, for a.e. $r\in(0, r_0)$ 
\begin{equation}\label{utileperdopo}
\int_{B_r} A \nabla U\cdot\nabla U\, dx -\int_{B_r} \tilde{f}U^2\, dx= \int_{\partial B_r} (A\nabla U\cdot \nu)U\, ds.
\end{equation} 
\end{lem}
\begin{proof}
%Since $U_n\in H^2(\mathcal{B}_{r,n}\setminus B_\delta)$ for every  $r\in (0,r_0)$, $n\geq \bar{n}$ and $\delta<1/n^{2\alpha}$ (as observed at the beginning of the proof of Proposition \ref{propopoho}), then we are allowed to  
We multiply equation \eqref{indicizzato} with $U_n$ itself and then we integrate over $\mathcal{B}_{r,n}\setminus B_\delta$. Taking into account that 
$U_n=G_n=0$ on $\gamma_{r,n}$, and consequently $U_n$ can be extended to zero in $B_r\setminus \mathcal{B}_{r,n}$, an integration by parts leas to the following identity 
\begin{equation}\label{dapassareallimite2}
\begin{split}
\int_{B_r\setminus B_\delta} A\nabla U_n\cdot\nabla U_n\,dx&-\int_{B_r\setminus B_\delta} \tilde{f}U_n^2\,dx\\
&=\int_{\partial B_r}(A\nabla U_n\cdot \nu)U_n\,ds  
-\int_{\partial B_\delta} (A\nabla U_n\cdot \nu)U_n\,ds. 
\end{split}
\end{equation}
Using \eqref{normadiA}, we can easily deduce that $A\nabla U_n\cdot \nabla U_n\in L^1(B_r)$; thus from the absolute continuity of the Lebesgue intergal, it follows that 
\begin{equation}\label{cheso1}
\int_{B_r\setminus B_\delta}A\nabla U_n\cdot \nabla U_n \, dx\to \int_{B_r}A\nabla U_n\cdot \nabla U_n \, dx \quad \text{as $\delta\searrow 0$}.
\end{equation}
Furthermore, we have that $\tilde{f}U_n^2\in L^1(B_r)$ either by \eqref{hardy} under assumption \eqref{asssuftilde1}, or by the H\"{o}lder inequality and \eqref{stimadiL2star} under assumption \eqref{asssuftilde2}; so, again by the absolute continuity of the Lebesgue integral, we obtain that 
\begin{equation}\label{cheso2}
\int_{B_r\setminus B_\delta} \tilde{f} U_n^2\, dx \to \int_{B_r}\tilde{f} U_n^2\, dx\quad \text{as $\delta\searrow 0$.}
\end{equation}
As for the limit as $\delta\searrow 0$ of the last term in \eqref{dapassareallimite2}, it is sufficient noting that $(A\nabla U_n\cdot\nu)U_n\in L^1(B_{r})$ thanks to \eqref{normadiA}; as a result, by the Lebesgue dominated convergence theorem, it holds that 
\begin{equation}\label{partialvaa0}
\int_{\partial B_{\delta}}(A\nabla U_n\cdot\nu)U_n\, ds= \int_{B_{r}} \chi_{\partial B_\delta}(x)(A\nabla U_n\cdot\nu)U_n\, dx \to 0 \quad \text{as $\delta\searrow 0$}. 
\end{equation}
Summing \eqref{cheso1}, \eqref{cheso2} and \eqref{partialvaa0}, and passing to the limit as $\delta\searrow 0$ in \eqref{dapassareallimite2}, we have that for all $r\in (0,r_0)$ and $n\geq \bar{n}$
\begin{equation}\label{mmm}
\int_{B_r}A\nabla U_n\cdot \nabla U_n \, dx - \int_{B_r}\tilde{f} U_n^2\, dx = \int_{\partial B_r} (A\nabla U_n\cdot\nu)U_n\, ds. 
\end{equation}
To conclude with, we have to pass to the limit as $n\to \infty$ in the previous identity. In particular, thanks to the H\"{o}lder inequality we have that
\begin{equation}\label{ftildeunquadro}
\begin{split}
\int_{B_r}& |\tilde{f}U_n^2-\tilde{f}U^2|\, dx = \int_{B_r} |\tilde{f}||U_n-U||U_n+U|\, dx \\
\leq &\left(\int_{B_r}|\tilde{f}||U_n-U|^2\, dx\right)^{1/2}\left(\int_{B_r}|\tilde{f}||U_n+U|^2\, dx\right)^{1/2}\to 0 \quad\text{as $n\to\infty$},\\
%\leq&\, r^\delta\left(\int_{B_r}|\nabla (U_n-U)|^2\,dx + \frac{N-1}{2r}\int_{\partial B_r} |U_n-U|^2\, ds\right)^{1/2}\\
%& \cdot \left(\int_{B_r}|\nabla (U_n+U)|^2\,dx + \frac{N-1}{2r}\int_{\partial B_r} |U_n+U|^2\, ds\right)^{1/2}\to 0\\ 
\end{split}
\end{equation}
if in addition either we use \eqref{hardy} if $\tilde{f}$ satisfies assumption \eqref{asssuftilde1}, or we exploit once more the H\"{o}lder inequality and \eqref{stimadiL2star} if $\tilde{f}$ satisfies assumption \eqref{asssuftilde2}. 
%where we have also exploited the convergence $U_n\to U$ in $H^1(B_{r})$ proved in Proposition \ref{convdelleun}, which in turn, put together with the continuity of the trace operator in \eqref{traceoperator}, implies that $U_n\to U$ in $L^2(\partial B_r)$ and that $\{U_n\}$ is bounded in $L^2(\partial B_r)$. Instead, under assumption \eqref{asssuftilde2},
Moreover, using \eqref{normadiA} and the H\"{o}lder inequality, we also infer that 
\begin{equation*}
\begin{split}
\int_{B_r}|&(A\nabla U_n\cdot\nu)U_n-(A\nabla U\cdot\nu)U|\, dx\\
\leq &  \int_{B_r} |(A\nabla (U_n-U)\cdot\nu)U_n|\, dx + \int_{B_r} | (A\nabla U\cdot\nu)(U_n-U)|\, dx\\
\leq & \,2 \biggl[\left(\int_{B_r}|U_n|^2\, dx\right)^{1/2}\left(\int_{B_r} |\nabla (U_n-U)|^2\, dx\right)^{1/2}\\
&+ \left(\int_{B_r} |\nabla U|^2\, dx\right)^{1/2}\left(\int_{B_r}|U_n-U|^2\,dx\right)^{1/2}\biggr]\to 0\quad \text{as $n\to \infty$.}
\end{split}
\end{equation*}
Therefore by the coarea formula, we can assert that up to a subsequence still denoted with $U_n$ 
\begin{equation}\label{ultimaaa}
\int_{\partial B_r} (A\nabla U_n\cdot\nu)U_n\, ds\to \int_{\partial B_r} (A\nabla U\cdot\nu)U\, ds\quad \text{as $n\to\infty$, for a.e. $r\in (0,r_0)$}.
\end{equation} 
Hence combining \eqref{servira} with \eqref{ftildeunquadro} and \eqref{ultimaaa}, and passing to the limit as $n\to \infty $ in \eqref{mmm},  we finally get \eqref{utileperdopo} for a.e. $r\in (0,r_0)$.
\end{proof}
\section{The monotonicity formula}\label{section3}
In this section we prove a monotonocity formula for the Almgren frequency function associated with solutions of problem \eqref{problU}. 
To such aim, we fix a non-trivial weak solution $U\in H^1_{\tilde{\Gamma}}(B_{\tilde{r}})$ to problem \eqref{problU} 
and we define the energy function as 
\begin{equation}\label{D}
D\colon r\in (0,\tilde{r})\mapsto D(r):=\frac{1}{r^{N-1}}\left(\int_{B_r} A\nabla U\cdot\nabla U\, dx - \int_{B_r} \tilde{f} U^2\, dx\right),
\end{equation}
and the height function as
\begin{equation}\label{H}
H: r\in (0,\tilde{r})\mapsto H(r):=\frac{1}{r^N} \int_{\partial B_r} \mu U^2\, ds.
\end{equation}
In the following lemma we show that the function $H(r)$ is strictly positive if $r$ is sufficiently small. 
\begin{lem}\label{lapositivitadiH}
Let $r_0>0$ be as in Lemma \ref{dis}. Then the function $H$ defined in \eqref{H} is strictly positive in $(0,r_0]$.
\end{lem}
\begin{proof}
We assume by contradiction that $H(r_1)=0$ for some $r_1\in (0,r_0]$. Hence, taking into account \eqref{mumaggiore}, we can deduce that the trace of $U$ is identically zero on $\partial B_{r_1}$. So, taking $\varphi=U$ in \eqref{weaksolu}, we have that 
\begin{equation*}
\int _{B_{r_1}} A\nabla U\cdot \nabla U\, dx= \int _{B_{r_1}}  \tilde{f} U^2 \, dx.
\end{equation*} 
From this and Lemma \ref{dis}, we can conclude that necessarily $U\equiv 0$ in $B_{r_1}$, which in turn implies that $U\equiv 0$ in $B_{\tilde{r}}$ as a consequence of unique continuation principles for second order elliptic equations with Lipschitz coefficients (see e.g. \cite{garofalo}). In particular this contradicts the fact that $U\not\equiv 0$ by assumption. 
\end{proof}

We are now in the position to define Almgren's frequency function as follows
\begin{equation}\label{N}
\mathcal{N}: r\in (0,r_0]\mapsto \mathcal{N}(r):=\frac{D(r)}{H(r)}.
\end{equation}
We premise a simple result which will be crucial to estimate $D'$. 
\begin{lem}Every $U\in H^1(B_r)$ with $r\in (0,r_0]$ satisfies the following inequality
\begin{equation}\label{utile1}
\int_{B_r} |\nabla U|^2\, dz \leq 4r^{N-1} \left(D(r)+\frac{N-1}{4}  H(r)\right).
\end{equation}
\end{lem}
\begin{proof}
Taking advantage of \eqref{unificata} and recalling the definitions of $D$ and $H$ given in \eqref{D} and \eqref{H} respectively, we have that 
\begin{equation*}\label{utile}
\int_{B_r} |\nabla U|^2\, dz \leq 4r^{N-1} \left(D(r)+C r^\varepsilon H(r)\right),
\end{equation*}
for every $U\in H^1(B_r)$ with $r\in (0,r_0]$.
Thus \eqref{utile1} is a trivial consequence of this and \eqref{serviraancora}. 
\end{proof}
In the following lemma we show that $\mathcal{N}$ is bounded from below. 
\begin{lem} 
We have that
\begin{equation}\label{Nmaggiore}
\mathcal{N}(r)>-\frac{N-1}{4}\quad \text{for every $r\in (0,r_0]$}.
\end{equation}
\end{lem}
\begin{proof}
By \eqref{unificata}, \eqref{D} and \eqref{H}, we have that 
\begin{equation*}
r^{N-1} D(r) \geq - Cr^{N-1+\varepsilon} H(r)\quad \text{for every $r\in (0,r_0]$}. 
\end{equation*}
Hence, recalling the definition of $\mathcal{N}$ given in \eqref{N}, we deduce that 
\begin{equation}\label{perdirecheillimiteepositivo}
\mathcal{N}(r)\geq - Cr^{\varepsilon}\quad \text{for every $r\in (0,r_0]$},
\end{equation}
which gives \eqref{Nmaggiore}, if in addition we take into account \eqref{serviraancora}.
\end{proof}
In the next lemma we exhibit an estimate from below for the derivative of $D$: to derive such a result, the Pohozaev-type inequality will turn out to be fundamental. 
\begin{lem}
Let $f$ satisfy either \eqref{asssuftilde1} or \eqref{asssuftilde2}. It holds that 
\begin{equation}\label{regolD}
D\in W^{1,1}_\mathrm{loc}(0,\tilde{r})
\end{equation}
and for a.e. $r\in (0,r_0)$
\begin{equation}\label{D'}
D' (r)\geq \frac{2}{r^{N-1}} \int_{\partial B_r} \frac{1}{\mu}(A\nabla U\cdot\nu)^2\, ds + O(r^{\bar{\varepsilon}})\left(D(r)+\frac{N-1}{2}H(r)\right)\quad \text{as $r\to 0^+$},
\end{equation}
where $\nu=\nu(x):=x/|x|$ for every $x\in \partial B_r$ and
\begin{equation}\label{eps}
\bar{\varepsilon} = \begin{cases}
0&\text{if }-1+\varepsilon\geq 0,\\
-1+\varepsilon&\text{if } -1+\varepsilon\leq 0,
\end{cases}
\end{equation}
being $\varepsilon>0$ as in \eqref{esponenteps}.
\end{lem}
\begin{proof}
\eqref{regolD} holds since $D$ is the product of the $W^{1,\infty}_{\mathrm{loc}}(0,\tilde{r})$-function $r\mapsto r^{1-N}$ and 
\begin{equation*}
r\mapsto\int_{ B_r} A\nabla U\cdot\nabla U \, dx- \int_{B_r} \tilde{f} U^2\, dx
\end{equation*}
which is in $W^{1,1}(0,\tilde{r})$; this last fact can be deduced by using the coarea formula, the fact that $U\in H^1(B_{\tilde{r}})$, \eqref{normadiA}, and additionally either \eqref{hardy} under assumption \eqref{asssuftilde1}, or the H\"{o}lder inequality and \eqref{stimadiL2star} under assumption \eqref{asssuftilde2}.  
Moreover the coarea formula also allows us to infer that
\begin{equation*}
\begin{split}
D'(r)=& \frac{1}{r^{N-1}} \left(\int_{\partial B_r} A\nabla U\cdot\nabla U \, ds -\int_{\partial B_r} \tilde{f} U^2\, ds\right)+\frac{1-N}{r^{N}}\left(\int_{ B_r} A\nabla U\cdot\nabla U \, dx -\int_{ B_r} \tilde{f} U^2\, dx\right) \\
\end{split}
\end{equation*}
in the sense of distributions and a.e. in $(0,\tilde{r})$.
Then we use \eqref{poho1} and \eqref{poho} under assumption \eqref{asssuftilde1} and \eqref{asssuftilde2} respectively to estimate from below the term $\int_{\partial B_r} A\nabla U\cdot \nabla U\, ds$, obtaining that for a.e. $r\in (0,r_0)$
\begin{equation}\label{D'primordiale}
\begin{split}
D'(r)\geq & \,\frac{2}{r^{N-1}} \int_{\partial B_r} \frac{1}{\mu}(A\nabla U\cdot\nu)^2\, ds +\frac{1}{r^N} \int_{B_r} (\mathrm{div}\beta) A\nabla U\cdot \nabla U\, dx -\frac{2}{r^N}\int_{B_r}\mathrm{Jac}\beta (A\nabla U)\cdot \nabla U\, dx \\
& +\frac{1}{r^N} \int_{B_r} (dA\nabla U\nabla U)\cdot\beta\, dx +\frac{1-N}{r^{N}}\int_{ B_r} A\nabla U\cdot\nabla U \, dx-\frac{1}{r^{N-1}} \int_{\partial B_r} \tilde{f} U^2\, ds\\
&+ \frac{2}{r^N} \int_{B_r} (\beta \cdot\nabla U) \tilde{f} U\, dx +\frac{N-1}{r^{N}}\int_{ B_r} \tilde{f} U^2\, dx,
\end{split}
\end{equation}
under assumption \eqref{asssuftilde1}, and
\begin{equation}\label{D'primordiale2}
\begin{split}
D'(r)\geq & \,\frac{2}{r^{N-1}} \int_{\partial B_r} \frac{1}{\mu}(A\nabla U\cdot\nu)^2\, ds +\frac{1}{r^N} \int_{B_r} (\mathrm{div}\beta) A\nabla U\cdot \nabla U\, dx \\
&-\frac{2}{r^N}\int_{B_r}\mathrm{Jac}\beta (A\nabla U)\cdot \nabla U\, dx +\frac{1}{r^N} \int_{B_r} (dA\nabla U\nabla U)\cdot\beta\, dx \\
&+\frac{1-N}{r^{N}}\int_{ B_r} A\nabla U\cdot\nabla U \, dx- \frac{1}{r^{N}}\int_{B_r}(\tilde{f} \mathrm{div}\beta + \nabla \tilde{f}\cdot\beta)U^2\,dx+\frac{N-1}{r^{N}}\int_{ B_r} \tilde{f} U^2\, dx,
\end{split}
\end{equation}
under assumption \eqref{asssuftilde2}. 
First of all we observe that 
\begin{equation}\label{zero}
\begin{split}
\frac{1}{r^N} \int_{B_r} (\mathrm{div}\beta) A\nabla U\cdot \nabla U\, dx -\frac{2}{r^N}&\int_{B_r}\mathrm{Jac}\beta (A\nabla U)\cdot \nabla U\, dx +\frac{1}{r^N} \int_{B_r} (dA\nabla U\nabla U)\cdot\beta\, dx  \\
+\frac{1-N}{r^N} \int_{B_r} A\nabla U\cdot\nabla U\, dx=\,& \frac{O(r)}{r^N}\int_{B_r}|\nabla U|^2\, dx \\
=\, & O(1) \left(D(r)+\frac{N-1}{4}H(r)\right) \quad \text{as $r\to 0^+$},
\end{split}
\end{equation}
as a direct consequence of \eqref{stimadiA}, \eqref{lestimedibeta} and \eqref{leproprietadidA}, and thanks to \eqref{utile1}.
%and moreover 
%\begin{equation*}
%\frac{O(r)}{r^N}\int_{B_r}|\nabla u|^2\, dz =O(1) \left(D(r)+\frac{N-1}{4}H(r)\right) \quad \text{as $r\to 0$}
%\end{equation*} 
Now we specifically focus on \eqref{D'primordiale}. For this, under assumption \eqref{asssuftilde1} we use \eqref{hardy}, \eqref{mumaggiore} and \eqref{utile1} to deduce that 
\begin{equation}\label{one}
\begin{split}
\left|\int_{B_r} \tilde{f} U^2\, dx\right| &\leq \mathrm{const}\,r^\delta \left[\left(\frac{2}{N-1}\right)^2 \int_{B_r} |\nabla U|^2\, dx + \frac{4}{N-1}r^{-1} \int_{\partial B_r}  \mu U^2\, ds \right]\\
&\leq \mathrm{const}\, r^{\delta+N-1}\left(D(r)+\frac{N-1}{2}H(r)\right),
\end{split}
\end{equation}
for some $\mathrm{const}>0$ independent of $r$; moreover by \eqref{mumaggiore} we also have that 
\begin{equation}\label{two}
\left|\int_{\partial B_r} \tilde{f} U^2\, ds\right| \leq \mathrm{const}\,r^{-2+\delta} \int_{\partial B_r} \mu U^2\, ds \leq \mathrm{const}\,r^{-2+\delta+N} H(r),
\end{equation}
for some $\mathrm{const}>0$ independent of $r$;
lastly, taking into consideration \eqref{betaequasix}, applying the H\"{o}lder inequality to $|x|^{-1}|U|$ and $|\nabla U|$, and using \eqref{hardy}, \eqref{mumaggiore} and \eqref{utile1}, we obtain that
\begin{equation}\label{three}
\begin{split}
\left|\int_{B_r} (\beta\cdot \nabla U) \tilde{f} U \,dx\right|&\leq \mathrm{const}\,r^\delta\int_{B_r}|x|^{-1} |U| |\nabla U|\, dx\\
&\leq \mathrm{const} \, r^\delta \sqrt{r^{N-1}\left(D(r)+\frac{N-1}{2}H(r)\right)}\cdot  \sqrt{r^{N-1}\left(D(r)+\frac{N-1}{4}H(r)\right)}\\
&\leq \mathrm{const} \, r^{-1+\delta+N}\left(D(r)+\frac{N-1}{2}H(r)\right),
\end{split}
\end{equation}
for some $\mathrm{const}>0$ independent of $r$. 
Putting together \eqref{D'primordiale}, \eqref{zero}, \eqref{one}, \eqref{two} and \eqref{three}, we can conclude that a.e. in $(0,r_0)$ under assumption \eqref{asssuftilde1}
\begin{equation*}
\begin{split}
D'(r)\geq & \,\frac{2}{r^{N-1}} \int_{\partial B_r} \frac{1}{\mu}(A\nabla U\cdot\nu)^2\, ds + O(1) \left(D(r)+\frac{N-1}{4}H(r)\right)\\
& + O(r^{-1+\delta}) \left(D(r)+\frac{N-1}{2}H(r)\right) \quad \text{as $r\to 0^+$}. 
\end{split}
\end{equation*}
From this \eqref{D'} immediately follows under assumption \eqref{asssuftilde1}. In order to conclude the proof, we turn to concentrate on \eqref{D'primordiale2}. In particular, under assumption \eqref{asssuftilde2}, by the H\"{o}lder inequality, \eqref{stimadiL2star}, \eqref{mumaggiore} and ultimately \eqref{utile1}, we have that   
\begin{equation}\label{ultimately}
\begin{split}
\left|\int_{B_r}\tilde{f}U^2\,dx\right|\leq&\, \mathrm{const} \,r^{\frac{2p-N-1}{p}}\left(\int_{B_r}|\nabla U|^2\,dx+\frac{N-1}{r}\int_{\partial B_r} \mu U^2\,ds\right)\\
\leq&\, \mathrm{const} \,r^{\frac{2p-N-1}{p}+N-1}\left( D(r)+\frac{N-1}{2}H(r)\right),
\end{split}
\end{equation}
for some $\mathrm{const}>0$ independent of $r$; furthermore, from \eqref{lestimedibeta}, \eqref{betaequasix}, the H\"{o}lder inequality, \eqref{stimadiL2star}, \eqref{mumaggiore} and  \eqref{utile1}, we arrive at
\begin{equation}\label{ultimately2}
\begin{split}
\left|\int_{B_r} (\tilde{f} \mathrm{div}\beta + \nabla \tilde{f} \cdot\beta) U^2\,dx\right|\leq &\, \mathrm{const} \,r^{\frac{2p-N-1}{p}}\left(\int_{B_r}|\nabla U|^2\,dx+\frac{N-1}{r}\int_{\partial B_r} \mu U^2\,ds\right)\\
\leq&\, \mathrm{const} \,r^{\frac{2p-N-1}{p}+N-1}\left( D(r)+\frac{N-1}{2}H(r)\right),  
\end{split}
\end{equation}
for some $\mathrm{const}>0$ independent of $r$. Combining \eqref{D'primordiale2}, \eqref{zero}, \eqref{ultimately} and \eqref{ultimately2}, we can deduce that a.e. in $(0,r_0)$ under assumption \eqref{asssuftilde2}  
\begin{equation*}
\begin{split}
D'(r)\geq & \,\frac{2}{r^{N-1}} \int_{\partial B_r} \frac{1}{\mu}(A\nabla U\cdot\nu)^2\, ds + O(1) \left(D(r)+\frac{N-1}{4}H(r)\right)\\
& + O(r^{-1+\frac{2p-N-1}{p}}) \left(D(r)+\frac{N-1}{2}H(r)\right) \quad \text{as $r\to 0^+$}, 
\end{split}
\end{equation*}
which implies \eqref{D'} under assumption \eqref{asssuftilde2}.
\end{proof}
In the following lemma we compute the derivative of $H$: differently to the result in \cite[Lemma 4.1]{DelFel}, a perturbing term which behaves as $H$ itself appears. 
\begin{lem}
We have that 
\begin{equation}\label{regolaritadiH}
H\in W^{1,1}_{\mathrm{loc}}(0,\tilde{r})
\end{equation}
and 
\begin{equation}\label{H'}
H'(r)= \frac{2}{r^N}\int_{\partial B_r} \mu U \frac{\partial U}{\partial\nu}\, ds + O(H(r))\quad \text{as $r\to 0^+$},
\end{equation}
in a distributional sense and a.e. in $(0,\tilde{r})$, where $\nu=\nu(x):=x/|x|$ for every $x\in \partial B_r$. 
\end{lem}
\begin{proof}
We have that $H\in L^1_{\mathrm{loc}}(0,\tilde{r})$ since it turns out to be the product between the $L^\infty_ {\mathrm{loc}}(0,\tilde{r})$-function $r\mapsto r^{-N}$ and 
\begin{equation*}\label{funzione2}
r\mapsto \int_{\partial B_r} \mu U^2\, ds,
\end{equation*}
which belongs to $L^1(0,\tilde{r})$. 
To show this, it is sufficient to use that $U\in H^1(B_{\tilde{r}})$ and that, 
\begin{equation}\label{muminore}
\mu(x)\leq\mathrm{const} \quad \text{for every $x\in  B_{\tilde{r}}$},
\end{equation}
for some $\mathrm{const} >1$ independent on $r$, as a consequence of \eqref{stimadimu}.
As for the derivative of $H$, we can reason precisely as in the proof of \cite[Lemma 3.1]{DelFelVit} (keeping in mind that here $s=1/2$ and hence the weight $t^{1-2s}$ does not occur) to deduce that the distributional derivative of $H$ is given by 
\begin{equation}\label{distrib}
H'(r)= \frac{2}{r^N}\int_{\partial B_r} \mu U \frac{\partial U}{\partial\nu}\, ds + \frac{1}{r^N}\int_{\partial B_r} (\nabla \mu \cdot\nu) U^2\, ds.
\end{equation}
Exploiting \eqref{muminore} and \eqref{nablamu}, we find out that
the right-hand side of \eqref{distrib} belongs to $L^1_{\mathrm{loc}} (0,\tilde{r})$. Thus \eqref{regolaritadiH} is proved and \eqref{distrib} holds a.e. in $(0,\tilde{r})$ (not only in a distributional sense). Eventually, in order to get \eqref{H'}, it is enough to observe that 
\begin{equation*}
\frac{1}{r^N}\int_{\partial B_r} (\nabla \mu\cdot\nu ) U^2\, ds=O(1)H(r)\quad \text{as $r\to 0$},
\end{equation*} 
by \eqref{nablamu} and \eqref{mumaggiore}, and recalling the defintion of $H$ given in \eqref{H}. 
\end{proof}
In the next lemma we prove an equivalent formulation of the $L^1_{\mathrm{loc}}$-derivative of $H$ and an important identity relating the derivative of $H$ with $D$. 
\begin{lem}
It holds that a.e. in $(0,\tilde{r})$
\begin{equation}\label{H'equiv}
H'(r) = \frac{2}{r^N}\int_{\partial B_r} (A\nabla U\cdot \nu) U\, ds + O(H(r))\quad \text{as $r\to 0$},
\end{equation}
with $\nu=\nu(x):=x/|x|$ for every $x\in\partial B_r$, and a.e. in $(0,r_0)$
\begin{equation}\label{H'D}
H'(r)=\frac{2}{r}D(r)+O(H(r))\quad \text{as $r\to 0$}.
\end{equation}
\end{lem}
\begin{proof}
We immediately notice that once \eqref{H'equiv} is shown, then \eqref{H'D} can be easily derived by using \eqref{utileperdopo} and recalling the definition of $D$ given in \eqref{D}. 
Therefore we have to prove \eqref{H'equiv}: to this aim, we observe that a.e. in $(0,\tilde{r})$
\begin{equation}\label{primaoss}
\begin{split}
\frac{2}{r^N}\int_{\partial B_r} \mu \left( \nabla U\cdot \frac{x}{|x|}\right)U\, ds=& \frac{2}{r^N}\int_{\partial B_r}\mu \left( \nabla U \cdot \frac{\beta}{|x|} \right)U\, ds- \frac{2}{r^N}\int_{\partial B_r} \mu \left( \nabla U \cdot \frac{\beta-x}{|x|}\right) U\, ds\\
=& \frac{2}{r^N}\int_{\partial B_r}\left(A\nabla U \cdot\frac{x}{|x|}\right) U \, ds- \frac{1}{r^N}\int_{\partial B_r} \mu \left(\nabla (U^2) \cdot \frac{\beta-x}{|x|}\right)\, ds,
\end{split}
\end{equation}
where in the first row we have simply added $\pm\mu \left(\nabla U\cdot \frac{\beta}{|x|}\right)U$ to the integrand function and in the second row we have used \eqref{beta} and the fact that $A$ is symmetric. Now, by \eqref{mu} and \eqref{beta} we have that $$\mu \left(\frac{\beta-x}{|x|}\right)\cdot x =0;$$ then applying the divergence theorem before, using the coarea formula after, we deduce that a.e. in $(0,\tilde{r})$
\begin{equation}\label{secondaoss}
0=\int_{\partial B_r}U^2\mathrm{div}\left(\frac{\mu(\beta-x)}{|x|}\right)\, ds + \int_{\partial B_r}\mu \left(\nabla (U^2)\cdot \frac{\beta-x}{|x|}\right)\, ds.
\end{equation}
Now we explicitly compute $\mathrm{div}\left(\frac{\mu(\beta-x)}{|x|}\right)$, having that
\begin{equation*}
\mathrm{div}\left(\frac{\mu(\beta-x)}{|x|}\right)= \nabla \mu \cdot \frac{\beta-x}{|x|}+ \frac{\mu \,\mathrm{div}(\beta-x)}{|x|}- \mu (\beta-x)\cdot \frac{x}{|x|^3},
\end{equation*}
which, together with \eqref{stimadimu}, \eqref{nablamu} and \eqref{lestimedibeta}, yields
\begin{equation}\label{terzaoss}
\mathrm{div}\left(\frac{\mu(\beta-x)}{|x|}\right)= O(1)\quad \text{as $|x|\to 0$}. 
\end{equation}
In conclusion, the thesis follows by combining \eqref{H'}, \eqref{primaoss}, \eqref{secondaoss}, \eqref{terzaoss}, recalling that $\nu= x/|x|$ and taking into account \eqref{mumaggiore}. 
\end{proof} 
Exploiting the previous results we are able to prove the desired estimate from below of the derivative of $\mathcal{N}$, which in turn will allow us to show as byproducts the boundedness from above of $\mathcal{N}$ and the existence of its limit as $r\to 0$.
\begin{lem}\label{lemmaN'}
It holds that 
\begin{equation}\label{regolaritaN}
\mathcal{N}\in W^{1,1}_{\mathrm{loc}}(0,r_0),
\end{equation}
and a.e. in $(0,r_0)$
\begin{equation}\label{N'}
\mathcal{N}'(r)\geq O(r^{\bar{\varepsilon}})\left(\mathcal{N}(r)+\frac{N-1}{2}\right)\quad \text{as $r\to 0^+$},
\end{equation}
where $\bar{\varepsilon}\in (-1,0)$ is defined in \eqref{eps}.
\end{lem}
\begin{proof}
\eqref{regolaritaN} follows by assembling the regularity of $D$ with the regularity and the positivity of $H$ (see \eqref{regolD}, \eqref{regolaritadiH} and Lemma \ref{lapositivitadiH} respectively). In order to derive \eqref{N'} we explicitly compute the distributional derivative of $\mathcal{N}$: exploiting \eqref{D'}, \eqref{H'equiv} and \eqref{H'D}, we have that a.e. in $(0,r_0)$
\begin{equation}\label{stimoN'}
\begin{split}
\mathcal{N}'(r)\geq &\frac{2r\left[\left(\int_{\partial B_r}\frac{(A\nabla U\cdot\nu)^2}{\mu}\, ds\right)\left(\int_{\partial B_r} \mu U^2\, ds\right)-\left(\int_{\partial B_r} U(A\nabla U\cdot\nu)\, ds\right)^2\right]}{\left(\int_{\partial B_r} \mu U^2\, ds\right)^2}\\
& + O(r)+O(r^{\bar{\varepsilon}})\left(\mathcal{N}(r)+\frac{N-1}{2}\right) +\frac{ O(r^2)}{r^{N}H(r)} \int_{\partial B_r} (A\nabla U\cdot\nu)U\, ds\quad \text{as $r\to 0$}.
\end{split}
\end{equation}
Using \eqref{H'equiv}, \eqref{H'D}, and recalling the definition of $\mathcal{N}$ given in \eqref{N},  we infer that  a.e. in $(0,r_0)$
\begin{equation*}\label{miserve}
\frac{1}{r^N H(r)}\int_{\partial B_r}(A\nabla U\cdot \nu) U\, ds= \frac{H'(r)}{2H(r)}+O(1)= \frac{\mathcal{N}(r)}{r}+O(1)\quad \text{as $r\to 0$}.
\end{equation*}
Plugging this into \eqref{stimoN'} and applying the Cauchy-Schwarz inequality to $\frac{A\nabla U\cdot\nu}{\sqrt{\mu}}$ and $\sqrt{\mu}U$ as vectors in $L^2(\partial B_r)$, we arrive at
\begin{equation*}
\mathcal{N}'(r)\geq \left(O(r^{\bar{\varepsilon}})+ O(r)\right) \left(\mathcal{N}(r)+\frac{N-1}{2}\right)\quad \text{as $r\to 0^+$ and a.e. in $(0,r_0)$},  
\end{equation*}
which, taking into account that $\bar{\varepsilon}\in (-1,0)$ by \eqref{eps}, implies \eqref{N'}.
\end{proof}
\begin{lem}
There exists a positive constant $C_4>0$ such that 
\begin{equation}\label{Nlimitatadasopra}
\mathcal{N}(r)\leq C_4 \quad \text{for every $r\in (0,r_0)$}.
\end{equation}
\end{lem}
\begin{proof}
Making use of \eqref{N'}, we have that for a.e. $\rho\in (0,r_0)$
\begin{equation}\label{maggiorediintegrabile}
\mathcal{N}'(\rho)\geq -C_3\, \rho^{\bar{\varepsilon}} \left(\mathcal{N}(\rho)+\frac{N-1}{2}\right) 
\end{equation}
for some $C_3>0$ independent of $r$ (up to take $r_0$ smaller from the beginning, without loss of generality). From this, observing that $$\mathcal{N}'(\rho)= \left(\mathcal{N}(\cdot)+\frac{N-1}{2}\right)'(\rho),$$ taking into consideration \eqref{Nmaggiore} and integrating with respect to $\rho\in(r,r_0) $ for any $r\in (0,r_0)$, we achieve \eqref{Nlimitatadasopra}. 
\end{proof}
\begin{lem}\label{lemmaesistenza}
The limit $\lim_{r\to 0^+}\mathcal{N}(r)$ does exist. Moreover it is finite and nonnegative. 
\end{lem}
\begin{proof}
From \eqref{regolaritaN} it follows that for every $r\in (0, r_0)$
\begin{equation}\label{assolutacont}
\mathcal{N}(r_0)-\mathcal{N}(r)=\int_r^{r_0} \mathcal{N}'(\rho)\,d\rho.
\end{equation}
Now we write $\mathcal{N}'$ as the sum of two terms, as follows
\begin{equation}\label{N'comesomma}
\mathcal{N}'(\rho) = \xi_1(\rho) + \xi_2(\rho)\quad \text{for a.e. $\rho\in (0,r_0)$}, 
\end{equation}
where
\begin{equation}\label{xi1}
\xi_1(\rho):= \mathcal{N}'(\rho)+C_3\, \rho^{\bar{\varepsilon}} \left(C_4+\frac{N-1}{2}\right)\geq 0,
\end{equation}
as a consequence of \eqref{Nlimitatadasopra} and \eqref{maggiorediintegrabile}, and 
\begin{equation}\label{xi2}
\xi_2(\rho):= - C_3\, \rho^{\bar{\varepsilon}} \left(C_4+\frac{N-1}{2}\right) \in L^1(0, r_0),
\end{equation}
since $\bar{\varepsilon}>-1$. Hence, if we insert \eqref{N'comesomma} into \eqref{assolutacont}, passing to the limit as $r\to 0^+$, we can conclude that $\lim_{r\to 0^+}\mathcal{N}(r)$ does exist, in light of \eqref{xi1} and \eqref{xi2} by applying the monotone convergence theorem and Lebesgue's dominated convergence theorem. In particular such a limit is finite in virtue of \eqref{Nmaggiore} and \eqref{Nlimitatadasopra}, and it is positive thanks to \eqref{perdirecheillimiteepositivo}.
\end{proof}
We set 
\begin{equation}\label{ell}
\ell:= \lim_{r\to 0^+}\mathcal{N}(r).
\end{equation}
and by Lemma \ref{lemmaesistenza} we have that $\ell\in \mathbb{R}_+$.

In the rest of this section we prove some results giving information on the growth of $H$.
\begin{lem}
There exists a positive constant $C_5>0$ such that
\begin{equation}\label{Hquasiralla2elle}
H(r)\leq C_5 r^{2\ell}\quad \text{for all $r\in (0,r_0)$},
\end{equation}
and for any $\sigma>0$ there exists a positive constant $C_6=C_6(\sigma)>0$ such that 
\begin{equation}\label{Hmaggioredisigma}
H(r)\geq C_6 r^{2\ell+\sigma}\quad \text{for all $r\in (0,r_0)$}. 
\end{equation}
\end{lem}
\begin{proof}
We first show \eqref{Hquasiralla2elle}. For this, we first notice that by Lemma \ref{lemmaesistenza} for every $\rho\in (0,r_0)$
\begin{equation}\label{N'seconda}
\mathcal{N}(\rho) -\ell= \int_0^\rho\mathcal{N}'(\tau)\, d\tau,
\end{equation}
with $\ell \in \mathbb{R}_+$ as in \eqref{ell}. 
%This can be proven if we observe that by \eqref{regolaritaN}, letting $0<s < \rho <\tilde{r}$
%\begin{equation*}
%\mathcal{N}(\rho)-\mathcal{N}(s)= \int_s^\rho\mathcal{N}'(\tau)\, d\tau;
%\end{equation*}
%in the above equality the left hand side tends to $\mathcal{N}(\rho)-\ell$ as $s\to 0$ by Lemma \ref{lemmaesistenza}, and the right hand side goes to $\int_0^\rho \mathcal{N}'(\tau)\, d\tau$ as $s\to 0$, reasoning precisely as in the proof of Lemma \ref{lemmaesistenza} (namely writing $\mathcal{N}'$ as the sum of a positive function and a $L^1(0,\tilde{r})$-function). 
Moreover from \eqref{xi1} we can deduce that 
\begin{equation*}
\mathcal{N}'(\tau)\geq - C_3 \,\tau^{\bar{\varepsilon}} \left(C_4+\frac{N-1}{2}\right)\quad \text{for a.e. $\tau\in (0,r_0)$}. 
\end{equation*}
Merging this with \eqref{N'seconda}, we obtain that for every $\rho\in (0,r_0)$ 
\begin{equation}\label{Ndirho}
\mathcal{N}(\rho)-\ell \geq -\mathrm{const}\,\rho^{\bar{\varepsilon}+1},
\end{equation}
for some $\mathrm{const}>0$. At this point, we exploit \eqref{H'D} and \eqref{Ndirho} to infer that for a.e. $\rho\in (0,r_0)$
\begin{equation*}
\frac{H'(\rho)}{H(\rho)}\geq \frac{2\ell}{\rho} - \mathrm{const}\, \rho^{\bar{\varepsilon}} \quad \text{as $r\to 0$} 
\end{equation*}
(up to choose $r_0$ smaller from the beginning); then integrating the above inequality over $(r,r_0)$ for all $r\in (0,r_0)$, we get \eqref{Hquasiralla2elle}. 
Now we prove \eqref{Hmaggioredisigma}. To such purpose, we use \eqref{ell} and \eqref{H'D} to claim that for any $\sigma>0$ there exists $r_\sigma>0$ such that for a.e. $\rho\in(0,r_\sigma)$
\begin{equation*}
\frac{H'(\rho)}{H(\rho)} \leq \frac{2\ell+\sigma}{\rho}. 
\end{equation*}
Integrating the above inequality over $(r,r_\sigma)$ with $r\in (0,r_\sigma)$, we deduce that $H(r)\geq\mathrm{const}\, r^{2\ell+\sigma}$ for every $r\in (0,r_\sigma)$ for some $\mathrm{const}>0$. 
The validity of this last inequality for every $r\in[r_\sigma,r_0)$ is a trivial consequence of \eqref{regolaritadiH}. \eqref{Hmaggioredisigma} is thereby proved. 
\end{proof}
\begin{lem}\label{lemmadellimitecheesiste}
The function $H(r)/r^{2\ell}$ admits a finite limit as $r\to 0^+$. 
\end{lem}
\begin{proof}
In view of \eqref{Hquasiralla2elle}, it remains to prove the existence of the limit. To this aim, we compute the derivative of $H(\rho)/\rho^{2\ell}$, obtaining that for a.e. $\rho\in (0,r_0)$
\begin{equation}\label{vaintegrata}
\frac{d}{d\rho} \left(\frac{H(\rho)}{\rho^{2\ell}}\right)= \frac{2H(\rho)}{\rho^{2\ell+1}}\left[ \int_0^\rho \mathcal{N}'(\tau)\, d\tau+O(\rho)\right],
\end{equation}
using \eqref{H'D} and \eqref{N'seconda}. Taking advantage of \eqref{N'comesomma}, \eqref{xi1} and \eqref{xi2}, the right-hand side of \eqref{vaintegrata} becomes 
\begin{equation*}
\frac{2H(\rho)}{\rho^{2\ell+1}}\int_0^\rho \xi_1(\tau)\, d\tau - \frac{2H(\rho)}{\rho^{2\ell}} \left(\frac{C_3}{\bar{\varepsilon}+1}\left(C_4+\frac{N-1}{2}\right) \rho^{\bar{\varepsilon}} + O(1)\right).
\end{equation*}
So, integrating \eqref{vaintegrata} over $(r,r_0)$ for any $r\in (0,r_0)$, we get 
\begin{equation}\label{lasteq}
\frac{H(r_0)}{r_0^{2\ell}}-\frac{H(r)}{r^{2\ell}}= \int_r^{r_0}\left[ \frac{2H(\rho)}{\rho^{2\ell+1}}\int_0^\rho \xi_1(\tau)\, d\tau \right]\,d\rho - \int_r^{r_0} \frac{2H(\rho)}{\rho^{2\ell}}\left(C_7\,\rho^\varepsilon+O(1)\right)\, d\rho,
\end{equation}
with $C_7:=\frac{C_3}{\bar{\varepsilon}+1}\left(C_4+\frac{N-1}{2}\right) >0$.
We now focus on the right-hand side of \eqref{lasteq}: the limit as $r\to 0^+$ of the first term does exist as a consequence of the monotone convergence theorem, taking into account \eqref{xi1}; the limit of the second term does exist as well and is finite, applying the Lebesgue dominated convergence theorem, thanks to \eqref{Hquasiralla2elle} and the fact that $\bar{\varepsilon}>-1$. From these considerations and \eqref{lasteq}, we can deduce the existence of the limit of $H(r)/r^{2\ell}$ as $r\to 0^+$. The lemma is thereby proved. 
\end{proof}

\section{The blow-up argument}\label{section4}
In this section we investigate the convergence properties and then we get information on the limit profile as $\lambda\to 0^+$ of the rescaled and renormalized family of functions $\{U^\lambda\}_{\lambda\in (0,r_0)}$ defined as follows: for any $\lambda\in (0,r_0)$
\begin{equation}\label{Ulambda}
U^\lambda: x\in B_{\frac{r_0}{\lambda}}\mapsto U^\lambda(x):= \frac{U(\lambda x)}{\sqrt{H(\lambda)}},  
\end{equation}
where $U\in H^1_{\tilde{\Gamma}}(B_{\tilde{r}})$ is a fixed non-trivial weak solution to \eqref{problU}.
We notice that the family is well-defined thanks to Lemma \ref{lapositivitadiH}. Moreover the word \emph{renormalized} is justified by 
\begin{equation}\label{int=1}
\int_{\partial B_1} \mu(\lambda\cdot) |U^\lambda(\cdot)|^2 ds=1.
\end{equation}
Furthermore, since $U$ weakly solves problem \eqref{problU}, by direct computations we have that $U^\lambda$ is a weak solution to 
\begin{equation}\label{probulambda}
\left\{\begin{aligned}
-\mathrm{div}(A(\lambda\cdot)\nabla U^\lambda(\cdot))&=\lambda^2\tilde
f(\lambda\cdot)U^\lambda(\cdot) &&\text{in}\ B_{\frac{r_0}{\lambda}}\setminus\tilde\Gamma, \\
U^\lambda&=0 && \text{on}\ \tilde\Gamma,
\end{aligned}\right.
\end{equation}
for every fixed $\lambda\in (0,r_0)$. 
This has to be interpreted in the sense that $U^\lambda$ belongs to the space
\begin{equation}\label{lospaziodiulambda}
H^1_{\tilde{\Gamma}}(B_{\frac{r_0}{\lambda}}):=\overline{C^\infty(\overline{B_{\frac{r_0}{\lambda}}}\setminus \tilde{\Gamma})}^{\Vert \cdot\Vert_{H^1(B_{\frac{r_0}{\lambda}})}},
\end{equation}
namely the closure with respect to the $H^1$-norm of the space of all $C^\infty(\overline{B_{\frac{r_0}{\lambda}}})$-functions vanishing in a neighbourhood of $\tilde{\Gamma}$, and it holds that
\begin{equation}\label{formulazdeb}
\int_{B_{\frac{r_0}{\lambda}}}A(\lambda x)\nabla U^\lambda(x)\cdot \nabla v(x)\, dx = \lambda^2\int_{B_{\frac{r_0}{\lambda}}} \tilde{f}(\lambda x ) U^\lambda(x) v(x)\, dx\quad \text{for every $v\in C^\infty_c(B_{\frac{r_0}{\lambda}}\setminus\tilde{\Gamma})$}.
\end{equation}
In the following lemma we prove that the family $\{U^\lambda\}_{\lambda\in (0,r_0)}$ is bounded in $H^1(B_1)$ which is the smallest possible space since as $\lambda\to 0^+$ the ball $B_{\frac{r_0}{\lambda}}$ becomes larger and larger. 
\begin{lem}\label{lemmalimitatainB1} 
There exists a positive constant $M>0$ such that 
\begin{equation*}
\Vert U^\lambda\Vert_{H^1(B_1)} \leq M \quad \text{for every $\lambda\in (0,r_0)$}. 
\end{equation*}
\end{lem}
\begin{proof}
Thanks to \eqref{unificata} and \eqref{serviraancora}, applying a suitable change of variable, we have that 
\begin{equation*}
\frac{1}{4}\lambda^{N-1}H(\lambda)\int_{B_1} |\nabla U^\lambda|^2\, dz\leq \lambda ^{N-1} \left(D(\lambda) + \frac{N-1}{4} H(\lambda)\right). 
\end{equation*}
Hence dividing each member of the above inequality by $\lambda^{N-1}H(\lambda)$ and taking into account \eqref{Nlimitatadasopra}, we deduce that 
\begin{equation}\label{limitatezzadelgradiente}
\Vert\nabla U ^\lambda\Vert_{L^2(B_1)}^2\leq 4C_4+N-1. 
\end{equation}
Moreover from \eqref{hardy} we can infer that 
\begin{equation*}
\left(\frac{N-1}{2}\right)^2 \lambda^{N-1} H(\lambda) \int_{B_1}|U^\lambda|^2\,dx\leq \lambda^{N-1} H(\lambda)\left( \int_{B_1} |\nabla U^\lambda|^2\, dx +N-1\right),
\end{equation*}
exploiting \eqref{mumaggiore} in order to make $H$ appear on the right-hand side. Dividing each member of the last inequality by $\lambda^{N-1}H(\lambda)$ and using \eqref{limitatezzadelgradiente}, we obtain that
\begin{equation}\label{limitatezzaulaminL2}
\Vert U ^\lambda\Vert_{L^2(B_1)}^2\leq \frac{16 C_4}{(N-1)^2}+\frac{8}{N-1}. 
\end{equation}
Combining \eqref{limitatezzadelgradiente} and \eqref{limitatezzaulaminL2}, we arrive at the thesis. 
\end{proof}
The following three lemmas are thought to derive the boundedness of the $L^2(\partial B_1)$-norm of a slightly modified version of $\nabla U^\lambda$ (see Lemma \ref{lemmagradienteboundary} below), which in turn will come into play to establish the convergence-type result of Proposition \ref{propconvergenza} below.

\begin{lem}\label{lemdoublingH}
There exists a positive constant $M_1>0$ such that for every $\lambda\in (0,\frac{r_0}{2})$ and for every $R\in [1,2]$
\begin{equation}\label{doublingH}
M_1^{-1} H(\lambda)\leq H(R\lambda)\leq M_1 H(\lambda).
\end{equation}
\end{lem}
\begin{proof}
Using \eqref{H'D}, \eqref{Nmaggiore} and \eqref{Nlimitatadasopra}, we can deduce that there exist two positive constants $c_1,c_2>0$ such that 
\begin{equation*}
-\frac{c_1}{r}\leq\frac{H'(r)}{H(r)}\leq \frac{c_2}{r} 
\end{equation*}
for a.e. $r\in (0,r_0)$, up to select $r_0$ smaller from the beginning. Integrating over $(\lambda,R\lambda)$ with $R\in (1,2]$ and $\lambda\in (0,\frac{r_0}{R})$, we obtain that 
\begin{equation*}
2^{-c_1}\leq \frac{H(R\lambda)}{H(\lambda)}\leq 2^{c_2}.
\end{equation*}
The above inequality still holds if $R=1$. So \eqref{doublingH} follows after observing that $(0,\frac{r_0}{2})\subseteq (0,\frac{r_0}{R})$.
\end{proof}
\begin{lem}\label{Lemmaprecedente}
Let $M_1>0$ be as in Lemma \ref{doublingH}. Then for every $\lambda\in (0,\frac{r_0}{2})$ and for every $R\in [1,2]$ it holds that
\begin{equation*}
\int_{B_R}  |U^\lambda|^2\, dx \leq M_1 2^{N+1}\int_{B_1}  |U^{R\lambda}|^2\,dx,
\end{equation*}
and
\begin{equation*}\label{s}
\int_{B_R}  |\nabla U^\lambda|^2\, dx \leq M_1 2^{N-1}\int_{B_1}  |\nabla U^{R\lambda}|^2\,dx.
\end{equation*}
\end{lem}
\begin{proof}
We omit the proof since one can proceed exactly as in the proof of \cite[Lemma 5.3]{DelFel}, using Lemma \ref{lemdoublingH} and applying suitable changes of variable. 
\end{proof}
\begin{lem}\label{lemmarlambda}
Let $U^\lambda$ be as in \eqref{Ulambda} with $\lambda\in (0,r_0)$. Then there exist $M_2>0$ and $\bar{\lambda}>0$ such that for every $\lambda\in (0,\bar{\lambda})$ there exists $R_\lambda
\in [1,2]$ such that 
\begin{equation*}\label{M_2}
\int_{\partial B_{R_\lambda}}| \nabla U^\lambda|^2\,ds\leq M_2\int_{ B_{R_\lambda}} |\nabla U^\lambda| ^2\,dx.
\end{equation*}
\end{lem}
\begin{proof}
For every fixed $\lambda\in (0,\frac{r_0}{2})$ the function 
\begin{equation}\label{glam}
g_\lambda: r\mapsto g_\lambda(r):= \int_{B_r}|\nabla U^\lambda(x)|^2\,dx
\end{equation}
is absolutely continuous in $[0,2]$ and thus
\begin{equation*}
g_\lambda'(r)= \int_{\partial B_r}|\nabla U^\lambda(x)|^2\,ds
\end{equation*}
in a distributional sense and a.e. in $(0,2)$. Now we assume by contradiction that for every $M_2>0$ there exists a sequence $\lambda_n\to 0^+$ such that $g_{\lambda_n}'(r)>M_2\, g_{\lambda_n}(r)$ a.e. in $(0,2)$ for every $n\geq 1$; so, integrating with respect to $r$, we deduce that $g_{\lambda_n}(2)> M_2\, g_{\lambda_n}(1)$ for every $n\geq 1$. From this we can infer that
\begin{equation}\label{passoallimitediM2}
\liminf_{\lambda\to 0^+}g_\lambda(1)\leq e^{-M_2}\limsup_{\lambda\to 0^+} g_\lambda(2),
\end{equation} 
where the $\limsup$ in the right-hand side is less than $+\infty$ as a consequence of the boundedness of $\{U^\lambda\}_{\lambda\in (0,r_0/2)}$ in $H^1(B_2)$ (this comes from Lemmas \ref{lemmalimitatainB1} and \ref{Lemmaprecedente}). Thus, passing to the limit as $M_2\to \infty$ in \eqref{passoallimitediM2} and recalling the definition of $g_\lambda$ given in \eqref{glam}, we get that 
\begin{equation}\label{liminfzero}
\liminf_{\lambda\to 0^+} \int_{B_1}|\nabla U^\lambda(x)|^2\, dx =0.
\end{equation}
Now we claim that there exists a sequence $\lambda_n\to 0^+$ such that $U^{\lambda_n}\rightharpoonup W$ in $H^1(B_1)$ for some $W\in H^1(B_1)$ and in addition 
\begin{equation*}
\lim_{n\to\infty}\int_{B_1}|\nabla U^{\lambda_n}(x)|^2\, dx =0.
\end{equation*}
This can be deduced by combining \eqref{liminfzero} and Lemma \ref{lemmalimitatainB1}. In particular, this and the weak lower semicontinuity of $L^2$-norm imply that 
\begin{equation*}
\int_{B_1} |\nabla W(x)|^2\, dx=0.
\end{equation*}
Hence we have that $W$ is equal to a non-zero constant since
\begin{equation}\label{dinuovointuga1}
\int_{\partial B_1}|W|^2\, ds=1,
\end{equation}
which is a consequence of \eqref{int=1} and the compactness of the trace operator \eqref{traceoperator}. We are almost done: indeed 
we notice that since $U^{\lambda_n}$ is in the space defined in \eqref{lospaziodiulambda} then $U^{\lambda_n}$ belongs to 
\begin{equation}\label{dovetsaubar}
\{v\in H^1(B_1):v=0\ \text{on $\tilde{\Gamma}$ in a trace sense}\},
\end{equation}
which is weakly closed in $H^1(B_1)$; thus necessarily $W\equiv 0$ in $B_1$, producing a contradiction with \eqref{dinuovointuga1}. 
\end{proof}
\begin{lem}\label{lemmagradienteboundary}
There exists a positive constant $M_3>0$ such that for every $\lambda\in (0, \mathrm{min}\{\bar{\lambda},\frac{r_0}{2} \})$
\begin{equation*}
\int_{\partial B_1} |\nabla U^{\lambda R_\lambda}| ^2\, ds\leq M_3,
\end{equation*} 
being $R_\lambda\in [1,2]$ as in Lemma \ref{lemmarlambda}.
\end{lem}
\begin{proof}
%In virtue of Lemma \ref{lemmarlambda} we can assert that for every $\lambda\in (0, \mathrm{min}\{\bar{\lambda},\frac{\tilde{r}}{2} \})$ there exists $R_\lambda\in [1,2]$ such that \eqref{M_2} holds true, and consequently  
%\begin{equation*}
%\begin{split}
%\int_{\partial B_1} |\nabla u^{\lambda R_\lambda}|^2\, dS &= R_\lambda^{2-N}\frac{H(\lambda)}{H(\lambda R_\lambda)}\int_{\partial B_{R_\lambda}}|\nabla u^{\lambda}|^2\, dS\leq C_3M_1\int_{ B_{R_\lambda}}|\nabla u^{\lambda}|^2\, dS\\
%&\leq 2^{N-1}C_3^2 M_1 \int_{ B_{1}}|\nabla u^{\lambda R_\lambda}|^2\, dz
%\end{split}
%\end{equation*}
%taking also into account Lemma \ref{doublingH}, applying \eqref{s} with $R=R_\lambda$ and Lemma \ref{lemmalimitatainB1} (since $\lambda R_\lambda\in (0,\tilde{r})$).  
We pass over the proof since one can reason precisely as in the proof of \cite[Lemma 5.5]{DelFel}  taking advantage of Lemmas \ref{lemdoublingH}, \ref{Lemmaprecedente} and \ref{lemmarlambda} and applying suitable changes of variable. 
\end{proof}
For future reference, it is useful to write the following lemma which is a straightforward consequence of Lemma \ref{lemmagradienteboundary}. 
\begin{lem}\label{lemmaispirazione}
If $\lambda_n\to 0^+$, then there exist a subsequence $\{\lambda_{n_k}\}_{k\geq 1}$ and a $L^2(\partial B_1)$- function $h$ such that 
\begin{equation*}\label{convergenzaconA}
A(\lambda_{n_k}R_{\lambda_{n_k}}\cdot)\nabla U^{\lambda_{n_k}R_{\lambda_{n_k}}}\cdot\nu\rightharpoonup h\quad \text{in $L^2(\partial B_1)$ as $k\to \infty$},
\end{equation*}
with $R_{\lambda_{n_k}}\in [1,2]$ as in Lemma \ref{lemmarlambda}. 
\end{lem}
\begin{proof}
Let $\lambda_n\to 0^+$. From Lemma \ref{lemmagradienteboundary} we can deduce that there exist a subsequence $\{\lambda_{n_k}\}_{k\geq 1}$ and a $L^2(\partial B_1)$- function $h$ such that for any $\psi\in L^2(\partial B_1)$
\begin{equation*}
\int_{\partial B_1}\frac{\partial U^{\lambda_{n_k}R_{\lambda_{n_k}}}}{\partial\nu}\psi\, ds\to \int_{\partial B_1} h \psi\, ds\quad \text{as $k\to \infty$},
\end{equation*} 
being $R_{\lambda_{n_k}}\in [1,2]$ chosen as in Lemma \ref{lemmarlambda}.
Summing this fact with \eqref{stimadiA}, we have that
\begin{equation*}
\begin{split}
\int_{\partial B_1}(&A(\lambda_{n_k}R_{\lambda_{n_k}}\cdot)\nabla U^{\lambda_{n_k}R_{\lambda_{n_k}}}\cdot\nu)\psi\, ds \\
&= \int_{\partial B_1} \frac{\partial U^{\lambda_{n_k}R_{\lambda_{n_k}}}}{\partial\nu}\psi\, ds+ O(\lambda_{n_k}R_{\lambda_{n_k}}) \int_{\partial B_1}  (\nabla U^{\lambda_{n_k}R_{\lambda_{n_k}}}\cdot \nu)\psi\, ds\\
&\to \int_{\partial B_1} h\psi\, ds \quad \text{as $k\to \infty$},
\end{split}
\end{equation*}
where in addition we have used that 
\begin{equation*}
\begin{split}
\left|\int_{\partial B_1}  (\nabla U^{\lambda_{n_k}R_{\lambda_{n_k}}}\cdot \nu)\psi\, ds\right|&\leq M_3\left(\int_{\partial B_1} |\psi|^2\, ds\right)^{\frac{1}{2}},
\end{split}
\end{equation*}
by the H\"{o}lder inequality and thanks to Lemma \ref{lemmagradienteboundary}. The proof is thereby complete. 
\end{proof}

\begin{lem}\label{illemmaaggiunto}
Let $\ell$ be as in \eqref{ell}.
For any sequence $\lambda_n\to 0^+$, given $R_{\lambda_n}\in [1,2]$ as in Lemma \ref{lemmarlambda}, there exists a subsequence $\{\lambda_{n_k}\}_{k\geq 1}$ and an homogeneous function $\overline{U}\in H^1(B_1)$ of degree $\ell$, i.e.
\begin{equation}\label{eomogenea}
\overline{U}(x)=|x|^\ell \overline{U}\left(\frac{x}{|x|}\right)\quad \text{for all $x\in B_1$},
\end{equation}
such that 
\begin{equation}\label{perintero}
U^{\lambda_{n_k} R_{\lambda_{n_k}}}\to \overline{U} \quad \text{in $H^1(B_1)$ as $k\to \infty$}.
\end{equation}
Moreover $\Psi:= \overline{U}|_{\mathbb{S}^N}$ is an $L^2$-normalized eigenfunction of problem \eqref{problagliautovalori} associated with $\ell(\ell+N-1)$.
\end{lem}
\begin{rem}\label{eimportante}
For future goals, we stress that from Lemma \ref{illemmaaggiunto} the number $\ell(\ell+N-1)$ turns out to be an eigenvalue of problem \eqref{problagliautovalori} and hence there exists $k_0\geq 1$ such that $\ell(\ell+N-1)=\mu_{k_0}$.
\end{rem}
\begin{proof}[proof of Lemma \ref{illemmaaggiunto}]
Let $\lambda_n\to 0^+$. Taking $R_{\lambda_n}\in [1,2]$ as in Lemma \ref{lemmarlambda}, we have that $\{U^{\lambda_n R_{\lambda_n}}\}_{n\geq 1}$ is uniformly bounded in $H^1(B_1)$  with respect to $n$, as a consequence of Lemma \ref{lemmalimitatainB1}. Therefore there exist a subsequence $\{\lambda_{n_k}\}_{k\geq 1}$ and $\overline{U}\in H^1(B_1)$ such that 
\begin{equation}\label{ukconvdeb}
U^{k}\rightharpoonup \overline{U}\ \text{in $H^1(B_1)$ as $k\to\infty$},
\end{equation}
having set $U^k:=U^{\lambda_{n_k} R_{\lambda_{n_k}}}$ for every $k\geq 1$ in order to lighten the notation. In particular, it holds that
\begin{equation}\label{illimitenonezero}
\overline{U}\not\equiv 0;
\end{equation}
indeed, combining \eqref{ukconvdeb} and the compactness of the trace operator defined in \eqref{traceoperator}, we have that
\begin{equation}\label{latraccia}
U^k\to \overline{U} \ \text{in $L^2(\partial B_1)$},
\end{equation}
which in turn, together with \eqref{stimadimu} and \eqref{int=1}, allows us to deduce that 
\begin{equation}\label{int'=1}
\int_{\partial B_1} |\overline{U}|^2\, ds=1.
\end{equation}
Now we aim at finding the boundary value problem solved by $\overline{U}$. For this, we first claim that $\overline{U}$ satisfies
\begin{equation}\label{formdebubar}
\int_{B_1} \nabla \overline{U}\cdot\nabla \varphi\, dx=0\quad \text{for every $\varphi\in C^\infty_c(B_1\setminus \tilde{\Gamma})$}. 
\end{equation}
To prove this, we fix any function $\varphi\in C^\infty_c(B_1\setminus \tilde{\Gamma})$. Since
\begin{equation}\label{inclusioneinBk}
B_1\subset B_{r_0/(\lambda_{n_k} R_{\lambda_{n_k}})} \ \text{for sufficiently large $k\geq 1$},
\end{equation}
in particular we have that $\varphi\in C^\infty_c(B_{r_0/(\lambda_{n_k} R_{\lambda_{n_k}})}\setminus \tilde{\Gamma})$ for sufficiently large $k\geq 1$.
Thus, taking in \eqref{formulazdeb} $\lambda=\lambda_{n_k}R_{\lambda_{n_k}}$ for sufficiently large $k$, we have that 
\begin{equation}\label{nonavevonumerato}
\int_{B_1} A(\lambda_{n_k}R_{\lambda_{n_k}}x) \nabla U^k(x)\cdot \nabla \varphi(x)\, dx= (\lambda_{n_k}R_{\lambda_{n_k}})^2 \int_{B_1} \tilde{f}(\lambda_{n_k}R_{\lambda_{n_k}}x)U^k(x)\varphi(x)\, dx 
\end{equation}
for sufficiently large $k$. Reasoning exactly as in the proof of Lemma \ref{lemmaispirazione} and exploiting \eqref{ukconvdeb}, we obtain that 
\begin{equation}\label{pezzosx}
\int_{B_1} A(\lambda_{n_k}R_{\lambda_{n_k}}x) \nabla U^k(x)\cdot \nabla \varphi(x)\, dx \to \int_{B_1} \nabla \overline{U}(x)\cdot\nabla \varphi(x)\, dx\quad\text{as $k\to\infty$}.
\end{equation}
Additionally we observe that under assumption \eqref{asssuftilde1}
\begin{equation*}
\begin{split}
(\lambda_{n_k}&R_{\lambda_{n_k}})^2\left|\int_{B_1} \tilde{f}(\lambda_{n_k}R_{\lambda_{n_k}}x)U^k(x)\varphi(x)\, dx\right|\\
\leq&\, (\lambda_{n_k}R_{\lambda_{n_k}})^\delta\left(\int_{B_1} \frac{|U^k(x)|^2}{|x|^2}\, dx\right)^{\frac{1}{2}} \cdot \left(\int_{B_1} \frac{|\varphi(x)|^2}{|x|^2}\, dx\right)^{\frac{1}{2}}\\
\leq &\,\frac{4(\lambda_{n_k}R_{\lambda_{n_k}})^\delta }{(N-1)^2}\left(\int_{B_1}|\nabla U^k(x)|^2\, dx+ (N-1)\int_{\partial B_1}\mu(\lambda_{n_k}R_{\lambda_{n_k}}\cdot)|U^k|^2\, ds \right)^{\frac{1}{2}}\cdot\left(\int_{B_1} |\nabla \varphi(x)|^2\, dx\right)^{\frac{1}{2}}\\
\leq &\,\frac{4(\lambda_{n_k}R_{\lambda_{n_k}})^\delta }{(N-1)^2}\left(\int_{B_1}|\nabla U^k(x)|^2\, dx+ N-1 \right)^{\frac{1}{2}}\left(\int_{B_1} |\nabla \varphi(x)|^2\, dx\right)^{\frac{1}{2}}=o(1)\  \text{as $k\to\infty$},
\end{split}
\end{equation*}
where we have applied in order the H\"{o}lder inequality, \eqref{hardy}, \eqref{mumaggiore}, \eqref{int=1} and the result in Lemma \ref{lemmalimitatainB1}.
Instead, under assumption \eqref{asssuftilde2}, we have that for some $\mathrm{const}>0$
\begin{equation*}
\begin{split}
(\lambda_{n_k}&R_{\lambda_{n_k}})^2\left|\int_{B_1} \tilde{f}(\lambda_{n_k}R_{\lambda_{n_k}}x)U^k(x)\varphi(x)\, dx\right|\\
\leq&\, \mathrm{const}\,(\lambda_{n_k}R_{\lambda_{n_k}})^{\frac{4p-N-1}{p}}\left(\int_{B_1} |\nabla U^k|^{2}\,dx + (N-1)\int_{\partial B_1}\mu(\lambda_{n_k}R_{\lambda_{n_k}}\cdot)|U^k|^2\,ds\right)^{\frac{1}{2}}\left(\int_{B_1}|\nabla \varphi|^2\,dx\right)^{\frac{1}{2}}\\
\leq&\, \mathrm{const}\,(\lambda_{n_k}R_{\lambda_{n_k}})^{\frac{4p-N-1}{p}}\left(\int_{B_1} |\nabla U^k|^{2}\,dx + N-1\right)^{\frac{1}{2}}\left(\int_{B_1}|\nabla \varphi(x)|^2\,dx\right)^{\frac{1}{2}}=o(1)\  \text{as $k\to\infty$},
\end{split}
\end{equation*}
using again the H\"{o}lder inequality, \eqref{stimadiL2star}, \eqref{mumaggiore}, \eqref{int=1} and the result in Lemma \ref{lemmalimitatainB1}.
So in both cases we have that 
\begin{equation}\label{pezzodx}
(\lambda_{n_k}R_{\lambda_{n_k}})^2 \int_{B_1} \tilde{f}(\lambda_{n_k}R_{\lambda_{n_k}}x)U^k(x)\varphi(x)\, dx\to 0 \quad \text{as $k\to\infty$}. 
\end{equation}
Thus summing \eqref{pezzosx} and \eqref{pezzodx}, passing to the limit as $k\to\infty$ in \eqref{nonavevonumerato}, we get \eqref{formdebubar}. 
Moreover we remark that, since $U^k$ belongs to the space defined in \eqref{lospaziodiulambda} for every $k$, in virtue of \eqref{inclusioneinBk}, for sufficiently large $k$ it holds that $U^k$ belongs to the space \eqref{dovetsaubar}, which is weakly closed in $H^1(B_1)$; so, thanks to \eqref{ukconvdeb}, we also have that $\overline{U}$ belongs to the above space. 
Putting together this information with \eqref{formdebubar}, we can deduce that $\overline{U}$ is a weak solution to 
\begin{equation}\label{problemarisoltodaubar}
\left\{\begin{aligned}
-\Delta
\overline{U}&=0 &&\text{in}\ B_1\setminus\tilde\Gamma, \\
\overline{U}&=0 && \text{on}\ \tilde\Gamma.
\end{aligned}\right.
\end{equation}
Our next goal is to prove that \eqref{ukconvdeb} actually holds in a strong sense, namely we want to prove that 
\begin{equation}\label{strong}
U^k\to \overline{U}\ \text{in $H^1(B_1)$ as $k\to\infty$}.
\end{equation}
To this aim, we test equation \eqref{probulambda} with any $\varphi\in C^\infty_c(\overline{B_1}\setminus \tilde \Gamma)$ (since it belongs to $C^\infty_c(\overline{B_{r_0/(\lambda_{n_k}R_{\lambda_{n_k}})}}\setminus \tilde \Gamma)$ for sufficiently large $k$), obtaining that for sufficiently large $k$
\begin{equation}\label{dapassareallimite}
\begin{split}
\int_{B_1} A(\lambda_{n_k}R_{\lambda_{n_k}}x) \nabla U^k(x)\cdot \nabla \varphi(x)\, dx=\,& (\lambda_{n_k}R_{\lambda_{n_k}})^2 \int_{B_1} \tilde{f}(\lambda_{n_k}R_{\lambda_{n_k}}x)U^k(x)\varphi(x)\, dx\\
& +  \int_{\partial B_1} (A(\lambda_{n_k}R_{\lambda_{n_k}}x) \nabla U^k(x)\cdot \nu(x))\varphi(x)\, dx.
\end{split}
\end{equation}
In light of \eqref{pezzosx}, \eqref{pezzodx} and Lemma \ref{lemmaispirazione}, up to a further subsequence still denoted with $\lambda_{n_k}$, taking the limit in \eqref{dapassareallimite} as $k\to\infty$, we get that for every $\varphi\in C^\infty_c(\overline{B_1}\setminus \tilde \Gamma)$ 
\begin{equation*}
\int_{B_1} \nabla \overline{U}\cdot \nabla \varphi\, dx= \int_{\partial B_1} h \varphi\, ds. 
\end{equation*}
The above identity still holds by density if we choose as $\varphi$ the function $\overline{U}$ itself since $\overline{U}$ belongs to the space \eqref{dovetsaubar} as observed above, and thus we can conclude that 
\begin{equation}\label{valeancora}
\int_{B_1} |\nabla \overline{U}|^2\, dx= \int_{\partial B_1} h \overline{U}\, ds.
\end{equation}
Furthermore we are allowed to compute \eqref{dapassareallimite} at $\varphi=U^k$ by a density argument since $U^k$ belongs to the space \eqref{dovetsaubar} for sufficiently large $k$, thus having that
\begin{equation}\label{larichiami}
\begin{split}
\lim_{k\to \infty} \int_{B_1} A(\lambda_{n_k}R_{\lambda_{n_k}}x)\nabla U^k(x)\cdot\nabla U^k(x)\, dx &= \lim_{k\to \infty} \int_{\partial B_1} (A(\lambda_{n_k}R_{\lambda_{n_k}}x) \nabla U^k(x)\cdot \nu(x)) U^k(x)\, ds\\
&=\int_{\partial B_1} h\overline{U}\, ds= \int_{B_1}|\nabla \overline{U}|^2\, dx, 
\end{split}
\end{equation}
if in addition we use \eqref{pezzodx} (which can proven in the same way as before even choosing $\varphi=U^k$), Lemma \ref{lemmaispirazione}, \eqref{latraccia} and ultimately \eqref{valeancora}. From this, proceeding as in the proof of Lemma \ref{lemmaispirazione}, we obtain that 
\begin{equation*}
\lim_{k\to\infty}|\nabla U^k|^2\, dx= \int_{B_1} |\nabla \overline{U}|^2\, dx,
\end{equation*}
which, combined with \eqref{ukconvdeb}, leads to have that $\nabla U^k\to \nabla \overline{U}$ in $L^2(B_1)$. This, together with \eqref{latraccia} and \eqref{hardy}, implies that $U^k\to \overline{U}$ in $L^2(B_1)$. The proof of \eqref{strong} is thereby complete.

Let us now prove that $\overline{U}$ is homogeneous of degree $\ell$. 
We then investigate the associated Almgren function, that, based on \eqref{problemarisoltodaubar}, is defined as follows (provided that the denominator is non-null):
\begin{equation}\label{Nubar}
\mathcal{N}_{\overline{U}}: t\in (0,1]\mapsto \mathcal{N} _{\overline{U}}(t):= \frac{D_{\overline{U}}(t)}{H_{\overline{U}}(t)},
\end{equation}
where 
\begin{equation*}\label{Dubar}
D_{\overline{U}}(t):= t^{1-N} \int_{B_t} |\nabla \overline{U}|^2\, dx,
\end{equation*}
and 
\begin{equation}\label{Hubar}
H_{\overline{U}}(t):= t^{-N} \int_{\partial B_t} |\overline{U}|^2\, ds.
\end{equation}
A similar discussion as that one used to prove Lemma \ref{lapositivitadiH} (taking $\tilde{f}\equiv 0$) ensures that $H_{\overline{U}}>0$,  taking advantage of \eqref{illimitenonezero}.

Let us now introduce the Almgren function associated with $U^k$ for every $k\geq 1$. For this we define, for any fixed $k\geq 1$, the following two functions 
\begin{equation*}\label{Dk}
\begin{split}
D_k :\, &t\in (0,1]\mapsto \\
&D_k(t):= t^{1-N}\left(\int_{B_t} A(\lambda_{n_k}R_{\lambda_{n_k}}\cdot)\nabla U^k\cdot\nabla U^k\, dx -(\lambda_{n_k} R_{\lambda_{n_k}})^2 \int_{B_t}\tilde{f}(\lambda_{n_k} R_{\lambda_{n_k}}\cdot)|U^k|^2\, dx \right), 
\end{split}
\end{equation*}
and \begin{equation*}\label{Hk}
H_k : t\in (0,1]\mapsto H_k(t):= t^{-N}\int_{\partial B_t} \mu (\lambda_{n_k} R_{\lambda_{n_k}}\cdot)|U^k|^2\, ds.
\end{equation*}
We notice that 
\begin{equation}\label{Hkequivalente}
H_k(t)=  \frac{H(\lambda_{n_k} R_{\lambda_{n_k}} t)}{H(\lambda_{n_k} R_{\lambda_{n_k}})},
\end{equation}
and thus $H_k>0$ by Lemma \ref{lapositivitadiH} (we stress that the fact that $\lambda_{n_k}\to 0^+$ guarantees that $\lambda_{n_k}R_{\lambda_{n_k}}\in (0,r_0]$). Hence the Almgren function given by
\begin{equation*}\label{N_k}
 \mathcal{N} _k:t\in (0,1]\mapsto \mathcal{N} _k(t):= \frac{D_k(t)}{H_k(t)}
\end{equation*}
is well-defined and satisfies 
\begin{equation}\label{unarelazioneutile}
\mathcal{N}_k(t)= \mathcal{N}(\lambda_{n_k} R_{\lambda_{n_k}} t),
\end{equation} 
thanks to \eqref{Hkequivalente} and since it also holds that 
\begin{equation*}
D_k(t)=  \frac{D(\lambda_{n_k} R_{\lambda_{n_k}} t)}{H(\lambda_{n_k} R_{\lambda_{n_k}})}. 
\end{equation*}
So from \eqref{larichiami}, \eqref{pezzodx} (which is still valid with $U^k$ in place of $\varphi$), \eqref{latraccia}, \eqref{stimadimu}, \eqref{unarelazioneutile} and \eqref{ell}, we can deduce that for every $t\in (0,1]$
\begin{equation}\label{Nubarugualeaelle}
\mathcal{N}_{\overline{U}}(t)= \lim_{k\to \infty } \mathcal{N}_k(t)= \lim_{k\to \infty } \mathcal{N}(\lambda_{n_k} R_{\lambda_{n_k}} t) = \ell. 
\end{equation}
Thus $\mathcal{N}'_{\overline{U}}(t)=0$ for a.e. $t\in (0,1)$ and consequently, considering that $\overline{U}$ is a weak solution to \eqref{problemarisoltodaubar} and arguing as to prove Lemma \ref{lemmaN'}, we have that for a.e. $t\in (0,1)$ 
\begin{equation*}
0\leq \frac{2t\left[\left(\int_{\partial B_t} \left|\frac{\partial \overline{U}}{\partial\nu}\right|^2\, ds\right)\left(\int_{\partial B_t} |\overline{U}|^2\, ds\right)- \left(\int_{\partial B_t} \overline{U}\frac{\partial \overline{U}}{\partial\nu}\, ds\right)^2\right]}{\left(\int_{\partial B_t} |\overline{U}|^2\, ds\right)^2}\leq \mathcal{N}'_{\overline{U}}(t)=0,
\end{equation*} 
by Cauchy-Schwarz's inequality. This in particular implies that $\overline{U}$ and $\partial\overline{U}/\partial\nu$ have the same direction in $L^2(\partial B_t)$ for a.e. $t\in (0,1)$ and hence, writing any $x\in \partial B_t$ as $x=t\vartheta$ with $\vartheta=\frac{x}{|x|}\in \mathbb{S}^N$, we have that 
\begin{equation}\label{daintegr3}
\frac{\partial \overline{U}}{\partial t}(t\vartheta)=\frac{\partial \overline{U}}{\partial\nu}(t\vartheta)= \eta(t) \overline{U}(t\vartheta)\quad \text{for a.e. $t\in (0,1)$ and for every $\vartheta\in \mathbb{S}^N$},
\end{equation}
for some function $\eta=\eta(t)$ defined a.e. in $(0,1)$. Multiplying the above identity by $\overline{U}(t\vartheta)$ itself and then integrating with respect to $\vartheta$ over $\mathbb{S}^N$, by \eqref{Hubar}, \cite[Lemma 4.1]{DelFel} applied to $H_{\overline{U}}(t)$, \eqref{Nubar} and at last \eqref{Nubarugualeaelle}, we can deduce that 
\begin{equation*}
\eta(t) = \frac{1}{2}\frac{H'_{\overline{U}}(t)}{H_{\overline{U}}(t)}=\frac{\mathcal{N}_{\overline{U}}(t)}{t} = \frac{\ell}{t}, 
\end{equation*}
which in turn allows us to establish that $\eta$ is summable staying far from 0. Plugging this into \eqref{daintegr3} and integrating with respect to $t$ over $(r,1)$ for any fixed $r\in (0,1)$, we obtain that 
\begin{equation}\label{explicit}
\overline{U}(r\vartheta)=r^\ell \overline{U}(\vartheta)= r^\ell \Psi(\vartheta) \quad \text{for every $r\in (0,1)$ and $\vartheta\in\mathbb{S}^N$}, 
\end{equation}
where $\Psi:= \overline{U}|_{\mathbb{S}^N}$. By \eqref{int'=1} we have that  
\begin{equation*}
\int_{\mathbb{S}^N}|\Psi|^2\, ds=1,
\end{equation*}
so that $\Psi$ is non-trivial on $\mathbb{S}^N$. Furthermore, using \eqref{explicit}, since $\overline{U}$ belongs to \eqref{dovetsaubar} then $\Psi\in \mathcal{H}_\Theta$ (defined in \eqref{Htheta}); from \eqref{problemarisoltodaubar}, we can find that $\Psi$ satisfies
\begin{equation*}
\Delta_{\mathbb{S}^N}\Psi(\vartheta) + \ell(N+\ell-1)\Psi(\vartheta)=0
\end{equation*} 
in a weak sense, that is \eqref{fordebpsi} holds. So by definition $\Psi$ turns out to be an eigenfunction of problem \eqref{problagliautovalori} associated with $\ell(N+\ell-1)$. 
\end{proof}
\begin{lem}\label{ellelegataautovalore}
Let $\ell$ be as in \eqref{ell} and $k_0\geq 1$ be as in Remark \ref{eimportante}. Then
\begin{equation}\label{ellekomezzi}
\ell=\frac{k_0}{2}. 
\end{equation}
\end{lem}
\begin{proof}
Let $k_0\geq 1$ be as in Remark \ref{eimportante}.
In order to prove \eqref{ellekomezzi}, we claim that  
\begin{equation*}
\ell = -\frac{N-1}{2}+\sqrt{\left(\frac{N-1}{2}\right)^2+\mu_{k_0}}.
\end{equation*}
Once this is shown, \eqref{ellekomezzi} follows from the explicit formula \eqref{autovespliciti} for the eigenvalues of problem \eqref{problagliautovalori}. In order to prove the claim, we take into account Remark \ref{eimportante} to conclude that $\ell$ solves  
\begin{equation*}
y^2 +y(N-1)- \mu_{k_0}=0,
\end{equation*}
which in general admits two solutions $y_{\pm}$ given by 
\begin{equation*}\label{sol1}
y_{\pm}= -\frac{N-1}{2} \pm \sqrt{\left(\frac{N-1}{2}\right)^2+\mu_{k_0}}. 
\end{equation*}
We are allowed to exclude that $\ell=y_{-}$ since in such a case we would have that
\begin{equation*}
\overline{U}(x)= |x|^{y_-}\Psi\left(\frac{x}{|x|}\right)\notin L^{2^\ast}(B_1),
\end{equation*}
which contradicts \eqref{dovetsaubar} due to the validity of \eqref{stimadiL2star}.  
The claim is thereby proved and hence the proof of lemma is complete.
\end{proof}
\begin{prop}\label{propconvergenza}
Let $k_0\geq 1$ be as in Lemma \ref{ellelegataautovalore}. 
For any sequence $\lambda_n\to 0^+$ there exist a subsequence $\{\lambda_{n_k}\}_{k\geq 1}$ and an $L^2$-normalized eigenfunction $\Psi$ of problem \eqref{problagliautovalori} associated with the eigenvalue $\mu_{k_0}$ such that 
\begin{equation*}
U^{\lambda_{n_k}}(x)\to |x|^{\frac{k_0}{2}}\Psi\left(\frac{x}{|x|}\right)\quad \text{in $H^1(B_1)$ as $k\to\infty$}. 
\end{equation*} 
\end{prop}
\begin{proof}
Let $\lambda_n\to 0^+$. By Lemma \ref{lemmalimitatainB1} we have that $\{U^{\lambda_n}\}_{n\geq 1}$ is uniformly bounded in  $H^1(B_1)$ with respect to $n$ and so there exist a subsequence $\{\lambda_{n_k}\}_{k\geq 1}$ and an $H^1(B_1)$-function $\underline{U}$ such that
\begin{equation}\label{underlineu}
U^{\lambda_{n_k}}\rightharpoonup \underline{U} \quad \text{in $H^1(B_1)$ as $k\to \infty$}.
\end{equation}
Now we observe that, if we take $R_{\lambda_{n_k}}\in [1,2]$ as in Lemma \ref{lemmarlambda}, there exist a subsequence (still denoted with $\lambda_{n_k}$) and $\overline{R}\in [1,2]$ such that $R_{\lambda_{n_k}}\to \overline{R}$.
Moreover we notice that, in virtue of \eqref{perintero}, up to a further subsequence (still denoted with $\lambda_{n_k}$) the sequences $U^{\lambda_{n_k}R_{\lambda_{n_k}}}$ and $\nabla U^{\lambda_{n_k}R_{\lambda_{n_k}}}$ are uniformly dominated by two functions in $L^2(B_1)$ and it also holds that $U^{\lambda_{n_k}R_{\lambda_{n_k}}}\to \overline{U}$ and $|\nabla U^{\lambda_{n_k}R_{\lambda_{n_k}}}|^2\to |\nabla \overline{U}|^2$ a.e. in $B_1$, being $\overline{U}$ the limit profile in Lemma \ref{illemmaaggiunto}. In addition, in light of Lemma \ref{lemdoublingH}, we have that up to a further subsequence (still denoted with $\lambda_{n_k}$)
$$\left\{\frac{H(\lambda_{n_k}R_{\lambda_{n_k}})}{H(\lambda_{n_k})}\right\}_{k\geq 1}$$  admits finite limit as $k\to \infty$ that will be denoted by $\eta$.
Now we consider any $v\in C^\infty(\overline{B_1})$. All the previous considerations and the Lebesgue dominated convergence theorem imply that 
\begin{equation*}
\begin{split}
\lim_{k\to \infty} \int_{B_1} U^{\lambda_{n_k}}(x) v(x)\, dx&= \lim_{k\to \infty} R_{\lambda_{n_k}}^{N+1}\sqrt{\frac{H(\lambda_{n_k}R_{\lambda_{n_k}})}{H(\lambda_{n_k})}} \int_{B_{1/R_{\lambda_{n_k}}}}  U^{\lambda_{n_k}R_{\lambda_{n_k}}} (x)v(R_{\lambda_{n_k}}x)\, dx\\
&=\overline{R}^{N+1}\sqrt{\eta} \int_{B_{1/\overline{R}}}\overline{U}(x)v(\overline{R}x)\, dx \\
&= \sqrt{\eta} \int_{B_1} \overline{U}_{\overline{R}}(x)v(x)\, dx, 
\end{split}
\end{equation*}
where $\overline{U}_{\overline{R}}(x):=\overline{U}(x/\overline{R})$ for every $x\in B_1$. By a density argument we thus have that $U^{\lambda_{n_k}} \rightharpoonup\sqrt{\eta}\,\overline{U}_{\overline{R}}$ in $L^2(B_1)$. Combining this with \eqref{underlineu}, we obtain that $\underline{U}=\sqrt{\eta}\,\overline{U}_{\overline{R}}$ and therefore $U^{\lambda_{n_k}} \rightharpoonup\sqrt{\eta}\,\overline{U}_{\overline{R}}$ in $H^1(B_1)$. Actually this last convergence holds in a strong sense: indeed, we have that
\begin{equation*}
\begin{split}
\lim_{k\to\infty}\int _{B_1} |\nabla U^{\lambda_{n_k}}(x)|^2\, dx &= \lim_{k\to\infty} R_{\lambda_{n_k}}^{N-1} \frac{H(\lambda_{n_k}R_{\lambda_{n_k}})}{H(\lambda_{n_k})} \int_{B_{1/R_{\lambda_{n_k}}}}|\nabla U^{\lambda_{n_k}R_{\lambda_{n_k}}}(x)|^2\, dx\\
& = \overline{R}^{N-1} \eta\int_{B_{1/\overline{R}}} |\nabla \overline{U}(x)|^2\, dx = \eta\int_{B_1}|\nabla \overline{U}_{\overline{R}}(x)|^2\, dx,
\end{split}
\end{equation*}
which in turn implies that $\nabla U^{\lambda_{n_k}}\to \sqrt{\eta}\nabla (\overline{U}_{\overline{R}})$ in $L^2(B_1)$; now summing this with the compactness of the trace operator \eqref{traceoperator}, from \eqref{hardy} we can infer that
\begin{equation}\label{convforteaeta}
U^{\lambda_{n_k}} \to\sqrt{\eta}\,\overline{U}_{\overline{R}} \quad \text{in $H^1(B_1)$ as $k\to\infty$}.
\end{equation}
In order to prove that $\sqrt{\eta}\,\overline{U}_{\overline{R}} = \overline{U}$, we exploit \eqref{eomogenea} to deduce that
\begin{equation}\label{usegnatoR}
\overline{U}_{\overline{R}}= \frac{1}{\overline{R}^\ell}\overline{U}. 
\end{equation}
Hence we finish if we prove that $\frac{\sqrt{\eta}}{\overline{R}^\ell}=1$. 
To this purpose, we observe that 
\begin{equation}\label{e1}
\begin{split}
\int_{\partial B_1} \mu(\lambda_{n_k})|U^{\lambda_{n_k}}|^2\, ds&= \int_{\partial B_1} |U^{\lambda_{n_k}}|^2\, ds + O(\lambda_{n_k}) \int_{\partial B_1} |U^{\lambda_{n_k}}|^2\, ds\to  \int_{\partial B_1} \eta|\overline{U}_{\overline{R}}|^2\, ds
\end{split}
\end{equation}
as $k\to\infty$, thanks to \eqref{stimadimu}, \eqref{convforteaeta} combined with the compactness of the trace operator \eqref{traceoperator}, which in turn together with  Lemma \ref{lemmalimitatainB1} leads to the uniform boundedness of $\{U^{\lambda_{n_k}}\}_{k\geq 1}$ in $L^2(\partial B_1)$. Summing \eqref{int=1}, \eqref{e1}, \eqref{usegnatoR} and \eqref{int'=1}, we get
\begin{equation*}
1=\int_{\partial B_1}\eta|\overline{U}_{\overline{R}}|^2\, ds = \eta\int_{\partial B_1} \frac{1}{\overline{R}^\ell}|\overline{U}|^2\, ds = \frac{\eta}{\overline{R}^\ell},
\end{equation*}
as desired. So, pushing this into \eqref{convforteaeta}, taking into account \eqref{usegnatoR}, we obtain that $U^{\lambda_{n_k}}\to\overline{U}$ in $H^1(B_1)$ as $k\to\infty$. The proof is thereby complete if we put together this with Lemma \ref{illemmaaggiunto}, Remark \ref{eimportante} and at last Lemma \ref{ellelegataautovalore}.

\end{proof}
Our next aim is to prove that the limit of $H(r)/r^{k_0}$ as $r\to 0^+$, which we already know to be finite by Lemmas \ref{lemmadellimitecheesiste} and \ref{ellelegataautovalore}, is strictly positive. For this (as made in \cite{DelFel}) an asymptotic expansion of the Fourier coefficients associated with $U$ is needed: in order to fully state the result, for every $k\geq 1$ let $m_k\geq 1$ be the multiplicity of the eigenvalue $\mu_{k}$ and let $\{Y_{k,m}\}_{m\in \{1,2,\dots,m_k\}}$ be a $L^2(\mathbb{S}^N)$-orthonormal basis of the eigenspace associated with $\mu_{k}$; then we define for every $k\geq 1$, $m\in \{1,2,\dots,m_k\}$ and $\lambda\in (0,r_0)$ 
\begin{equation}\label{coeffdiF}
\varphi_{k,m}(\lambda)= \int_{\mathbb{S}^N} U(\lambda\vartheta)Y_{k,m}(\vartheta)\, ds
\end{equation} 
and 
\begin{equation*}\label{iupsilon}
\begin{split}
\Upsilon_{k,m}(\lambda)= &-\int_{B_\lambda}(A-\mathrm{Id}_{N+1})\nabla U(x)\cdot\frac{1}{|x|}\nabla_{\mathbb{S}^N} Y_{k,m}\left(\frac{x}{|x|}\right)\,dx + \int_{B_\lambda} \tilde{f}(x)U(x)Y_{k,m}\left(\frac{x}{|x|}\right)\,dx\\
&+ \int_{\partial B_\lambda}(A-\mathrm{Id}_{N+1})\nabla U(x)\cdot\frac{x}{|x|}Y_{k,m}\left(\frac{x}{|x|}\right)\,ds.
\end{split}
\end{equation*}
We simply rewrite the expansion in \cite[Lemma 6.5]{DelFel} adapted to our setting, namely with $\tilde{f}$ satisfying either \eqref{asssuftilde1} or \eqref{asssuftilde2}, without providing the proof since it is precisely the same (indeed in \cite{DelFel} at this point of dissertation the authors apply a diffeomorphism that gives rise to a new formulation of their problem with the same features of our problem \eqref{problU}). 
\begin{lem}
Let $k_0\geq 1$ be as in Lemma \ref{ellelegataautovalore} and let $m_{k_0}$ the multiplicity of the eigenvalue $\mu_{k_0}$. For every $m\in \{1,2,\dots,m_{k_0}\}$ and $R\in (0,r_0]$
\begin{equation}\label{espansione}
\begin{split}
\varphi_{k_0,m}(\lambda)=\lambda^{\frac{k_0}{2}}\biggl(&\frac{\varphi_{k_0,m}(R)}{R^{\frac{k_0}{2}}}+\frac{2N+k_0-2}{2(N+k_0-1)}\int_\lambda^R t^{-N-\frac{k_0}{2}}\Upsilon_{k_0,m}(t)\,dt\\
&+\frac{k_0 R^{-N+1-k_0}}{2(N+k_0-1)}\int_0^R t^{\frac{k_0}{2}-1}\Upsilon_{k_0,m}(t)\,dt\biggr)+O(\lambda^{\frac{k_0}{2}+\bar{\varepsilon}+1})\quad \text{as $\lambda\to 0^+$},
\end{split}
\end{equation}
being $\bar{\varepsilon}\in (-1,0)$ defined in \eqref{eps}. 
\end{lem}
\begin{lem}\label{lemmalimpositivo}
Let $k_0\geq 1$ be as in Lemma \ref{ellelegataautovalore}. Then 
\begin{equation}\label{limitemagdizero}
\lim_{r\to 0^+}\frac{H(r)}{r^{k_0}}>0.
\end{equation}
\end{lem}
\begin{proof}
Expanding the function $\vartheta\mapsto U(\lambda\vartheta)\in L^2(\mathbb{S}^N)$ for every fixed $\lambda\in (0,r_0)$ in Fourier series with respect to the $L^2(\mathbb{S}^N)$-orthonormal basis $\{Y_{k,m}: k\geq 1, m\in \{1,2,\dots,m_k\}\}$, we have that 
\begin{equation*}
U(\lambda\vartheta)= \sum_{k\geq 1} \sum_{m=1}^{m_k} \varphi_{k,m}(\lambda)Y_{k,m}(\vartheta).
\end{equation*}
From this and \eqref{H}, applying a suitable change of variable, using \eqref{stimadimu} and the Parseval identity, we get 
\begin{equation}\label{Hparseval}
H(\lambda)=(1+O(\lambda))\sum_{k\geq 1}\sum_{m=1}^{m_k} |\varphi_{k,m}(\lambda)|^2\quad \text{for every fixed $\lambda\in (0,r_0)$}. 
\end{equation}
Now we assume by contradiction that the limit in \eqref{limitemagdizero} is zero; so from \eqref{Hparseval} it follows that 
\begin{equation*}
\lim_{\lambda \to 0^+}\frac{\varphi_{k_0,m}(\lambda)}{\lambda^{\frac{k_0}{2}}}=0\quad \text{for every $m\in \{1,2,\dots,m_{k_0}\}$}.
\end{equation*}
This, together with \eqref{espansione}, implies that for every $m\in \{1,2,\dots,m_{k_0}\}$ and $R\in (0,r_0]$
\begin{equation*}
\frac{\varphi_{k_0,m}(R)}{R^{\frac{k_0}{2}}}+\frac{2N+k_0-2}{2(N+k_0-1)}\int_0^R t^{-N-\frac{k_0}{2}}\Upsilon_{k_0,m}(t)\,dt+\frac{k_0 R^{-N+1-k_0}}{2(N+k_0-1)}\int_0^R t^{\frac{k_0}{2}-1}\Upsilon_{k_0,m}(t)\,dt=0,
\end{equation*}
which in turn, substituted into \eqref{espansione}, gives us that for every $m\in \{1,2,\dots,m_{k_0}\}$
\begin{equation}\label{piuomeno}
\varphi_{k_0,m}(\lambda)=  O(\lambda^{\frac{k_0}{2}+\bar{\varepsilon}+1})\quad \text{as $\lambda\to 0^+$},
\end{equation}
if in addition we take into consideration that for every $m\in \{1,2,\dots,m_{k_0}\}$ 
\begin{equation*}
\int_0^\lambda t^{-N-\frac{k_0}{2}}\Upsilon_{k_0,m}(t)\,dt =O(\lambda^{\bar{\varepsilon}+1})\quad \text{as $\lambda\to 0^+$}
\end{equation*}
(for this we suggest the reader to consult the proof of \cite[Lemma 6.5]{DelFel}). Therefore, exploiting \eqref{coeffdiF} and \eqref{piuomeno}, we can deduce that for every $m\in \{1,2,\dots,m_{k_0}\}$ 
\begin{equation}\label{sfinita}
\sqrt{H(\lambda)}(U^\lambda, Y_{k_0,m})_{L^2(\mathbb{S}^N)}= O(\lambda^{\frac{k_0}{2}+\bar{\varepsilon}+1})\quad \text{as $\lambda\to 0^+$}.
\end{equation}
Now by \eqref{Hmaggioredisigma} we observe that $\sqrt{H(\lambda)}\geq \mathrm{const}\,\lambda^{\frac{k_0+\bar{\varepsilon}+1}{2}}$ for some $\mathrm{const}>0$ depending only on $\bar{\varepsilon}$ and for every $\lambda\in (0,r_0)$; so using this, from \eqref{sfinita} we obtain that 
\begin{equation}\label{sfinita1}
(U^\lambda, Y)_{L^2(\mathbb{S}^N)}= O(\lambda^{\frac{\bar{\varepsilon}+1}{2}})\quad \text{as $\lambda\to 0^+$}
\end{equation}
for every $Y\in \mathrm{span}\{Y_{k_0,m}:m\in \{1,2,\dots,m_{k_0}\}\}$. On the other hand, by Proposition \ref{propconvergenza} and the continuity of the trace map in \eqref{traceoperator}, for any sequence $\lambda_n\to 0^+$ there exist a subsequence (which we will still denote with $\lambda_n$) and an $L^2$-normalized function $\Psi\in \mathrm{span}\{Y_{k_0,m}:m\in \{1,2,\dots,m_{k_0}\}\}$ such that 
\begin{equation*}
U^{\lambda_n}\to \Psi \quad \text{in $L^2(\mathbb{S}^N)$ as $n\to\infty$}. 
\end{equation*}
We end up by observing that this leads to 
\begin{equation*}
(U^{\lambda_n}, \Psi)_{L^2(\mathbb{S}^N)} \to\Vert \Psi\Vert_{L^2(\mathbb{S}^N)}^2 =1 \quad \text{as $n\to\infty$},
\end{equation*}
which is in contradiction with \eqref{sfinita1}.
\end{proof}
We are now in the condition of proving the asymptotic behaviour around 0 of solutions to \eqref{problU} and accordingly for solutions to \eqref{problocaliz}.  
\begin{teo}\label{teocheassicura}
Let $k_0\geq 1$ be as in Lemma \ref{ellelegataautovalore}. If $m_{k_0}$ is the multiplicity of the eigenvalue $\mu_{k_0}$ and $\{Y_{k_0,m}\}_{m\in \{1,2,\dots,m_{k_0}\}}$ is a $L^2(\mathbb{S}^N)$-orthonormal basis of the eigenspace associated with $\mu_{k_0}$, then for every $m\in \{1,2,\dots,m_{k_0}\}$
\begin{equation}\label{asintotica1}
\frac{U(\lambda x)}{\lambda^{\frac{k_0}{2}}}\to |x|^{\frac{k_0}{2}}\sum_{m=1}^{m_{k_0}} \beta_m Y_{k_0,m}\left(\frac{x}{|x|}\right)\quad \text{in $H^1(B_1)$ as $\lambda\to 0^+$},
\end{equation}
where $(\beta_1,\beta_2,\dots,\beta_{m_{k_0}})\in \R^{m_{k_0}}\setminus\{0\}$ and for every $m\in \{1,2,\dots,m_{k_0}\}$
\begin{equation}\label{betamcomesonofatti}
\begin{split}
\beta_m=&\frac{\varphi_{k_0,m}(R)}{R^{\frac{k_0}{2}}}+\frac{2N+k_0-2}{2(N+k_0-1)}\int_0^R t^{-N-\frac{k_0}{2}}\Upsilon_{k_0,m}(t)\,dt\\
&+\frac{k_0 R^{-N+1-k_0}}{2(N+k_0-1)}\int_0^R t^{\frac{k_0}{2}-1}\Upsilon_{k_0,m}(t)\,dt\quad \text{for all $R\in (0,r_0]$}.
\end{split}
\end{equation}
\end{teo}
\begin{proof}
By Proposition \ref{propconvergenza} it holds that for every sequence $\lambda_n\to 0^+$ there exist a subsequence $\{\lambda_{n_k}\}$ and a non-null vector $(\beta_1,\beta_2,\dots,\beta_{m_{k_0}})\in \R^{m_{k_0}}$ such that 
\begin{equation}\label{precedente}
\frac{U(\lambda_{n_k} x)}{\lambda_{n_k}^{k_0/2}}\to |x|^{\frac{k_0}{2}}\sum_{m=1}^{m_{k_0}} \beta_m Y_{k_0,m}\left(\frac{x}{|x|}\right)\quad \text{in $H^1(B_1)$ as $k\to \infty$},
\end{equation}
if we also use Lemmas \ref{lemmadellimitecheesiste} and \ref{lemmalimpositivo}. We remark that if we show \eqref{betamcomesonofatti} then \eqref{asintotica1} automatically follows from Urysohn's subsequence principle. Hence we proceed by combining \eqref{coeffdiF}, \eqref{precedente} with the continuity of the trace map in \eqref{traceoperator}, obtaining that for every $m\in \{1,2,\dots,m_{k_0}\}$
\begin{equation}\label{assemb}
\lim_{k\to\infty} \frac{\varphi_{k_0,m}(\lambda_{n_k})}{\lambda_{n_k}^{k_0/2}}=\sum_{j=1}^{m_{k_0}}\beta_j\int_{\mathbb{S}^N}Y_{k_0,j}(\vartheta)Y_{k_0,m}(\vartheta)\, ds=\beta_m,
\end{equation}
being $\{Y_{k_0,m}\}_{m\in \{1,2,\dots,m_{k_0}\}}$ orthonormal. So assembling \eqref{assemb} and \eqref{espansione}, we arrive at \eqref{betamcomesonofatti}. This completes the proof.
\end{proof}
\begin{proof}[proof of Theorem \ref{mainresult}]
If $u\in H^1(B_{\bar{r}})$ is a non-trivial weak solution to \eqref{problocaliz}, letting $F$ as in Section \ref{subdiffeo}, $U=u\circ F\in H^1(B_{\tilde{r}})$ is a non-trivial weak solution to \eqref{problU}. 
Theorem \ref{teocheassicura} ensures that there exist $k_0\geq 1$ and an eigenfunction $\Psi$ of problem \eqref{problagliautovalori} associated with $\mu_{k_0}$ such that 
\begin{equation}\label{convergenza1}
\frac{U(\lambda x)}{\lambda^{\frac{k_0}{2}}}\to |x|^{\frac{k_0}{2}}\Psi\left(\frac{x}{|x|}\right)\quad \text{in $H^1(B_1)$ as $\lambda\to 0^+$}. 
\end{equation}
Setting $U_\lambda(\cdot):=U(\lambda\cdot)$, it holds that 
\begin{equation}\label{cambi}
u(\lambda x)= U_\lambda\left(\frac{F^{-1}(\lambda x)}{\lambda}\right)\quad \text{and }\quad \nabla ( u(\lambda x))=\nabla U_\lambda\left(\frac{F^{-1}(\lambda x)}{\lambda}\right) \mathrm{Jac}F^{-1}(\lambda x)
\end{equation}
for every $x\in B_1$ and $\lambda<\tilde{r}$. 
Moreover, from \eqref{serveperlultimoteo} and \eqref{stimesuF}, we can deduce that as $\lambda\to 0^+$
\begin{equation*}
\left|\frac{F^{-1}(\lambda x)}{\lambda}-x\right|\to 0 \quad \text{and }\quad \Vert \mathrm{Jac}F^{-1}(\lambda x)-\mathrm{Id}_{N+1}\Vert \to 0\quad \text{uniformly in $B_1$}.
\end{equation*}
So the thesis follows from this, \eqref{convergenza1} and \eqref{cambi}. 
\end{proof}
\section*{Acknowledgments}
The author would like to thank Professor Veronica Felli for encouraged her to do this project.


\begin{thebibliography}{10}
\bibitem{CafSil}
{\sc Caffarelli L. and Silvestre L.}
\newblock An extension problem related to the fractional {L}aplacian.
\newblock {\em Comm. Partial Differential Equations 32}, 7-9 (2007),
  1245--1260.
\bibitem{DelFel}
{\sc De~Luca A. and Felli V.}
\newblock Unique continuation from the edge of a crack.
\newblock {\em Math. Eng. 3}, 3 (2021), Paper No. 023, 40.
\bibitem{DelFelSic}
{\sc De~Luca A., Felli V., and Siclari G.}
\newblock Strong unique continuation from the boundary for
 the spectral fractional laplacian.
 \newblock {\em ESAIM. COCV, 29} (2023), Paper No. 50.
\bibitem{DelFelVit}
{\sc De~Luca A., Felli V. and Vita S.}
\newblock Strong unique continuation and local asymptotics at the boundary for
  fractional elliptic equations.
  \newblock {\em Adv. Math. 400\/} (2022), Paper No. 108279, 67.
  
\bibitem{Dip}
{\sc Dipierro S., Felli V., and Valdinoci E.}
\newblock Unique continuation principles in cones under nonzero Neumann boundary conditions.
 \newblock {\em Ann. I. H. Poincar\'{e} Anal. non lin\'{e}aire 37}, 4 (2020), 785--815. 
  
  
  
\bibitem{FabGarLin}
{\sc Fabes E.B., Garofalo N. and Lin F.H.}
\newblock A partial answer to a conjecture of B. Simon concerning
unique continuation. 
\newblock{\em J. Funct. Anal. 88} (1990), 194-210. 
\bibitem{FalFelFerNiang}
{\sc Fall, M. M., Felli, V., Ferrero, A. and Niang, A.}
\newblock Asymptotic expansions and unique continuation at {D}irichlet--{N}eumann boundary junctions for planar elliptic equations.
\newblock {\em Math. Eng. 1}, 1 (2019), 84--117.


\bibitem{FelFer}
{\sc Felli, V. and Ferrero, A.}
\newblock Almgren-type monotonicity methods for the classification of behaviour at corners of solutions to semilinear elliptic equations.
\newblock {\em P. Roy. Soc. Edinb. A 143}, 5 (2013), 957--1019.





\bibitem{Sus1}
{\sc Felli V., Ferrero A. and Terracini S.} 
\newblock Asymptotic behavior of solutions to Schr\"{o}dinger equations near an isolated singularity of the electromagnetic potential.
\newblock {\em J. Eur. Math. Soc. (JEMS) 13}, 1 (2011), 119--174.
 
\bibitem{Sus2}
{\sc Felli V., Ferrero A. and Terracini S.} 
\newblock A note on local asymptotics of solutions to singular elliptic equations via monotonicity methods.
\newblock {\em Milan J. Math. 80}, 1 (2012), 203--226. 

\bibitem{garofalo}
{\sc Garofalo, N. and Lin, F.H.}
\newblock Monotonicity properties of variational integrals, A p weights and unique continuation.
\newblock {\em Indiana University Mathematics Journal 35}, 2 (1986), 245--268.

\bibitem{gri}
{\sc Grisvard, P.}
\newblock Elliptic problems in nonsmooth domains.
\newblock{\em Monographs and Studies in Mathematics 24}, 2 (1985).
%  publisher={Springer}

\bibitem{hofmann2010singular}
{\sc Hofmann, S., Mitrea, M. and Taylor, M.}
\newblock Singular integrals and elliptic boundary problems on regular Semmes--Kenig--Toro domains.
\newblock{\em International Mathematics Research Notices 2010}, 14 (2010), 2567--2865.


\bibitem{ui}
{\sc Yu, H.}
\newblock Unique continuation for fractional orders of elliptic equations.
\newblock{\em Annals of PDE 3}, 2 (2017). 

\bibitem{LiZhu}
{\sc Li Y. and Zhu M.}
\newblock Sharp Sobolev inequalities involving boundary terms.
\newblock{\em Geometric and Functional Analysis 8}, 1 (1998), 59--122. 
\bibitem{Wang}
{\sc Wang Z.Q. and Zhu M.}
\newblock Hardy inequalities with boundary terms.
\newblock{\em Electronic Journal of Differential Equations}, (2003), 1--8. 

\end{thebibliography}
\end{document}